\documentclass{amsart}
\usepackage[latin1]{inputenc}
\usepackage[english]{babel}
\usepackage{csquotes}
\usepackage{amsthm}
\usepackage{tikz-cd}
\usepackage{verbatim}
\usepackage{xcolor}

\usepackage{enumerate}
\usepackage{color} 
\usepackage{calc}
\usepackage[colorlinks=true,linktoc=all,linkcolor=blue,citecolor=blue,filecolor=pink,urlcolor=blue]{hyperref}
\usepackage{cleveref}

\usepackage{amsmath, amssymb, latexsym, amscd, graphicx, epstopdf}
\usepackage{mathabx}
\usepackage{cancel}
\theoremstyle{plain}                    
\newtheorem{theorem}{Theorem}[section]
\newtheorem{lemma}[theorem]{Lemma}
\newtheorem{conj}[theorem]{Conjecture}

\newtheorem{corollary}[theorem]{Corollary}
\theoremstyle{definition}

\theoremstyle{plain}
\newtheorem{remark}[theorem]{Remark}

\numberwithin{equation}{section}

\graphicspath{ {figures/} }

\DeclareMathOperator{\Flip}{Flip}
\DeclareMathOperator{\Rot}{Rot}
\DeclareMathOperator{\Glue}{Glue}

\DeclareMathOperator{\Alt}{Alt}
\DeclareMathOperator{\Sym}{Sym}
\DeclareMathOperator{\Cay}{Cay}

\DeclareMathOperator{\Hom}{Hom}

\DeclareMathOperator{\im}{im}

\DeclareMathOperator{\dev}{dev}
\DeclareMathOperator{\hol}{hol}
\DeclareMathOperator{\Hol}{Hol}

\DeclareMathOperator{\GL}{GL}
\DeclareMathOperator{\PGL}{PGL}

\DeclareMathOperator{\PSO}{PSO}
\DeclareMathOperator{\PO}{PO}

\DeclareMathOperator{\CW}{CW}

\newcommand{\tet}{\mathrm{T}}
\newcommand{\statet}{\tet_\circ}

\newcommand{\rp}{\mathbb{RP}}

\newcommand{\cp}{\mathbb{CP}}
\newcommand{\A}{\mathcal{A}}
\newcommand{\E}{\mathcal{E}}
\newcommand{\F}{\mathcal{F}}
\newcommand{\G}{\mathcal{G}}
\renewcommand{\H}{\mathcal{H}}
\newcommand{\J}{\mathcal{J}}
\newcommand{\T}{\mathcal{T}}
\newcommand{\X}{\mathcal{X}}
\newcommand{\Y}{\mathcal{Y}}
\newcommand{\CC}{\mathbb{C}}
\newcommand{\HH}{\mathbb{H}}

\newcommand{\NN}{\mathbb{N}}
\newcommand{\PP}{\mathbb{P}}
\newcommand{\RR}{\mathbb{R}}
\newcommand{\XX}{\mathbb{X}}
\newcommand{\ZZ}{\mathbb{Z}}
\newcommand{\edgeface}{(\mathrm{E})\mathrm{F}}
\newcommand{\FL}{\mathrm{FL}}
\newcommand{\cycFL}{\FL_*}
\newcommand{\ccycFL}{\mathcal{FL}_*}
\newcommand{\ordFL}{\FL}
\newcommand{\cordFL}{\mathcal{FL}}
\newcommand{\triFL}{\FL_{\triangle}}
\newcommand{\ctriFL}{\cordFL_{\triangle}}
\newcommand{\tetFL}{\FL_{\tet}}

\newcommand{\ctetFL}{\cordFL_{\tet}}

\newcommand{\TriFL}{\FL_{\Delta}}
\newcommand{\cTriFL}{\cordFL_{\Delta}}

\newcommand{\cross}[1]{\left[#1\right]}

\newcommand{\3}{\cancel{3}}
\newcommand{\abs}[1]{\left|#1\right|}
\newcommand{\op}[1]{#1_{op}}

\newcommand{\Tet}{\mathsf{Tet}}
\newcommand{\Edge}{\mathsf{Ed}}
\newcommand{\Face}{\mathsf{Fa}}
\newcommand{\Vertex}{\mathsf{Ver}}

\newcommand{\Prism}{\mathsf{P}}
\newcommand{\tetsubvariety}{\mathcal{D}_\tet^+}

\newcommand{\Ext}{\mathrm{Ext}_{\Delta}}

\newcommand{\defspace}{\mathcal{D}}

\newcommand{\complex}{\mathfrak{C}}
\newcommand{\graph}{\mathfrak{G}}

\newcommand{\wt}{\widetilde}

\newcommand{\basis}[1]{\mathbf{e_{#1}}}
\newcommand{\dbasis}[1]{\mathbf{e_{#1}^*}}
\newcommand{\Co}{\mathsf{Co}}
\newcommand{\Psist}{\Psi_{\sigma,\tau}}

\newcommand{\wDelta}{\widetilde{\Delta}}

\newcommand{\coch}{\mathrm{C}^1}

\newcommand{\cocy}{\mathrm{Z}^1}

\newcommand{\ADS}{\mathbb{A}\mathrm{d}\mathbb{S}^3}

\newcommand\restr[2]{{
  \left.\kern-\nulldelimiterspace 
  #1 
  \vphantom{\big|} 
  \right|_{#2} 
  }}

\renewcommand{\Re}{\mathsf{Re}}
\renewcommand{\Im}{\mathsf{Im}}

\def\Tt{\mathcal{CP}}

\def\Char{\mathfrak{X}}
\def\End{\E}

\def\n@te#1{\textsf{\boldmath \textbf{$\langle\!\langle$#1$\rangle\!\rangle$}}\leavevmode}
\def\Note#1{\textcolor{red}{\n@te {#1}}}

\newcommand{\cyclic}[1]{\mathrel{ ( \mkern-3.85mu ( \mkern-6mu} {#1} \mathrel{\mkern-6mu ) \mkern-3.85mu)}}

\newcounter{notes}%

\begin{document}
\title[Gluing equations ]{Gluing equations for real projective structures on $3$--manifolds}
\author{Samuel A. Ballas}
\address{Department of Mathematics,
 Florida State University,
  1017 Academic Way, Tallahassee, FL 32306, USA}
\email{ballas@math.fsu.edu}

\author{Alex Casella}
\address{Department of Mathematics, Florida State University, 1017 Academic Way, Tallahassee, FL 32306, USA}
\email{casella@math.fsu.edu}
\date{\today}

\begin{abstract}
Given an orientable ideally triangulated $3$--manifold $M$, we define a system of real valued equations and inequalities whose solutions can be used to construct projective structures on $M$. These equations represent a unifying framework for the classical Thurston gluing equations in hyperbolic geometry and their more recent counterparts in Anti-de Sitter and half-pipe geometry. Moreover, these equations can be used to detect properly convex structures on $M$. The paper also includes a few explicit examples where the equations are used to construct properly convex structures.
\end{abstract}

\maketitle
\tableofcontents


\section*{Preface}

\emph{Real projective geometry} is a special instance of a $(G,X)$--geometry, in the sense of Klein~\cite{KL}, whose model space is the real projective space $\rp^n$, acted upon by the group of projective transformations $\PGL({n+1},\RR)$. This geometry represents a unifying framework in which many interesting geometries (Riemmannian,  Lorentzian, etc.) can be viewed simultaneously. This perspective often allows one to see seemingly unrelated connections between disparate types of geometric objects.

The space of equivalence classes of real projective structures on a manifold $M$ is denoted by $\rp(M)$. If $M$ has an ideal triangulation $\Delta$, one can use this decomposition to concretely construct elements in $\rp(M)$. This idea goes back to Thurston~\cite{THUNOTES}, in the context of $3$--dimensional hyperbolic structures, and has been developed and generalized by many others since then (for instance \cite{BFGFlags,DAN,FockGonch,ST11} and others. In his notes, Thurston constructs hyperbolic structures by realizing each tetrahedron of $\Delta$ as a hyperbolic ideal tetrahedron, and then gluing them together with hyperbolic isometries. This construction gives rise to a hyperbolic structure on the complement of the $1$--skeleton in $\Delta$. When the structure extends to the $1$--skeleton so that the angle is an integer multiple of $2\pi$, we say that the structure is \emph{branched} with respect to $\Delta$. The space of equivalence classes of hyperbolic structures on $M$, that are branched with respect to $\Delta$, is denoted by $\HH(M;\Delta)$. The brilliance in Thurston's idea relies on few facts. First, the isometry class of an oriented hyperbolic tetrahedron is uniquely determined by a single complex number, with positive imaginary part, called the \emph{Thurston's parameter}. Second, two hyperbolic tetrahedra can always be glued together along a face with a unique orientation reversing hyperbolic isometry. Finally, the branching condition can be encoded in one complex valued equation per edge, called the \emph{Thurston's gluing equation}. The set of Thurston's parameters satisfying Thurston's gluing equations is a semi-algebraic affine subset of $\CC^k$, where $k$ is the number of tetrahedra in $\Delta$. We call it \emph{Thurston's deformation space} of $M$ with respect to $\Delta$ and denote it by $\mathcal{D}_{\HH}(M;\Delta)$. Then Thurston's construction translates into a map $\mathrm{Ext}_{\HH}:\mathcal{D}_{\HH}(M;\Delta)\to \HH(M;\Delta)$,  that assigns to a solution the corresponding class of hyperbolic structures.

In this paper, we generalize Thurston's technique to the space of equivalence classes of real projective structures on a non-compact orientable $3$--manifold $M$, that are branched with respect to an ideal triangulation $\Delta$. Such space is denoted by $\rp(M;\Delta)$. Although angles are not well defined in $\rp^3$, the notion of branching can be defined by insisting that the holonomy of a structure around edges of $\Delta$ is trivial (cf.~\S\ref{subsec:proj_structures}). The main idea is to introduce the space of equivalence classes of \emph{triangulation of flags} $\ctriFL$ (cf. \S\ref{subsec:tri_of_flags}). Roughly speaking, a triangulation of flags is a decoration of each ideal vertex in $\Delta$ with an \emph{(incomplete) flag} of $\rp^3$, namely a point and a plane containing such point, which is equivariant with respect to some representation $\pi_1(M) \rightarrow \PGL(4)$. We show that every triangulation of flags can be extended to a branched real projective structure, thus giving rise to an analogous map $\mathrm{Ext}_\Delta: \ctriFL \to \rp(M;\Delta)$ (cf. \S\ref{subsec:developing_tri_of_flags}). One of the main results of this paper is that $\ctriFL$ admits an explicit parametrization by a semi-algebraic affine set $\mathcal{D}_\Delta = \mathcal{D}_\rp(M;\Delta)$, called the \emph{deformation space of a triangulation of flags} (cf. Theorem~\ref{thm:parametrization_def_space}). The definition of $\mathcal{D}_\rp(M;\Delta)$ and its identification with $\ctriFL$ involve innovative ideas developed by combining the work of Thurston~\cite{THUNOTES}, Fock and Goncharov~\cite{FockGonch}, Bergeron--Falbel--Guilloux \cite{BFGFlags}, and Garoufalidis--Goerner--Zickert \cite{GGZGluing}. Since they are of a technical nature, we postpone describing them until the next section.

The identification $\mathcal{D}_\rp(M;\Delta) \cong \ctriFL$ leads to the following result.

\begin{theorem}\label{thm:main_thm}
Let $M$ be a non-compact orientable $3$--manifold equipped with an ideal triangulation $\Delta$. There are embeddings $\varphi : \mathcal{D}_{\HH}(M;\Delta)\hookrightarrow \mathcal{D}_{\rp}(M;\Delta)$ and $\mathrm{f}: \HH(M;\Delta)\hookrightarrow \rp(M;\Delta)$ that make the following diagram commute:
$$
\begin{tikzcd}
\mathcal{D}_{\HH}(M;\Delta)\ar[d,hookrightarrow,"\varphi"] \ar[r,"\mathrm{Ext}_{\HH}"] & \HH(M;\Delta)\ar[d,hookrightarrow, "\mathrm{f}"]\\
\mathcal{D}_\rp(M;\Delta)\ar[r, "\Ext"] & \rp(M;\Delta)
\end{tikzcd}
$$
Furthermore, the image of $\varphi$ can be explicitly described.
\end{theorem}
Roughly speaking, Theorem~\ref{thm:main_thm} says that $\mathcal{D}_\rp(M;\Delta)$ gives rise to (branched) projective structures on $M$ in a manner that is compatible with the way Thurston's deformation variety determines (branched) hyperbolic structures on $M$. Theorem~\ref{thm:main_thm} turns out to be a corollary of two more general results, Theorem~\ref{thm:phi_image} and Theorem~\ref{thm:X_comm_diag} (cf. \S\ref{subsec:proof_main_thm}). By work of Danciger~\cite{DAN2}, there are generalizations of Thurston's deformation space to both \emph{Anti-de Sitter} and \emph{half-pipe} geometries, and they can each be embedded in $\mathcal{D}_\rp(M;\Delta)$. In \S\ref{sec:thu_equations} we show that $\mathcal{D}_\rp(M;\Delta)$ represents a unifying framework for these three geometries and show that a version of Theorem~\ref{thm:main_thm} is true for each of these geometries.

If a projective (resp. hyperbolic) structure is contained in the image of $\Ext$ (resp.\ $\mathrm{Ext}_{\HH})$ then we say that this structure is \emph{realized} by $\mathcal{D}_\rp(M;\Delta)$ (resp.\ $\mathcal{D}_\HH(M;\Delta)$).

Aside from hyperbolic structures, there are other interesting projective structures that the space $\mathcal{D}_\rp(M;\Delta)$ realizes. A \emph{properly convex projective structure} on $M$ is a projective structure for which the developing map is a diffeomorphism onto a properly convex set (cf. \S\ref{sec:prop_convex_projective_st}). Let $\Tt(M)\subset \rp(M)$ be the set of equivalence classes of properly convex projective structures on $M$. A \emph{generalized cusp} is a properly convex manifold that generalizes the notion of a cusp in a finite volume hyperbolic manifold (cf. \S\ref{subsec:gen_cusp}). These types of cusps were originally defined by Cooper-Long-Tillmann in~\cite{CLTII}, and have recently been classified in~\cite{BCL}. Let $\Tt^c(M)\subset \Tt(M)$ be the set of equivalence classes of properly convex structures on $M$ where each topological end of $M$ has the structure of a generalized cusp. For instance, every complete finite volume hyperbolic structure is a properly convex projective structure with generalized cusp ends. It is conjectured that $\Tt^c(M) = \Tt(M)$ for a large class of manifolds (cf. Conjecture~\ref{conj:gen_cusps}). We show that the deformation space $\mathcal{D}_\rp(M;\Delta)$ also realizes many projective structures in $\Tt^c(M)$.

\begin{theorem}\label{thm:main_thm2}
Let $M$ be a non-compact orientable $3$--manifold that is the interior of a compact manifold whose boundary is a finite union of disjoint tori and suppose that $M$ is equipped with an ideal triangulation $\Delta$. If $U:=\Ext^{-1}(\Tt^c(M))$ then the restriction
$$
\restr{\Ext}{U}:U\to \Tt^c(M)
$$
is open and finite-to-one.
\end{theorem}

Theorem \ref{thm:main_thm2} should be viewed in the context of results of Fock-Goncharov \cite{FG2} concerning character varieties of non-compact surfaces. Specifically,  they show that if $\Sigma$ is a non-compact surface of finite type then the space of equivalence classes of ``framed'' representations is a branched cover of the standard character variety.

The proof of Theorem~\ref{thm:main_thm2} is the main result of \S\ref{sec:prop_convex_projective_st}. It has the following immediate corollary.

\begin{corollary}\label{cor:detects_nearby_hyperbolic}
Let $M$ be a finite volume hyperbolic $3$--manifold, and suppose the complete hyperbolic structure is realized by $\mathcal{D}_\HH(M;\Delta)$. Then every properly convex projective structure with generalized cusp ends that is sufficiently close to the complete hyperbolic structure is realized by $\mathcal{D}_\rp(M;\Delta)$.
\end{corollary}

We remark that it is still unknown whether there always exists a triangulation $\Delta$ such that the complete hyperbolic structure is realized by $\mathcal{D}_\HH(M;\Delta)$, but there is strong evidence that this is typically true.

In a complementary direction to Theorem~\ref{thm:main_thm2}, we are also able to show that for a few small examples it is possible to find explicit families in $\mathcal{D}_\HH(M;\Delta)$ that correspond to properly convex projective structures on $M$ (cf. \S\ref{sec:examples}). In particular, we construct previously unknown properly convex deformations of the complete hyperbolic structure for the figure-eight sister manifold.

\begin{theorem}\label{thm:sis_thm}
Let $M$ be either the figure-eight knot complement or the figure-eight sister manifold, then there is a $1$--parameter family of finite volume (non-hyperbolic) properly convex projective structures on $M$ corresponding to a curve in $\mathcal{D}_{\rp}(M;\Delta)$. 
\end{theorem}

In the above theorem, the volume is the Hausdorff measure coming from the Hilbert metric (see \cite{BAL15} for a description of this measure in the context of properly convex geometry).

We conclude the paper by computing a modified version of our gluing equations of a certain hyperbolic orbifold $O$ obtained from the Hopf link that has an ideal triangulation with a single tetrahedron. The case for orbifolds turns out to be more subtle than the case of manifolds. Specifically, in this context, we are able to show that Theorem \ref{thm:main_thm2} does not generalize in a straightforward way.


\subsection*{Summary of the main technique}\label{subsec:main_technique}

The first half of the paper (\S\ref{sec:coords_teton}--\S\ref{sec:coords_triangulation}) is devoted to defining the deformation space $\mathcal{D}_\Delta = \mathcal{D}_\rp(M;\Delta)$, and proving that it parametrizes the space of projective classes of triangulations of flags $\ctriFL$. The construction is technical, but insightful, and the ideas developed along the way are of interest on their own. Here we provide the reader with a broad overview in the hope that it will make the more technical heart of the paper easier to follow.

Our main object of study is a \emph{tetrahedron of flags}, namely four (incomplete) flags in $\rp^3$ that satisfy certain genericity conditions (cf. \S\ref{subsec:tris_and_tets_flags}). One should think of a tetrahedron of flags as encoding the vertices of a \emph{projective tetrahedron} in $\rp^3$, together with a plane through each vertex. Here a (projective) tetrahedron is a region of $\rp^3$ that is projectively equivalent to the projectivization of the positive orthant in $\RR^4$. Although in general there are several projective tetrahedra that have the same four points as their vertices, a tetrahedron of flags always singles out a \emph{unique} projective tetrahedron, namely the one whose interior is disjoint from the union of the planes coming from the flags.

Unlike projective tetrahedra, not all tetrahedra of flags are projectively equivalent, thus giving rise to a non-trivial moduli space. The space of projective classes of tetrahedra of flags $\ctetFL$ is determined by twelve \emph{edge ratios} and twelve \emph{triple ratios} (cf. \S\ref{subsec:triple_edge_ratios}), satisfying some \emph{internal consistency equations} (cf. Lemma~\ref{lem:internal_eqs}). We devote \S\ref{sec:coords_teton} to show that these ratios cut out a semi-algebraic affine set $\mathcal{D}_\tet^+ \subset \RR^{24}$ which is homeomorphic to both $\RR^5_{>0}$ and $\ctetFL$ (cf. Theorem~\ref{thm:parametrization_flag_space}). Conceptually, this is a generalization of the fact that the space of orientation preserving isometry classes of hyperbolic ideal tetrahedra is homeomorphic to the subset of complex numbers with positive imaginary part. Edge ratios and triple ratios are inspired by the work of Fock-Goncharov~\cite{FockGonch} and Bergeron-Falbel-Guilloux~\cite{BFGFlags}, although their setting is different as it involves \emph{complete flags} instead.

Next in \S\ref{sec:coords_pair_teta}, we turn our attention to \emph{gluing} two tetrahedra of flags along a face. Roughly speaking, this involves applying a projective transformation that maps three flags of one tetrahedron of flags to three flags of the other one. In general, ordered triples of incomplete flags are not projectively equivalent. This phenomenon has no hyperbolic analog, as any simplicial map between two hyperbolic ideal triangles is always realizable by a unique hyperbolic isometry. On the other hand, two tetrahedra of flags are \emph{glueable} if and only if their parameters satisfy some \emph{face pairing equations} (cf. Lemma~\ref{lem:face_pair_eq}). Additionally, when two tetrahedra of flags can be glued along a face, there is a $1$--parameter family of inequivalent ways to glue them. This is analogous to the fact there is a $1$--parameter family of ways to glue two hyperbolic ideal triangles along an ideal edge. For symmetry reasons, it is more convenient to think about this $1$--parameter family as parametrized by six \emph{gluing parameters} (cf.~\S\ref{subsec:gluing_param}) satisfying some \emph{gluing consistency equations} (cf. Lemma~\ref{lem:gluing_eqs}). In hyperbolic geometry, the requirement that the face pairings are orientation reversing ensures that the tetrahedra are glued \emph{geometrically}, namely not ``inside out''. The same effect is achieved with tetrahedra of flags by imposing that all gluing parameters are positive (cf. \S\ref{subsec:geometric_gluing}). The rest of \S\ref{sec:coords_pair_teta} is then devoted to show that the \emph{deformation space of pairs of glued tetrahedra of flags} $\mathcal{D}_{\sigma,\tau}$ is homeomorphic to $\RR_{>0}^{10}$ (cf. Theorem~\ref{thm:parametrization_gluing_space}).

Next, we iterate the above gluing construction and attempt to glue together the cycle of tetrahedra of flags that about a common edge, to obtain a \emph{triangulation of flags} (cf.~\ref{subsec:tri_of_flags}). In hyperbolic geometry, a cycle of hyperbolic ideal tetrahedra closes up around a common edge to form an angle which is an integer multiple of $2\pi$ if and only if the Thurston's parameters satisfy a single complex valued equation. Similarly, a cycle of tetrahedra of flags glues around a common edge so that the underlying projective tetrahedra form a branched structure if and only if their parameters satisfy certain additional \emph{edge gluing equations} (cf.~\S\ref{subsec:gluing_eq}). These equations are not as simple to define as Thurston's gluing equations, and they require the introduction of a tool we call the \emph{monodromy complex} (cf.~\S\ref{subsec:monodromy_complex}).

The monodromy complex $\complex_\Delta$ is a $2$--dimensional CW--complex embedded in $M$ \emph{dual to} $\Delta$, with the same fundamental group as $M$ (cf. Lemma~\ref{lem:complex_fund_group}). Given a choice of edge ratios, triple ratios and gluing parameters for each tetrahedron in $\Delta$ we can label the edges of $\complex_\Delta$ with elements of $\PGL(4)$ to define a $\PGL(4)$--cochain (see Eqs \eqref{eq:rot_mat}-\eqref{eq:glue_mat}). Each of these labels has geometric meaning. Roughly speaking, the edge and triple ratios determine the projective classes of each tetrahedron of flags and the labels on the edges encode how these tetrahedron of flags are glued together. In general, this construction does not give rise to a triangulation of flags since as one uses the monodromy complex to assemble the cycle of tetrahedra of flags around an edge of $\Delta$ they do ``close up.'' In Theorem~\ref{thm:deformation_space_to_cocycles}, we show that these tetrahedra of flags glue to form a triangulation of flags if and only if the corresponding cochain is a $\PGL(4)$--\emph{cocycle}, namely the product of all matrices along the boundary of $2$--cells are trivial (cf.~\S\ref{subsec:cocycles}). The equations the ensure the triviality of this product are the previously mentioned edge gluing equations.

Our use of cochains and cocycles was inspired by the work of Garoufalidis--Goerner--Zickert \cite{GGZGluing}, where they use these concepts to study representations of $3$--manifold groups into $\PGL(n,\CC)$. The set of parameters satisfying the edge gluing equations is called the \emph{deformation space of a triangulation of flags}, $\mathcal{D}_\Delta = \mathcal{D}_\rp(\Delta;M)$, and it is homeomorphic to the space of equivalence classes of triangulations of flags $\ctriFL$ (cf. Theorem~\ref{thm:parametrization_def_space}).


\subsection*{Motivation}

Our original motivation for constructing $\mathcal{D}_\rp(\Delta;M)$ was to study properly convex structures on $M$, particularly in the case where $M$ is a finite volume hyperbolic $3$--manifold. From this perspective, the use of incomplete flags is quite natural. Assuming all ends of $M$ are \emph{generalized cusps} (cf.~\S\ref{subsec:gen_cusp} and Conjecture~\ref{conj:gen_cusps}), the holonomy of each peripheral subgroup of $\pi_1(M)$ preserves at least one and at most finitely
many incomplete flags in $\rp^3$ (cf. Lemma~\ref{lem:framing_end}). At the expense of modifying the developing map in a way that does not change the underlying projective structure, one of these flags can always be chosen to be a \emph{supporting flag}. Assigning this flag to the vertices of the tetrahedra in $\wDelta$ that correspond to that peripheral subgroup is how we determine a point in $\mathcal{D}_\rp(\Delta;M)$ in the proof Theorem~\ref{thm:main_thm2}.

With this goal in mind, we mainly drew inspiration from Thurston's work~\cite{THUNOTES}, but also from several other generalizations such as \cite{BFGFlags},\cite{FockGonch},\cite{GGZGluing},\cite{THUNOTES} and others. While many of the ideas and techniques in these works are similar to ours, our work differs in several important ways. One of the main differences is the strong geometric flavour of our construction. For instance, the generalized gluing equations defined in~\cite{GGZGluing} are naturally complex valued, and their set of solutions ends up parametrizing representations from $\pi_1(M)$ into $\PGL(n,\CC)$. In general, these representations are not the holonomy representations of a geometric structure on $M$, at least not in an obvious way. Furthermore, the main tools in these other constructions are \emph{complete} flags. The use of complete flags has the computational advantage of simpler looking equations, usually at the expenses of larger systems of equations. However, from the perspective of properly convex structures, this approach is less natural. As described above there are natural geometric choices for incomplete flags, there is generally no geometrically meaningful way to complete these flags with projective lines. Moreover, for certain projective structures (as complete hyperbolic structures), there are \emph{infinitely many} ways to decorate the vertices of an ideal triangulation with complete flags (see Remark~\ref{rem:cplt_flags_bad}). Such behavior is undesirable if one's goal is to parametrize projective structures.


\subsection*{Organization of paper} Section~\ref{sec:background} describes necessary background for subsequent results. Sections~\ref{sec:coords_teton}-\ref{sec:coords_triangulation} form the technical heart of the paper. In particular \S\ref{sec:coords_teton} describes the moduli space of a single tetrahedron of flags, \S\ref{sec:coords_pair_teta} describes the moduli space for gluing two tetrahedra of flags, and  \S\ref{sec:coords_triangulation} describes the gluing equations for an entire triangulation of flags of a $3$--manifolds. The remaining sections are applications of the aforementioned parametrizations to prove the main theorems and construct examples. Section~\ref{sec:thu_equations} analyzes the relationships between our machinery and the more classical settings of hyperbolic, Anti-de Sitter, and half-pipe structures. It also contains the proof of Theorem~\ref{thm:main_thm}. Section~\ref{sec:prop_convex_projective_st} describes how convex projective structures can be framed and contains the proof of Theorem~\ref{thm:main_thm2}. Section~\ref{sec:examples} contains explicit examples where the gluing equations are solved and used to produce interesting projective structures. In particular, Section~\ref{sec:sister} produces a previously unknown family of convex projective structures on the figure-eight sister manifold (cf. Theorem~\ref{thm:sis_thm}). 


\subsection*{Acknowledgements} The authors would like to thank J.\ Porti and S.
 Tillmann for several useful conversations and for making us aware of the new convex projective structures they found for the Hopf link orbifold. The authors would also like to thank J.\ Danciger for several useful discussions regarding Anti-de Sitter and half-pipe geometry. The first author was partially supported by NSF grant DMS 1709097.
 

\section{Background}\label{sec:background}

	\subsection{Projective space}\label{subsec:projective_space}
	
	Let $W$ be a finite dimensional real vector space. There is an equivalence relation on $W\setminus\{0\}$ given by $v\sim w$ if and only if $v = \lambda w$, for some $\lambda \in \RR^\times := \RR \setminus \{ 0\}$. The quotient of $W\setminus\{0\}$ by this equivalence relation is the \emph{projective space of $W$}, which we denote $\PP(W)$. More conceptually, $\PP(W)$ is the space of $1$--dimensional subspaces of $W$. A \emph{$k$--dimensional plane} of $\PP(W)$ is the projectivization of a $(k+1)$--dimensional subspace of $W$. In particular, if $n = \dim(W)$, a \emph{hyperplane} (resp. \emph{line}) of $\PP(W)$ is an
	$(n-1)$--dimensional (resp. $1$--dimensional) plane of $\PP(W)$.
	
	There is a natural quotient map $W\setminus\{0\} \rightarrow \PP(W)$ that maps a vector $v \in W\setminus\{0\}$ to its equivalence class $[v] \in \PP(W)$. We will usually drop the brackets for elements of $\PP(W)$, unless we need to distinguish them from their representatives in $W\setminus\{0\}$.
	
	If $W^*$ is the \emph{dual vector space} to $W$, then $\PP(W^*)$ is the \emph{dual projective space}. We can regard $\PP(W^*)$ as the set of hyperplanes of $\PP(W)$ by identifying $[f] \in W^*$ with the hyperplane $[\ker(f)] \subset \PP(W)$. Henceforth we will make use of this identification implicitly.
	
	Let $\{V_i\}_{i=1}^k$ be a set of points of $\PP(W)$, then we denote by $V_1V_2\ldots V_{k-1}V_k$ the plane of $\PP(W)$ spanned by this set. Similarly, if $\{\eta_i\}_{i=1}^k$ is a set of planes of $\PP(W)$, then $\eta_1\eta_2 \ldots \eta_{k-1}\eta_k$ is the plane of $\PP(W)$ obtained by intersecting them.
	
	Elements of the general linear group $\GL(W)$ take $1$--dimensional subspaces to $1$--dimensional subspaces, so the natural (left) action of $\GL(W)$ on $W$ descends to an action on $\PP(W)$. This action is not faithful: the kernel consists of non-zero scalar multiples of the identity $I$, and so the \emph{projective general linear group} $\PGL(W):=\GL(W)/\RR^\times I$ acts faithfully on $\PP(W)$. The group $\PGL(W)$ also admits a (left) action on $\PP(W^*)$ given by $[A]\cdot [f]=[f \circ A^{-1}]$, for all $[A]\in \PGL(W)$ and $[f] \in \PP(W^*)$.
	
	 Let $n := \dim (W)$. Then a collection of distinct points (resp. planes) $\mathcal{C}$ in $\PP(W)$ is in \emph{general position} if no $k$ of them lie in a $(k-2)$--dimensional plane (resp. intersects in an $(n-k+1)$--dimensional plane) of $\PP(W)$, for all $1\leq k\leq n$. Notice that if the collection $\mathcal{C}$ contains at least $n+1$ points then they are in general position if and only if no $n$ of them lie in a hyperplane. A collection of exactly $n+1$ points in general position is called a \emph{projective basis} for $\PP(W)$. It is an elementary fact about projective geometry that the group $\PGL(W)$ acts simply transitively on the set of ordered projective bases of $\PP(W)$. Consequently, the action of an element of $\PGL(W)$ on $\PP(W)$ is completely determined by its action on a projective basis.
	 
	 In this paper we are mainly interested in the case where $W=\RR^{n+1}$, for which we adopt the notation $\rp^n : = \PP(\RR^{n+1})$ and $\PGL(n) := \PGL(\RR^{n+1})$.

	 
\subsection{Cross ratios}\label{subsec:cross_ratio}

We fix an identification $\rp^1 = \RR\cup\{\infty\}$. Three distinct points of $\rp^1$ form a projective basis, hence given an ordered quadruple of four distinct points $(x_1, x_2, x_3, x_4)$ of $\rp^1$ there is a unique projective transformation $G\in \PGL(2)$ that takes $(x_1, x_2, x_3)$ to $(\infty,0,1)$. The \emph{cross ratio of $(x_1, x_2, x_3, x_4)$} is the quantity $\cross{x_1,x_2,x_3,x_4}:=G(x_4) \in \rp^1$. It is easy to check that in coordinates:
$$
\cross{x_1,x_2,x_3,x_4} = \frac{(x_1 - x_3)(x_2 - x_4)}{(x_1 - x_4)(x_2 - x_3)}.
$$
By definition, the cross ratio is projectively invariant. Furthermore, it is also invariant under certain symmetries. Specifically, let $\Sym(n)$ be the \emph{group of permutations on $n$ symbols}, and consider the subgroup $H \leqslant \Sym(4)$ of $2$--$2$ cycles. The group $\Sym(4)$ acts on ordered quadruples of points in $\rp^1$ by permuting them and a simple computation shows that $H$ is the kernel of this action. 

This definition extends to ordered quadruples of collinear points $(x_1, x_2, x_3, x_4)$ in $\rp^n$. If $\ell$ is the line spanned by $(x_1, x_2, x_3, x_4)$, then $\ell$ is projectively equivalent to $\rp^1$, and one defines the \emph{cross ratio $\cross{x_1,x_2,x_3,x_4}$} through this equivalence. This definition does not depend on the chosen identification $\ell = \rp^1$ as the cross ratio is projectively invariant.

Similarly, one defines the cross ratio of an ordered quadruple of distinct hyperplanes $(\eta_1,\eta_2,\eta_3,\eta_4)$ of $\rp^n$, intersecting at a common $(n-2)$--dimensional plane $\ell$. The \emph{pencil of hyperplanes through $\ell$} is the set $\ell^*$ of hyperplanes of $\rp^n$ containing $\ell$. Then $\ell^*$ is projectively equivalent to $\rp^1$, and one defines the \emph{cross ratio $\cross{\eta_1,\eta_2,\eta_3,\eta_4}$} through this equivalence. Once again, this definition does not depend on the chosen identification $\ell^* = \rp^1$.

The cross ratio has a straightforward positivity property that we record in the following result.

\begin{lemma}\label{lem:cross_ratio_positivity}
    Let $(x_1, x_2, x_3, x_4)$ (resp. $(\eta_1,\eta_2,\eta_3,\eta_4)$) be an ordered quadruple of distinct points on a line $\ell$ (resp. hyperplanes in a pencil $\ell^*$) of $\rp^n$. Then $\cross{x_1,x_2,x_3,x_4} > 0$ (resp. $\cross{\eta_1,\eta_2,\eta_3,\eta_4} > 0$) if and only if $x_3$ and $x_4$ (resp. $\eta_3$ and $\eta_4$) belong to the same connected component of $\ell \setminus \{x_1,x_2\}$ (resp. $\ell^* \setminus \{\eta_1,\eta_2\}$).
\end{lemma}


\subsection{Incomplete flags}\label{subsec:flags}
    
    Let $W$ be a finite dimensional real vector space. An \emph{incomplete (projective) flag} is a pair $(V,\eta) \in \PP(W)\times \PP(W^*)$, such that $\overline{\eta}(\overline{V})=0$ for some (and hence for all) representatives $\overline{V} \in W$ and $\overline{\eta} \in W^*$ of $V$ and $\eta$ respectively. Geometrically, an incomplete flag is a point in $\PP(W)$ and a hyperplane in $\PP(W)$ containing that point. The \emph{space of incomplete flags} is denoted by $\FL(W)$.
    
    We remark that $\eta(V)=\eta(G^{-1}GV)=(G\cdot \eta)(G\cdot V)$,  for all $G\in \PGL(W)$ and $(V,\eta) \in \FL(W)$. Thus the natural diagonal action of $\PGL(W)$ on $\PP(W)\times \PP(W^*)$ gives an action of $\PGL(W)$ on $\FL(W)$.
    
    A collection of incomplete flags $\{(V_i,\eta_i)\}_{i=1}^k$ is \emph{non-degenerate} when
	$$
	\eta_i(V_j) = 0 \iff i = j.
	$$
	We denote by $\FL^k(W)$ the collection of \emph{non-degenerate ordered} $k$--tuples of flags, for which we adopt the notation
	$$
	(V_i,\eta_i)_{i=1}^{k} := \left((V_1,\eta_1),(V_2,\eta_2),\dots,(V_k,\eta_k)\right).
	$$
	The set $\FL^k(W)$ is $\PGL(W)$ invariant, therefore there is a well defined quotient $\cordFL^k(W) := \FL^k(W) / \PGL(W)$.
	
	Furthermore, we denote by $\cycFL^k(W)$ be the collection of \emph{non-degenerate cyclically ordered} $k$--tuples of flags, for which we adopt the notation
	$$
	\cyclic{V_i,\eta_i}_{i=1}^{k} := \cyclic{(V_1,\eta_1),(V_2,\eta_2),\dots,(V_k,\eta_k)}.
	$$
	Formally, this is the quotient of $\FL^k(W)$ by the subgroup of $\Sym(k)$ generated by the $k$--cycle $(12\ldots k)$. The set $\cycFL^k(W)$ is also $\PGL(W)$ invariant, thus we let $\ccycFL^k(W)$ be the corresponding quotient.

    Since we will only be working with incomplete flags in this paper, we will henceforth refer to them simply as \emph{flags}. Moreover, from now on we will focus on the case $W=\RR^{4}$, for which we adopt the shorter notations:
    $$
    \FL^k := \FL^k(\RR^{4}), \quad \cordFL^k := \cordFL^k(\RR^{4}), \quad  \cycFL^k := \cycFL^k(\RR^{4}), \quad \ccycFL^k := \ccycFL^k(\RR^{4}).
    $$
    

\subsection{Triangles of flags and tetrahedra of flags}\label{subsec:tris_and_tets_flags}

The \emph{standard $n$--dimensional simplex} is the set
$$
\left\{(x_1,\dots,x_{n+1}) \in \mathbb{R}^{n+1}_{\geq 0} \ | \ x_1 + \dots + x_{n+1} = 1  \right\}.
$$
Its set of \emph{vertices} $\{\basis{1},\dots,\basis{n+1}\}$ is the standard basis of $\mathbb{R}^{n+1}$. Their natural order $(\basis{1},\dots,\basis{n+1})$ induces an orientation on the standard $n$--dimensional simplex. We will be mostly interested in the standard $3$--dimensional simplex, that we call the \emph{standard tetrahedron} and denote it by $\statet$. Furthermore, we define a \emph{(projective) triangle} (resp. a \emph{(projective) tetrahedron}) of $\rp^3$ to be a region in $\rp^3$ projectively equivalent to the closure of the projectivization of the positive orthant in $\RR^3$ (resp.\ $\RR^4$).

Let $\T = \cyclic{V_i,\eta_i}_{i=1}^3 \in \cycFL^3$ be a non-degenerate \emph{cyclically ordered} triple of flags. We say that $\T$ is a \emph{triangle of flags} if the following conditions are satisfied:
\begin{enumerate}
    \item the points $V_i$ are in general position;
	\item there is a triangle $\triangle_{\T} \subset V_1V_2V_3$ with vertices the three points $V_i$ whose interior is disjoint from all planes $\eta_i$.
\end{enumerate}

Because the points $V_i$ are in general position, they belong to a unique plane $V_1V_2V_3$ of $\rp^3$. The lines $V_iV_j$ through pairs of points $\{V_i,V_j\}$ divide $V_1V_2V_3$ in four projectively equivalent projective triangles. By non-degeneracy, each plane $\eta_i$ intersect $V_1V_2V_3$ in a line through $V_i$, but distinct from $V_iV_j$ and $V_iV_k$. Thus every plane $\eta_i$ intersects the interior of exactly two of the four projective triangles. It is easy to see that if they all miss one of them, then this triangle is unique.

We denote by $\triFL \subset \cycFL^3$ the subset of triangles of flags. As noted above, every triangle of flags $\T \in \triFL$ corresponds to a unique triangle $\triangle_{\T}$ in $\rp^3$, with a canonical cyclical ordering of the vertices, but not an ordering. The map $\T \mapsto \triangle_{\T}$ is surjective, but not finite-to-one, as we can always decorate the vertices of a triangle with infinitely many appropriate planes.

Let $\ctriFL \subset \ccycFL^3$ be the image of $\triFL$ under the quotient map $\pi_3 : \cycFL^3 \rightarrow \ccycFL^3$. The definition of a triangle of flags is invariant under projective transformation, therefore $\pi_3^{-1}(\ctriFL) = \triFL$. Given a triangle of flags $\T$, there is a corresponding triangle of flags $\overline{\T}$ obtained from $\T$ by applying an odd permutation to the flags in $\T$. They share the same underlying triangle but they have opposite cyclical ordering of the vertices.

Let $\F = (V_i,\eta_i)_{i=1}^4 \in \FL^4$ be a non-degenerate \emph{ordered} quadruple of flags. We say that $\F$ is a \emph{tetrahedron of flags} if the following conditions are satisfied:
\begin{enumerate}
    \item the points $V_i$ are in general position;
	\item there is a tetrahedron $\tet_{\F} \subset \rp^3$ with vertices the four points $V_i$ whose interior is disjoint from all planes $\eta_i$.
\end{enumerate}

\begin{figure}
  \centering
  \def\svgscale{.4}
  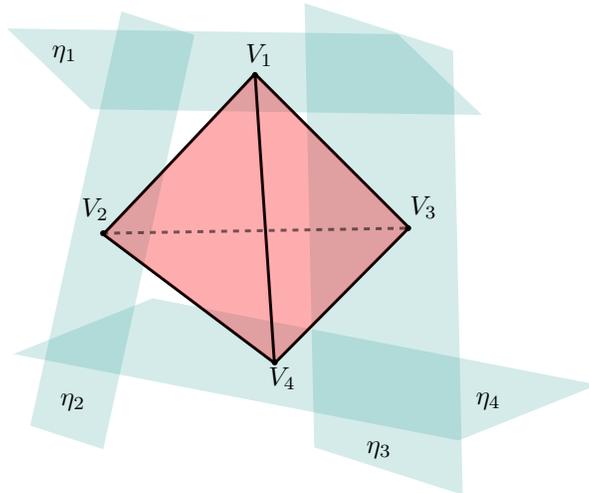
  \caption{A tetrahedron of flags in $\rp^3$}\label{tetflags}
\end{figure}

Because the points $V_i$ are in general position, the four planes $V_iV_jV_k$ through the triples of points $\{V_i,V_j,V_k\}$ divide $\rp^3$ in eight projectively equivalent projective tetrahedra. Once again, it is a simple exercise to show that if all of the planes $\eta_i$ miss a tetrahedron, then it is the only one.

We denote by $\tetFL \subset \FL^4$ the subset of tetrahedra of flags. Then by definition, every tetrahedron of flags $\F \in \tetFL$ corresponds to a unique tetrahedron $\tet_{\F}$ in $\rp^3$, with a canonical ordering of the vertices induced by the ordering of the flags. In particular it comes with a simplicial identification to the standard tetrahedron $\statet$. We say that $\tet_{\F}$ is \emph{positively oriented} if this identification to $\statet$ is orientation preserving. Otherwise it is \emph{negatively oriented}. We remark that $\PGL(4)$ contains orientation reversing projective transformations, therefore a tetrahedron may change orientation under projective transformation. Once again, the map $\F \mapsto \tet_{\F}$ is surjective but not finite-to-one.

Let $\ctetFL \subset \cordFL^4$ be the image of $\tetFL$ under the quotient map $\pi_4 : \FL^4 \rightarrow \cordFL^4$. The definition of a tetrahedron of flags is invariant under projective transformation, therefore $\pi^{-1}_4(\ctetFL) = \tetFL$.
	
\begin{remark}\label{rem:tet_are_not_dual}
The definition of a tetrahedron of flags is not invariant under projective duality, in the sense that the planes $\eta_i$ don't have to be in general position. However, when the planes are in general position, then they form a unique tetrahedron $\tet_{\F}^\ast \subset \rp^3$, whose faces are contained in the four planes $\eta_i$, that contains the tetrahedron $\tet_{\F}$.
\end{remark}

The following result shows that the action of $\PGL(4)$ on $\ctetFL$ is free.

\begin{lemma}\label{lem:tet_stab_is_trivial}
    The stabilizer of a tetrahedron of flags in $\PGL(4)$ is trivial.
\end{lemma}

\begin{proof}
    We recall that the stabilizer of a projective basis of $\rp^3$ in $\PGL(4)$ is trivial. Let $\F = (V_i,\eta_i)_{i=1}^4 \in \tetFL$ be a tetrahedron of flags.
    The statement follows from the fact that at least one of the following $5$--tuples of points is a projective basis: 
    {
    $$
    \{V_1,V_2,V_3,V_4,\eta_1\eta_2\eta_3\}, \quad \text{or} \quad \{V_1,V_2,V_3,V_4,\eta_1\eta_2\eta_4\}, \quad \text{or} \quad \{V_1,V_2,V_3,\eta_1\eta_2\eta_3,\eta_2\eta_3\eta_4\}.
    $$}
\end{proof}

{
Let $\F = (V_m,\eta_m)_{m=1}^4 \in \tetFL$ be an ordered quadruple of flags. For every even permutation $(ijkl)$ of $(1234)$, there is an associated ordered triple of flags
$$
(\F)_{ijk} := \left((V_i,\eta_i),(V_j,\eta_j),(V_k,\eta_k)\right).
$$
We call $(\F)_{ijk}$ a \emph{marked face} of $\F$. The corresponding cyclically ordered triple of flags
$$
\cyclic{\F}_{ijk}:= \cyclic{(V_i,\eta_i),(V_j,\eta_j),(V_k,\eta_k)}
$$
is simply a \emph{face} of $\F$. The terminology is only meaningful when $\F$ is a tetrahedron of flags. Indeed it follows directly from the definitions that faces of tetrahedra of flags are triangles of flags.} We record this fact in the following result for future reference.

\begin{lemma}\label{lem:tet_to_triangles}
    Every face of a tetrahedron of flags is a triangle of flags.
\end{lemma}

Let $\F=(V_m,\eta_m)_{m=1}^4$ and $\E=(W_m,\zeta_m)_{m=1}^4$ be two tetrahedra of flags and let $(ijkl)$ and $(i'j'k'l')$ be two even permutations of $(1234)$. If
$$
(V_i,\eta_i) = (W_{j'},\zeta_{j'}), \qquad (V_j,\eta_j) = (W_{i'},\zeta_{i'}), \qquad \text{ and } \qquad
(V_k,\eta_k) = (W_{k'},\zeta_{k'}),
$$
then we say that $\F$ and $\E$ are \emph{glued (along the marked faces $(\F)_{ijk}$ and $(\E)_{i'j'k'}$)}. We remark that, in this case, $\cyclic{\F}_{ijk} = \overline{\cyclic{\E}}_{i'j'k'}$. Let $\tet_\F,\tet_\E$ be the tetrahedra of $\rp^3$ associated to $\F$ and $\E$. If $\F$ and $\E$ are glued together, then $\tet_\F$ and $\tet_\E$ share a face $f$. In general $f \subset \tet_{\E} \cap \tet_{\F}$, and we say that the pair $(\F,\E)$ is \emph{geometric} if $f = \tet_{\E} \cap \tet_{\F}$. 

We remark that a pair $(\F,\E)$ is not always geometric, as both tetrahedra $\tet_{\F},\tet_{\E}$ might be ``on the same side'' of $f$.


\subsection{Triangulations of flags}\label{subsec:tri_of_flags}

We begin by defining ideal triangulations, inspired from~\cite{CLT17}. Let $\mathcal{U} = \sqcup_{i=1}^k\statet^i$ be the disjoint union of $k$ copies of the standard tetrahedron, and $\Psi$ be a family of orientation-reversing simplicial isomorphisms pairing the faces of $\mathcal{U}$, such that:
\begin{itemize}
    \item $\varphi \in \Psi$ if and only if $\varphi^{-1}\in \Psi$;
    \item every face is the domain of a unique element of $\Psi$.
\end{itemize}
The elements of $\Psi$ are called \emph{face pairings}. The quotient space
$$
\dot{M} = \mathcal{U}_{/\Psi}
$$
is a closed, orientable $3$--dimensional $\CW$--complex, and the quotient map is denoted $p : \mathcal{U} \to \dot{M}$. The triple $\Delta = ( \mathcal{U}, \Psi, p)$ is a \emph{(singular) triangulation} of $\dot{M}$. The adjective \emph{singular} is usually omitted, and we will not need to distinguish between the cases of a simplicial or a singular triangulation.
We will always assume that $\dot{M}$ is connected. In the case where $\dot{M}$ is not connected, the results of this paper apply to its connected components.

The set of non-manifold points of $\dot{M}$ is contained in the $0$--skeleton $\dot{M}^{(0)}$, thus $M := \dot{M} \setminus \dot{M}^{(0)}$ is a non-compact orientable $3$--manifold. We say that $\Delta$ is an \emph{ideal triangulation} of $M$ and $\dot{M}$ is its \emph{end-compactification}.

We adopt the following notation:
$\Vertex(\Delta),\Edge(\Delta),\Face(\Delta),\Tet(\Delta)$ will denote the sets of $0$--cells, $1$--cells, $2$--cells and $3$--cells of $\dot{M}$, respectively. These sets are called the sets of \emph{(ideal) vertices, edges, faces and tetrahedra of $\Delta$}, respectively.

Let $\wDelta$ be the ideal triangulation of the universal cover $\widetilde{M}$ obtained by lifting $\Delta$. The \emph{space of (ideal) triangulations of flags $\TriFL$} (of $M$ with respect to $\Delta$) is the set of pairs $(\Phi,\rho)$ where
$$
\Phi : \Vertex(\wDelta) \rightarrow \FL, \qquad \text{and} \qquad \rho : \pi_1(M) \rightarrow \PGL(4),
$$
satisfy the following conditions.
\begin{enumerate}
    \item \label{item:def_tri_of_flags_1} For all $\tet \in \Tet(\wDelta)$, the image of the vertices of $\tet$ forms a tetrahedron of flags. Namely
    $$
    \Phi(\tet) := \left(\Phi(v_1),\Phi(v_2),\Phi(v_3),\Phi(v_4) \right) \in \tetFL,
    $$
    for $\{v_1,v_2,v_3,v_4\} = \Vertex(\tet)$.
    \item \label{item:def_tri_of_flags_2} If $\tet,\tet' \in \Tet(\wDelta)$ are glued along a face $f$, then the corresponding tetrahedra of flags $\Phi(\tet),\Phi(\tet')$ are glued geometrically along the their faces that correspond to $f$.
    \item \label{item:def_tri_of_flags_3} The map $\Phi$ is $\rho$--equivariant. Namely if $v \in \Vertex(\wDelta)$ and $\gamma\in \pi_1(M)$ then 
    $$
    \Phi(\gamma\cdot v)=\rho(\gamma)\cdot\Phi(v).
    $$
\end{enumerate}
The map $\Phi$ is called a \emph{flag decoration} of $\rho$. There is an action of $\PGL(4)$ on $\TriFL$ given by postcomposition in the first factor and conjugation in the second factor, and we let $\cTriFL$ be the corresponding quotient space. When its clear from the context, we will sometimes refer to a class of triangulation of flags simply as a triangulation of flags. We will see in \S\ref{subsec:developing_tri_of_flags} that classes of triangulations of flags are closely related to branched projective structures on $M$. One of the main goals of the first half of this paper is to parametrize the space $\cTriFL$ (cf. Theorem~\ref{thm:parametrization_def_space}).

\begin{remark}\label{rem:dev_determines_rho}
Let $\gamma\in \pi_1(M)$ and let $\tet\in \Tet(\wDelta)$. If $\Phi(\tet)$ and $\Phi(\gamma\cdot \tet)$ are projectively equivalent then there is a unique projective transformation mapping one to the other (cf. Lemma~\ref{lem:tet_stab_is_trivial}). It follows that the holonomy of a triangulation of flags can be recovered from the flag decoration.
\end{remark}


\subsection{Projective structures}\label{subsec:proj_structures}

Let $M$ be a manifold of dimension $n$. In this section we describe projective structures on $M$. A projective structure is a special instance of a $(G,X)$--structure; a detailed account of such structures can be found in \cite{RAT} or \cite{THUNOTES}. A \emph{projective structure} on $M$ consists of a (maximal) atlas of charts $(U_\alpha,\phi_\alpha)_{\alpha\in \mathcal{A}}$, where the $U_\alpha$ cover $M$, $\phi_\alpha:U_\alpha\to \rp^n$ is a diffeomorphism onto its image, and if $U_\alpha\cap U_\beta\neq \emptyset$ then $\phi_\alpha\circ \phi_\beta^{-1}$ restricts to an element of $\PGL(n+1)$ on each connected component of $U_\alpha\cap U_\beta$. There is a more global description of such a structure given by a pair $(\dev,\hol)$, where $\dev:\wt M\to \rp^n$ is a local diffeomorphism called a \emph{developing map}, $\hol:\pi_1(M)\to \PGL(n+1)$ is a representation called a \emph{holonomy representation}, and $\dev$ is $\hol$--equivariant in the sense that
$$
\dev(\gamma\cdot x)=\hol(\gamma)\dev(x),\qquad \text{for every} \quad x\in \wt M, \quad \gamma\in \pi_1(M).
$$
Given a structure, a developing map can be constructed via analytic continuation. 

There is a natural equivalence relation that can be put on the set of projective structures on $M$. We begin by describing the simpler case where $M$ is compact. In this context, we say that two structures $(\dev,\hol)$ and $(\dev',\hol')$ are \emph{equivalent} if there is an element $G\in \PGL(n+1)$ so that, up to precomposing $\dev$ and $\dev'$ with equivariant isotopies, 
$$
\dev'=G\circ\dev, \qquad \text{and} \qquad \hol'=G \cdot \hol,
$$
where the action of $G$ on $\hol$ is by conjugation. Our primary case of interest is when $M$ is non-compact, but is the interior of a compact manifold $\hat M$. In this setting we call a compact submanifold $N\subset M$ a \emph{compact core} if $N$ is the complement in $\hat M$ of an open collar neighborhood of the boundary of $\hat M$. Two projective structures $(\dev,\hol)$ and $(\dev',\hol')$ on $M$ are \emph{equivalent} if for some choice of compact core, $N$, they are equivalent (using the previous definition) when restricted to $N$. Let $\rp(M)$ be the \emph{space of equivalence classes of projective structures on $M$} and let $\Char(M):=\Hom(\pi_1(M),\PGL(n+1))/\PGL(n+1)$ be the $\PGL(n+1)$--\emph{character variety} of $\pi_1(M)$, then there is a map 
$$
\Hol:\rp(M)\to \Char(M), 
$$
that takes a class of structures to the conjugacy class of its holonomy representations. 

Both $\rp(M)$ and $\Char(M)$ are natural topological spaces. The space of developing maps of projective structures on $M$ is a subspace of the set $C^\infty(\wt M,\rp^n)$ of smooth maps from $\wt M$ to $\rp^n$. The space $C^\infty(\wt M,\rp^n)$ can be equipped with the \emph{weak topology} (see \cite[pp.\ 35]{HIR} for definition), the space of developing maps can then be equipped with the corresponding subspace topology, and $\rp(M)$ can be equipped with the corresponding quotient topology. The space $\Hom(\pi_1(M),\PGL(n+1))$ can be equipped with the compact-open topology and $\Char(M)$ can be equipped with the corresponding quotient topology. With respect to these topologies, we have the following well known result (see \cite{CEG} for a nice exposition).

\begin{theorem}[The Ehresmann-Thurston principle, \cite{THUNOTES}]\label{thm:ersh_thu_principle}
    The map $\Hol:\rp(M)\to \Char(M)$ is a local homeomorphism.
\end{theorem}

A $(G,X)$--geometry such that $X\subset \rp^n$ and $G\subset \PGL(n+1)$ is called a \emph{subgeometry} of projective geometry. Using the previous construction, we can build the space $X(M)$ of equivalence classes of $(G,X)$--structures on $M$. If $(G,X)$ is a subgeometry of projective geometry, there is a map $X(M)\to \rp(M)$ that associates an equivalence class of $(G,X)$--structures to the underlying class of projective structures. For certain subgeometries this map can fail to be injective. Two examples that will be relevant for our purposes are \emph{hyperbolic geometry}, modeled on $(\PSO(3,1),\HH^3)$, and \emph{Anti-de Sitter geometry}, modeled on $(\PSO(2,2),\ADS)$ (cf. \S\ref{sec:thu_equations}). 

There is another type of structures that will come out throughout this work, called \emph{branched structures}. Roughly speaking, they are generalizations of geometric structures where instead of insisting that the charts are local diffeomorphisms, we only require that they are branched covering maps. This construction is quite general, but we will only have occasion to use a very specific instance, which we now describe. Let $\Delta$ be an ideal triangulation of a $3$--manifold $M$. A \emph{branched projective structure on $M$ with respect to $\Delta$} is a (maximal) collection of charts $(U_\alpha,\phi_\alpha)_{\alpha\in \mathcal{A}}$ so that
\begin{itemize}
    \item if $U_\alpha$ is disjoint from $\Edge(\Delta)$, then $\phi_\alpha:U_\alpha\to \rp^3$ is a local diffeomorphism;
    \item if $U_\alpha\cap \Edge(\Delta)\neq \emptyset$ then $\phi_\alpha$ maps the components of $U_\alpha\cap \Edge(\Delta)$ diffeomorphically to disjoint projective line segment, and each $p\in U_\alpha\cap \Edge(\Delta)$ has a neighborhood $\mathcal{N}_p$ on which the restriction of $\phi_\alpha$ is a branched cyclic covering of finite order with branch locus $\mathcal{N}_p\cap \Edge(\Delta)$.
\end{itemize} 
As before, we also insist that if the domains of two charts intersect then the transition map restricts to an element of $\PGL(4)$, in each component of intersection. The \emph{branch locus} $\Sigma$ of a branched structure is the subset of $\Edge(\Delta)$ where the charts are non-trivially branching. If $\Sigma = \emptyset$, then the structure is a genuine projective structure.\\

As with projective structures, branched projective structures can be described globally as a pair $(\dev,\hol)$, where $\hol:\pi_1(M)\rightarrow \PGL(4)$ is a  representation, and $\dev:\wt M\rightarrow \rp^3$ is a $\hol$--equivariant local diffeomorphism away from $\Edge(\wDelta)$, and locally a cyclic branched cover of finite order at $\Edge(\wDelta)$. 

Finally, the same equivalence relation we adopted for projective structures can be applied verbatim to branched projective structures. We denote by $\rp(M,\Delta)$ the \emph{space of equivalence classes of branched projective structures} on $M$ with respect to $\Delta$. If $(G,X)$ is a subgeometry of projective geometry, we denote by $X(M;\Delta) \subset \rp(M;\Delta)$ the space of equivalence classes of \emph{branched} $(G,X)$--structures.


\subsection{Developing triangulations of flags}\label{subsec:developing_tri_of_flags}

Let $\Delta$ be an ideal triangulation of a $3$--manifold $M$. We conclude this section by describing the relationship between triangulation of flags $\cTriFL$ and branched projective structures $\rp(M,\Delta)$.
Let $[\Phi,\rho]\in \cTriFL$ be a triangulation of flags and let $(\Phi,\rho)$ be a representative pair. For each tetrahedron $\wt \tet\in \Tet(\wDelta)$, the image of the vertices of $\wt \tet$ under $\Phi$ is a tetrahedron of flags $\F_{\wt \tet}$. Then by definition, there is a unique tetrahedron ${\wt \tet}_\F$ in $\rp^3$ associated to $\F_{\wt \tet}$, and we can simplicially extend $\Phi$ to an isomorphism $\wt \tet \rightarrow {\wt \tet}_\F$. Repeating this construction for each tetrahedron in $\Tet(\wDelta)$ gives a map $\wDelta \to \rp^3$, hence in particular a map $\hat \Phi:\wt M\to \rp^3$. By construction, $\hat \Phi$ is $\rho$--equivariant. We say that $\hat \Phi$ is the \emph{simplicial developing extension} of $(\Phi,\rho)$.

\begin{theorem}\label{thm:geom_flag_struc_are_developable}
    Let $[\Phi,\rho]\in \cTriFL$, and let $\hat\Phi$ be the simplicial developing extension of a representative pair $(\Phi,\rho)$. Then $[\hat \Phi,\rho]\in \rp(M,\Delta)$. In particular, there is a well defined continuous map
    $$
    \Ext : \cTriFL \rightarrow \rp(M,\Delta), \qquad \text{where} \qquad [\Phi,\rho] \mapsto [\hat \Phi,\rho].
    $$
\end{theorem}

The map $\Ext$ is called the \emph{extension map} of $\cTriFL$.

\begin{proof}
    First we are going to show that $[\hat \Phi,\rho]\in \rp(M,\Delta)$. Since $\hat \Phi:\wt M\to \rp^3$ and $\hat \Phi$ is $\rho$--equivariant, we only need to show that $\hat \Phi$ is a local homeomorphism away from the edges of $\wDelta$, and locally a cyclic branched cover of finite order otherwise.
    
    We recall that $\hat \Phi$ is a homeomorphism when restricted to each individual tetrahedron $\wt \tet$ of $\wDelta$, therefore a local homeomorphism in the interior of $\wt \tet$. Item~\eqref{item:def_tri_of_flags_2} in the definition of an ideal triangulation of flags implies that $\hat \Phi$ is a local homeomorphism in the interior of each face too. Furthermore, each edge in $\Edge(\wDelta)$ is mapped to a segment of a projective line. To conclude that $\hat \Phi$ is a cyclic branched cover of finite order in a neighbourhood of each edge of $\wDelta$, it is enough to notice that $\hat \Phi$ is $\rho$--equivariant and $\rho$ is a representation of $\pi_1(M)$. Therefore any closed simple loop around an edge $s \in \Edge(\wDelta)$, and contained in a small enough neighbourhood of $s$, is mapped to a closed loop in $\rp^3$ going around $\hat \Phi(s)$ finitely many times. It follows that $[\hat \Phi,\rho]\in \rp(M,\Delta)$.
    
    Finally, suppose $(\Phi_1,\rho_1)$ is another representative pair for $[\Phi,\rho]$. Then $(\Phi_1,\rho_1)$ is equivalent to $(\Phi,\rho)$, and it is easy to see that so are $(\hat \Phi_1,\rho_1)$ and $(\hat \Phi,\rho)$. Hence $[\hat \Phi_1,\rho_1] = [\hat \Phi,\rho]$ and $\Ext$ is well defined. The restriction of $\hat\Phi$ to each tetrahedron $\wt\tet\in \Tet(\wt \Delta)$ depends continuously on the vertices of the tetrahedron of flags $\Phi(\wt \tet)$. Since there are only finitely many tetrahedra in $\Delta$ and $\hat \Phi$ is equivariant, continuity of $\Ext$ follows. 
\end{proof}

	
\section{Coordinates on a tetrahedron of flags}\label{sec:coords_teton}

In this section we describe a convenient set of coordinates on the space $\ctetFL$ of $\PGL(4)$--classes of tetrahedra of flags. This is the first step towards a parametrization of $\ctriFL$. The neatest way to do so is to introduce the concept of an \emph{edge-face} of a tetrahedron (cf.\ \S\ref{subsec:edge_faces}), and to associate to each of them a meaningful positive real number. The coordinates we use are \emph{triple ratios} and \emph{edge ratios} (cf.\ \S\ref{subsec:triple_edge_ratios}), partially inspired by \cite{BFGFlags,FockGonch}. These projective coordinates are defined for quadruple of flags, but are positive if and only if the flags form a tetrahedron of flags (cf. Lemma~\ref{lem:positive_coords}). We show that they satisfy some \emph{internal consistency equations} (cf. Lemma~\ref{lem:internal_eqs}), defining a \emph{deformation space} $\tetsubvariety$, homeomorphic to $\RR^5_{>0}$. Using \emph{edge-face standard position} (cf.\ \S\ref{subsec:standard_position}), we will show that $\ctetFL$ is homeomorphic to $\tetsubvariety$ in Theorem~\ref{thm:parametrization_flag_space}. We conclude with a remark that projective tetrahedra coming from tetrahedra of flags in edge-face standard position are always \emph{positively oriented} (cf.\ \S\ref{subsec:orientation}).


\subsection{Edge-faces}\label{subsec:edge_faces}

Recall that $\statet$ is the standard $3$--dimensional simplex (cf. \S\ref{subsec:tris_and_tets_flags}). An \emph{edge-face} of $\statet$ is an ordered pair $\sigma = (e,f)$ consisting of an (oriented) face $f$ of $\statet$ and an (oriented) edge $e$ of $f$, so that the orientation of $f$ is the one induced from $\statet$, and the orientation of $e$ is the one induced from $f$. The components of an edge-face $\sigma$ are called the \emph{edge} and \emph{face of $\sigma$}, respectively. We denote by $\edgeface$ the set of all edge-faces of $\statet$. 

There is a natural identification between $\edgeface$ and the \emph{alternating group on four symbols} $\Alt(4) \leqslant \Sym(4)$, namely the subgroup of even permutations. This can be best described using the non-standard notation where $[ijkl] \in \Alt(4)$ is the permutation mapping $(1,2,3,4)$ to $(i,j,k,l)$. Since every vertex of $\statet$ is numbered, every edge-face $\sigma = (e,f)$ can be encoded as $(ij)k$, where $(ij)$ is the oriented edge $e$ and $\cyclic{ijk}$ is the oriented face $f$. Then we have a bijection
$$
\Alt(4) \rightarrow \edgeface, \qquad \text{where} \qquad [ijkl] \mapsto  (ij)k.
$$
Henceforth we will implicitly make use of this identification, and abuse the notation by indicating with $\sigma$ both the edge-face and the corresponding even permutation.

The group $\Alt(4)$ acts simply transitively on itself \emph{on the right} by multiplication, thus giving a simply transitive right action of $\Alt(4)$ on $\edgeface$. As a consequence, one can visualize the set $\edgeface$ as the set of vertices of the Cayley graph $\Cay(4)$ of $\Alt(4)$.

We introduce the following notation (cf. Figure~\ref{fig:EdgeFace_labels}). Let $\alpha=[3124]$ and $\beta=[2143]$ be a generating set of $\Alt(4)$ (these are $(132)$ and $(12)(34)$ in the standard notation). If $\sigma=(ij)k$, then we define
$$
\sigma_+:=\sigma\cdot \alpha=(ki)j, \qquad \text{and} \qquad \sigma_-:=\sigma\cdot \alpha^{-1}=(jk)i.
$$
These are the three edge-faces that share the face $\cyclic{ijk}$. Furthermore, let
$$
\overline{\sigma}:=\sigma\cdot \beta=(ji)l, \qquad \op{\sigma}:=\sigma\cdot \alpha^{-1}\beta\alpha=(lk)j, \qquad \text{and} \qquad \overline{\op{\sigma}} := \sigma\cdot \alpha\beta\alpha^{-1}=(kl)i,
$$
be the \emph{conjugate}, \emph{positive opposite} and \emph{negative opposite} edge-faces of $\sigma$, respectively. We remark that conjugate edge-faces share the same unoriented edge, but with opposite orientations. On the other hand, the edges of opposite edge-faces are opposite in $\statet$, namely they do not share a vertex.

\begin{remark}
It should be noted that some symbols commute while some do not. For example $\overline{\op{\sigma}} = \op{\overline{\sigma}}$, but $ \overline{\sigma_-}\neq  \overline{\sigma}_-=(\op{\overline{\sigma_-})}$.
\end{remark}

In light of the above notation, one can meaningfully embed $\Cay(4)$ inside $\statet$ so that each edge-face $(e,f)$ lies inside $f$ and next to $e$ (cf. Figure~\ref{fig:EdgeFace_labels}). We remark that the set of faces of $\statet$ has a natural identification with (left) cosets of $\langle \alpha\rangle$ in $\Alt(4)$, while the set of edges can be identified with (left) cosets of $\langle \beta\rangle$ in $\Alt(4)$.  

\begin{figure}[t]
    \centering
    \includegraphics[width=\textwidth]{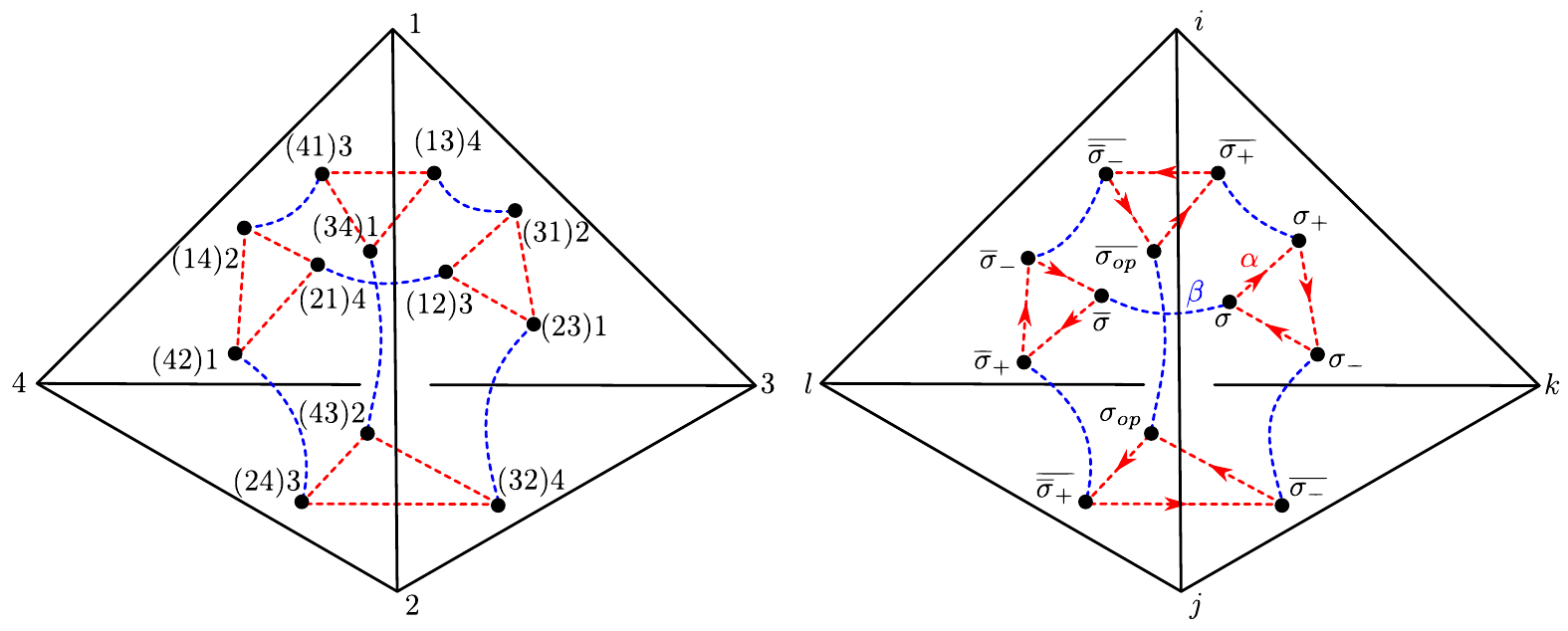}
    \caption{An embedding of the Cayley graph $\Cay(4)$ of $\Alt(4)$ into $\statet$ where each edge-face $\sigma = (ij)k$ lies inside the face $\cyclic{ijk}$ and next to the edge $(ij)$.}    \label{fig:EdgeFace_labels}
\end{figure}


\subsection{Triple ratios and edge ratios}\label{subsec:triple_edge_ratios}

Each flag in $\rp^3$ has $5$ degrees of freedom, and hence $\cycFL^3$ is a $15$--dimensional space. Since $\PGL(4)$ is also a $15$--dimensional space, one might naively expect that all cyclically ordered triples of flags are in the same $\PGL(4)$--orbit, however, this turns out not to be the case. 

Recall that $\RR^\times := \RR \setminus\{0\}$. We define the continuous map
$$
\overline{\3}_\circ : \cycFL^3 \rightarrow \RR^\times, \qquad \text{where} \qquad \T = \cyclic{V_m,\eta_m}_{m=1}^3 \mapsto \frac{\overline{\eta_1}(\overline{V_2})\overline{\eta_2}(\overline{V_3})\overline{\eta_3}(\overline{V_1})}{\overline{\eta_1}(\overline{V_3})\overline{\eta_2}(\overline{V_1})\overline{\eta_3}(\overline{V_2})} \in \RR^\times,
$$
where $\overline{V_m} \in \RR^4$ and $\overline{\eta_m} \in (\RR^4)^*$ are representatives of $V_m$ and $\eta_m$, respectively. It is easy to check that this quantity is well defined: it is independent of the choice of representatives $\overline{V_m}$ and $\overline{\eta_m}$,
numerator and denominator are non-zero by non-degeneracy of the flags,  and it is invariant under cyclic permutations of the three flags. The \emph{triple ratio of $\T$} is the number
$$
t^\T := \overline{\3}_\circ(\T).
$$
We remark that the triple ratio is invariant under projective transformations, namely
$$
\overline{\3}_\circ(G \cdot \T) = \overline{\3}_\circ(\T), \qquad \forall G \in \PGL(4).
$$
Therefore $\overline{\3}_\circ$ descends to a function
$$
\3_\circ : \ccycFL^3 \rightarrow \RR^\times.
$$
It turns out that one can characterize the subspace of triangle of flags $\ctriFL \subset \ccycFL^3$ via $\3_\circ$. The following result is a straightforward consequence of Theorem~\cite[Theorem $2.2$]{FockGonch}. See also \cite{CTT20} for a proof of continuity and more details.

\begin{lemma}\label{lem:triple_ratio}
    Both the map $\3_\circ : \ccycFL^3 \rightarrow \RR^\times$ and the restriction map $\restr{\3_\circ}{\ctriFL} : \ctriFL \rightarrow \RR_{>0}$ are homeomorphisms. 
\end{lemma}
\begin{proof}
If $\T = \cyclic{V_m,\eta_m}_{m=1}^3 \in \cycFL^3$, then the intersections of the plane $V_1V_2V_3$ with $\eta_1$, $\eta_2$, and $\eta_3$ give three lines $\ell_1$, $\ell_2$, and $\ell_3$ in $V_1V_2V_3$. By identifying $V_1V_2V_3$ with $\rp^2$, we see that $\T' := \cyclic{V_m,[\ell_m]}_{m=1}^3$ is a cyclically ordered triple of complete flags in $\rp^2$. Then the triple ratio $\3_\circ(\T)$ is equal to the Fock-Goncharov triple ratio of $\T'$ and the result follows from \cite[Theorem $2.2$]{FockGonch}.
\end{proof}

For each edge-face $\sigma \in \edgeface$ and each ordered quadruple of flags $\F \in \FL^4$, we recall that $\cyclic{\F}_\sigma$ is a cyclically ordered triple of flags. Thus we have a well defined continuous map
$$
\overline{\3} : \ordFL^4 \times \edgeface \rightarrow \RR^\times, \qquad \text{where} \qquad (\F,\sigma) \mapsto \3_\circ\left( \cyclic{\F}_\sigma \right).
$$
The \emph{triple ratio of $\F$ with respect to $\sigma$} is
\begin{equation}\label{eq:triple_ratio_def}
    t_{\sigma}^\F := \3_\circ\left(\cyclic{\F}_\sigma\right) = \overline{\3}(\F,\sigma).
\end{equation}
The fact that the triple ratio is invariant under projective transformations implies that the map $\overline{\3}$ descends to a continuous map
$$
\3 : \cordFL^4 \times \edgeface \rightarrow \RR^\times,
$$
which by abuse of notation we have also denote by $\3$. This is our first type of coordinate on $\cordFL^4 \times \edgeface$.

Now we are going to define a second function on the space $\cordFL^4 \times \edgeface$. Recall that the edge-face $\sigma=(ij)k$ corresponds to the permutation $[ijkl] \in \Alt(4)$. We define the continuous map
$$
\overline{\xi} : \ordFL^4 \times \edgeface \rightarrow \RR^\times, \qquad \text{where} \qquad (\F,\sigma) = \left((V_m,\eta_m)_{m=1}^4,(ij)k\right) \mapsto \frac{\overline{\eta_i}(\overline{V_k})\overline{\eta_j}(\overline{V_l})}{\overline{\eta_i}(\overline{V_l})\overline{\eta_j}(\overline{V_k})} \in \RR^\times.
$$
for representatives $\overline{V}_m \in \RR^4$ and $\overline{\eta}_m \in (\RR^4)^*$ of $V_m$ and $\eta_m$, respectively. Once again, the quantity $\overline{\xi}(\F,\sigma)$ is well defined as the numerator and denominator are non-zero by non-degeneracy of the flags, and it is independent of the choice of representative of $V_i$ and $\eta_i$. The \emph{edge ratio of $\F$ with respect to $\sigma$} is then
\begin{equation}\label{eq:edge_ratio_def}
    e_\sigma^\F := \overline{\xi}(\F,\sigma).
\end{equation}
As with the triple ratio, the edge ratio is invariant under projective transformations. Therefore $\overline{\xi}$ descends to a continuous map
$$
\xi : \cordFL^4 \times \edgeface \rightarrow \RR^\times.
$$
The edge ratio is inspired from a coordinate defined in \cite{BFGFlags}, and we will see shortly that it admits a geometric description coming from its interpretation as a cross ratio (cf. Lemma~\ref{lem:edge_to_cross}) .

For both triple ratios and edge ratios, it is often convenient to have a notation that makes the edge-face $\sigma$ more explicit. If $\sigma=(ij)k$ then we sometimes use the notations
$$
t^\F_{ijk} = t^\F_\sigma = t^{[\F]}_\sigma, \qquad \text{and} \qquad e_{ij}^\F=e_\sigma^\F = e_\sigma^{[\F]}.
$$
This convention for the edge ratio should not create confusion: given distinct $i,j \in \{1,2,3,4\}$ there is a unique choice of $k\in \{1,2,3,4\}\setminus\{i,j\}$ so that $(ij)k\in \edgeface$. Furthermore, we will occasionally omit superscripts when $\F$ is clear from context (see Lemma~\ref{lem:internal_eqs} for example). 

It follows directly from the definitions of the triple ratio \eqref{eq:triple_ratio_def} and the edge ratio \eqref{eq:edge_ratio_def} that these two ratios satisfy several relations. They are called the \emph{internal consistency equations}, and are summarized in the following result.

\begin{lemma}\label{lem:internal_eqs}
   Let $\F \in \ordFL^4$ and let $\sigma = (ij)k \in \edgeface$, then
    \begin{alignat}{2}
    &t_\sigma^\F=t_{\sigma_+}^{\F}=t_{\sigma_-}^\F, \qquad &&\text{or} \qquad t_{ijk}=t_{kij}=t_{jki}, \label{eq:triple_rat_eq}\\
    &e^\F_\sigma=e^\F_{\overline{\sigma}}, \qquad &&\text{or} \qquad e_{ij}=e_{ji}, \label{eq:edge_rat_eq}\\
    &t^\F_{\sigma} (e^\F_{\sigma}e^\F_{\sigma_+}e^\F_{\sigma_-})=1, \qquad &&\text{or} \qquad t_{ijk} (e_{ij}e_{ki}e_{jk})=1, \label{eq:fe1}\\
    &t_\sigma^\F= e_{\op{\sigma}}^\F e_{\op{(\sigma_+)}}^\F e_{\op{(\sigma_-)}}^\F, \qquad &&\text{or} \qquad t_{ijk}= e_{kl} e_{jl} e_{il}. \label{eq:fe2}
\end{alignat}
In particular
\begin{equation}
    e^\F_\sigma e^\F_{\sigma_+}e^\F_{\sigma_-}e_{\op{\sigma}}^\F e_{\op{(\sigma_+)}}^\F e_{\op{(\sigma_-)}}^\F=1, \qquad \text{or} \qquad e_{ij} e_{ki}e_{jk}e_{kl} e_{jl} e_{il}=1. \label{eq:fe3}
\end{equation}
\end{lemma}

There is a simple way to associate triple ratios and edge ratios to the edges and faces of the standard simplex $\statet$ that makes the internal consistency equations easier to remember. For every edge-face $\sigma = (ij)k$, we label the (unoriented) edge of $\statet$ with vertices $\{i,j\}$ with the edge ratio $e_\sigma$, and we label the (unoriented) face of $\statet$ with vertices $\{i,j,k\}$ with the triple ratio $t_\sigma$. The relations \eqref{eq:triple_rat_eq} and \eqref{eq:edge_rat_eq} show that this labeling is well defined. Equation \eqref{eq:fe1} says that the product of the three edge ratios on the edges of a face in $\statet$ is the inverse of the triple ratio of that face. Similarly, equation \eqref{eq:fe2} says that the product of the edge ratios on the three edges emanating from a vertex is equal to the triple ratio of the face that is opposite to that vertex.


\subsection{Edge-face standard position}\label{subsec:standard_position}

Here we develop one of the main tools for constructing the parametrization of $\ctetFL$ in \S\ref{subsec:parametrization_tet_flags}. In few words, for every edge-face $\sigma \in \edgeface$, we are going to construct a continuous section $\ctetFL \rightarrow \tetFL$ that allows us to single out a preferred tetrahedron of flags in its $\PGL(4)$--class.

To state the next result, we use the following notation. For $1\leq i\leq 4$, we denote by $\basis{i}$ (resp. $\dbasis{i}$) the $i$--th standard basis vector in $\RR^4$ (resp. $(\RR^4)^*$), and by $[\basis{i}]$ (resp. $[\dbasis{i}]$) the corresponding point in $\rp^3$ (resp. $(\rp^3)^*$). Furthermore, for all $[\F]\in \ctetFL$ and all $\sigma = (ij)k \in \edgeface$, we define
$$
\mu_\sigma^\F := e^\F_{\sigma_-}e^\F_{\sigma_+} -e_{\sigma_-}^\F + 1,\qquad \text{or} \qquad \mu_{ij}^\F := e^\F_{jk}e^\F_{ki} -e_{jk}^\F + 1.
$$

\begin{lemma}\label{lem:std_pos_tetrahedron}
    For each edge-face $\sigma=(ij)k$ and for each class $[\F] \in \ctetFL$ of tetrahedra of flags there is a unique representative $\F = (V_m,\eta_m)_{m=1}^4 \in \tetFL$ such that
    \begin{align}\label{eq:sp1}
    (V_i,\eta_i) = ([\basis{1}],[\dbasis{2}]), \qquad (V_j,\eta_j) = ([\basis{2}],[\dbasis{1}]),
    \\\label{eq:sp2}
    (V_k,\eta_k) = ([\basis{1}+\basis{2}+\basis{3}],[t_{\sigma} \cdot \dbasis{1} +\dbasis{2} - (t_{\sigma}+1) \cdot \dbasis{3}],
    \\\label{eq:sp3}
    V_l = [e_{\sigma} \cdot \basis{1} + \basis{2} + X_\sigma \cdot \basis{3} - \basis{4}],
    \\\label{eq:sp4}
    \eta_l = [e_{\overline{\sigma}_-}e_{\overline{\sigma}_+} \cdot \dbasis{1} + \dbasis{2} - \mu_{\overline{\sigma}}\cdot \dbasis{3} +  \mu_{\overline{\sigma}}\left( Y_{\overline{\sigma}} - X_\sigma \right) \cdot \basis{4}^*],
    \end{align}
    where
    $$
    X_\sigma := \frac{\mu_{\sigma}t_\sigma e_\sigma}{t_{\sigma}+1}, \qquad \text{ and } \qquad Y_{\overline{\sigma}} := \frac{t_{\overline{\sigma}}+1}{t_{\overline{\sigma}}\mu_{\overline{\sigma}}}.
    $$
\end{lemma}

The representative $\F$ from Lemma~\ref{lem:std_pos_tetrahedron} is \emph{the $\sigma$--standard representative of $[\F]$}. We will also say that a tetrahedron of flags $\F$ is \emph{in $\sigma$--standard position} if it is equal to the $\sigma$--standard representative of its $\PGL(4)$--class.

\begin{proof}
    Let $\F = (V_m,\eta_m)_{m=1}^4$ be a class representative for $[\F] \in \ctetFL$. We recall that the face $\cyclic{\F}_\sigma$ is a triangle of flags (cf. Lemma~\ref{lem:tet_to_triangles}). We claim that the quadruple of points $\{V_i,V_j,V_k,
    \eta_i\eta_j\eta_k\}$ is in general position. Otherwise $\eta_i\eta_j\eta_k$ would belong to the plane $V_iV_jV_k$ and there would not be a triangular region in $V_iV_jV_k$ disjoint from all planes $\eta_i$, contradicting the fact that $\cyclic{\F}_\sigma$ is a triangle of flags.
    
    Since $\{[\basis{1}],[\basis{2}],[\basis{1}+\basis{2}+\basis{3}],[\basis{4}]\}$ is also in general position, and $\PGL(4)$ acts transitively on quadruples of points in general position, we can change $\F$ in its class so that:
    $$
    V_i = [\basis{1}], \quad V_j = [\basis{2}], \quad V_k = [\basis{1}+\basis{2}+\basis{3}], \quad \eta_i\eta_j\eta_k = [\basis{4}].
    $$
    By non-degeneracy, the line $\eta_i
    \eta_j$ intersects the plane $V_iV_jV_k$ in a single point $P$ that is disjoint from the line $V_iV_j$. It follows that $\{V_i,V_j,V_k,P\}$ is a projective basis for the projective plane $V_iV_jV_k$, and hence there is an element of $\PGL(4)$ that fixes pointwise $\{V_i, V_j, V_k,\eta_i\eta_j\eta_k\}$ and maps $P$ to $[\basis{3}]$. Any such element maps $\eta_i\eta_j$ to the line $[\basis{3}][\basis{4}]$.  In this setting, $\eta_i = [\dbasis{2}]$ and $\eta_j = [\dbasis{1}]$ are forced, which proves \eqref{eq:sp1}.
    
    Since $\eta_k$ is a plane through $V_k = [\basis{1}+\basis{2}+\basis{3}]$ and $\eta_i\eta_j\eta_k = [\basis{4}]$, it is of the form $
    \eta_k = [t \cdot \dbasis{1} + \dbasis{2} - (t+1)\cdot \dbasis{3}]$. A simple computation shows that
    $$
    t_\sigma = \frac{1\cdot 1 \cdot t }{1 \cdot 1 \cdot 1} = t,
    $$
    \noindent which proves \eqref{eq:sp2}. Next we notice that, by non-degeneracy, $\eta_i(V_l) \not= 0$ and $\eta_l(V_j) \not= 0$, thus we may normalize so that
    $$
    V_l = [A \cdot \basis{1} + \basis{2} + C \cdot \basis{3} + D \cdot \basis{4}], \qquad \text{ and } \qquad \eta_l = [a \cdot \dbasis{1} + \dbasis{2} + c \cdot \dbasis{3} + d \cdot \basis{4}^*].
    $$
    It follows from the definition of these edge ratios and the internal consistency equations~\eqref{eq:fe1} and~\eqref{eq:fe2} that
    $$
    e_{\sigma} = A,\qquad e_{\sigma_+} = \frac{1}{t_{\sigma}A + 1 - (t_{\sigma}+1)C}, \quad  \Longrightarrow \quad C = \frac{\mu_\sigma}{e_{\sigma_-}e_{\sigma_+}(t_\sigma + 1)} = X_\sigma,
    $$
    and
    $$
    e_{\overline{\sigma}_-} = a+1+c,\qquad e_{\op{\sigma}} = \frac{t_{\sigma}}{a},  \quad  \Longrightarrow \quad a = e_{\overline{\sigma}_-}e_{\overline{\sigma}_+}, \quad c = -\mu_{\overline{\sigma}}.
    $$
    Imposing that $\eta_l(V_l) = 0$, gives the additional equation
    $$
    d \cdot D = -a \cdot A -1 - c \cdot C = -\frac{1}{t_{\overline{\sigma}}} - 1 - \mu_{\overline{\sigma}} X_\sigma = -\mu_{\overline{\sigma}}(Y_{\overline{\sigma}} - X_\sigma).
    $$
    But $D \not=0$ because the points $\{V_m\}_{m=1}^4$ are in general position, thus we can rewrite
    $$
    d = \frac{-\mu_{\overline{\sigma}}}{D}(Y_{\overline{\sigma}} - X_\sigma).
    $$
    To conclude, we claim that we can change $\F$ in its $\PGL(4)$--class so that $D = -1$, while everything else stays fixed. This can be done by applying the  projective transformation
    $$
    G = 
    \begin{bmatrix}
        1 & 0 & 0 & 0 \\
        0 & 1 & 0 & 0 \\
        0 & 0 & 1 & 0 \\
        0 & 0 & 0 & -\frac{1}{D} \\
    \end{bmatrix}.
    $$
    We remark that uniqueness follows from the fact that the stabilizer of a tetrahedron of flags is trivial (cf. Lemma~\ref{lem:tet_stab_is_trivial}).
\end{proof}


\subsection{The parametrization of \texorpdfstring{$\ctetFL$}{TEXT}}\label{subsec:parametrization_tet_flags}

Let $\RR^{\edgeface}_\times$ be the set of functions from $\edgeface$ to $\RR^\times$. By combining the maps $\3$ and $\xi$ from \S\ref{subsec:triple_edge_ratios}, we have a well defined continuous map
$$
\Psi :  \ctetFL \rightarrow \RR^{\edgeface}_\times \times \RR^{\edgeface}_\times, \qquad \text{where} \qquad [\F] \mapsto \left( \3([\F],\ast), \xi([\F],\ast)\right) = \left( t^\F_\ast, e^\F_\ast \right).
$$
It follows from Lemma~\ref{lem:internal_eqs} that the image of $\Psi$ is contained in the algebraic variety defined by the internal consistency equations. The goal of this section is to determine the image of $\Psi$, and show that it is a homeomorphism onto its image. The first step is to understand the edge ratios in geometric terms, through cross ratios.

\begin{lemma}\label{lem:edge_to_cross}
	Let $\F = (V_m,\eta_m)_{m=1}^4\in \ordFL^4$. For every $\sigma=(ij)k \in \edgeface$, let $\ell_{ij}$ be the line $V_iV_j$, and let $\ell^*_{ij}$ be the pencil of planes containing $\eta_i$ and $\eta_j$. Let $P_{k}=\ell_{ij}\cap \eta_k$ and $P_l=\ell_{ij}\cap \eta_l$. Let $p^*_{k}$ be the plane in $\ell_{ij}^*$ that contains $V_k$ and let $p_l^*$ the plane in $\ell_{ij}^*$ that contains $V_l$. Then
	$$
	e_{ij}^\F=e_{\sigma}^\F=\cross{\eta_i,
	\eta_j,p_k^*,p_l^*}, \qquad \text{and} \qquad e_{kl}^\F=e_{\op{\sigma}}^\F=\cross{V_i,V_j,P_k,P_l}.
	$$
\end{lemma}

\begin{proof}
	To simplify the notation we are going to drop the superscripts.	We consider the function 
	$$
	\ell^*_{ij} \rightarrow \rp^1, \qquad \text{that maps} \qquad \eta \mapsto \frac{\overline{\eta}(\overline{V_k})\overline{\eta_j}(\overline{V_l})}{\overline{\eta}(\overline{V_l})\overline{\eta_j}(\overline{V_k})},
	$$
	for representatives $\overline{\eta},\overline{\eta_j},\overline{V_l},\overline{V_k}$ of $\eta,\eta_j,V_l,V_k$, respectively. As in the definition of the edge ratios, this function is well defined by non-degeneracy of the flags and it does not depend on the choice of representatives.
	
	It is easy to check that this is the unique projective map that takes $(p_l^*,p_k^*,\eta_j,\eta_i)$ to $(\infty,0,1,e_{ij})$. But since $\ell^*_{ij} \cong \rp^1$, by definition of cross ratio,
	$$
	e_{ij} = [p_l^*,p_k^*,\eta_j,\eta_i] = [\eta_i,\eta_j,p_k^*,p_l^*].
	$$
	A dual (but analogous) argument applies to the function
	$$
	\ell_{ij} \rightarrow \rp^1, \qquad \text{that maps} \qquad V \mapsto \frac{\overline{\eta_k}(\overline{V_i})\overline{\eta_l}(\overline{V})}{\overline{\eta_k}(\overline{V})\overline{\eta_l}(\overline{V_i})},
	$$
	to show that $e_{kl}=\cross{P_k,P_l,V_i,V_j}=\cross{V_i,V_j,P_k,P_l}$.
\end{proof}

Combining Lemma~\ref{lem:edge_to_cross} with Lemma~\ref{lem:cross_ratio_positivity}, we can show that the edge ratios of a tetrahedron of flags are always positive. Recall that this was already proven for the triple ratios in Lemma~\ref{lem:triple_ratio}. For convenience, we combine them both in the following result.

\begin{lemma}\label{lem:positive_coords}
	For every $\F \in \tetFL$ and every $\sigma \in \edgeface$,
	$$
	t_{\sigma}^\F >0,\qquad \text{and} \qquad
	e_{\sigma}^\F  > 0.
	$$
\end{lemma}

\begin{proof}
	Let $\sigma \in \edgeface$ and $\F \in \tetFL$. Since edge ratios and triple ratios are invariant under projective transformations, we can change $\F$ to be in $\sigma$--standard position (cf. Lemma~\ref{lem:std_pos_tetrahedron}).
	
	By Lemma~\ref{lem:tet_to_triangles}, the face $\cyclic{\F}_\sigma$ is a triangle of flags hence $t_{\sigma}^\F > 0$ (cf. Lemma~\ref{lem:triple_ratio}).
	
	Now we consider the affine patch $\A := \rp^3 \setminus [\dbasis{1} + \dbasis{2} - \dbasis{3}]$. The three planes $\eta_i,\eta_j,\eta_k$ have a unique common intersection $\eta_i\eta_j\eta_k = [\basis{4}] \in [\dbasis{1} + \dbasis{2} - \dbasis{3}]$, therefore they form an infinite triangular prism $\Prism$ in $\A$. The prism $\Prism$ contains the triangle associated to the face $\cyclic{\F}_\sigma$, thus it must contain the entire tetrahedron associated to the tetrahedron of flags $\F$.
	
	Let $p_k^*$ and $p_l^*$ be the planes in the pencil $\ell^*_{ij}$ through the line $\eta_i\eta_j$, containing the points $V_k$ and $V_l$, respectively. 
	The prism $\Prism$ is contained in one of the four connected components of $\A\setminus \{\eta_i,\eta_j\}$. The points $V_k$ and $V_l$ are contained in $\Prism$ and so the planes $p_k^*,p_l^*$ intersect $\Prism$. It follows that $p_k^*$ and $p_l^*$ are contained in the same region, namely they are in the same connected component of $\ell^*_{ij} \setminus \{\eta_i,\eta_j\}$. Combining Lemma~\ref{lem:edge_to_cross} and Lemma~\ref{lem:cross_ratio_positivity}, we conclude that
	{
		$$
	e_{\sigma}^\F = \cross{\eta_i,\eta_j,p_k^*,p_l^*} > 0.
	$$
}
\end{proof}

Let $\defspace_\tet$ be the algebraic variety of $\RR^{\edgeface}_\times \times \RR^{\edgeface}_\times$ defined by the internal consistency equations \eqref{eq:triple_rat_eq}--\eqref{eq:fe2}. The \emph{deformation space of a tetrahedron of flags} is the semi-algebraic set
$$
\tetsubvariety := \defspace_\tet \cap \left(\RR^{\edgeface}_{>0} \times \RR^{\edgeface}_{>0}\right).
$$
An immediate corollary of Lemma~\ref{lem:positive_coords} is that $\Psi$ has image in $\tetsubvariety$. Furthermore, it is easy to check that $\tetsubvariety$ is homeomorphic to $\RR^5_{>0}$. Indeed one can use equations \eqref{eq:fe1}--\eqref{eq:fe2} to write all triple ratios as edge ratios, reducing the number of variables from $24$ to $12$. Equation $\eqref{eq:edge_rat_eq}$ implies that half of the edge ratios are redundant, and by $\eqref{eq:fe3}$ only $5$ are necessary. Hence there is a homeomorphism $\tetsubvariety \rightarrow \RR^5_{>0}$ that maps
$$
\left( t^\F_\star, e^\F_\star \right) \mapsto \left(e^\F_{12},e^\F_{13},e^\F_{23},e^\F_{14},e^\F_{24}\right).
$$

\begin{theorem}\label{thm:parametrization_flag_space}
The map $\Psi : \ctetFL \rightarrow \tetsubvariety \cong \RR^5_{>0}$ is a homeomorphism.
\end{theorem}

\begin{proof}
The map $\Psi$ is continuous by continuity of $\3$ and $\xi$, hence it will be enough to construct a continuous inverse.

We define a map $\Phi: \tetsubvariety  \rightarrow \ctetFL$ as follows.  For each $(\tau_\star,\epsilon_\star) \in \tetsubvariety\subset \RR_{>0}^{\edgeface}\times \RR_{>0}^{\edgeface}$, let $[\F] := \Phi(\tau_\star,\epsilon_\star)$ be the $\PGL(4)$--class of the tetrahedron of flags $\F = (V_m,\eta_m)_{m=1}^4$ defined by
    $$
    (V_1,\eta_1) = ([\basis{1}],[\dbasis{2}]), \qquad (V_2,\eta_2) = ([\basis{2}],[\dbasis{1}]),
    $$
    $$
    (V_3,\eta_3) = ([\basis{1}+\basis{2}+\basis{3}],[\tau_{(12)3} \cdot \dbasis{1} +\dbasis{2} - (\tau_{(12)3}+1) \cdot \dbasis{3}],
    $$
    $$
    V_4 = [\epsilon_{(12)3} \cdot \basis{1} + \basis{2} + \X \cdot \basis{3} - \basis{4}],
    $$
    $$
    \eta_4 = [\epsilon_{(14)2}\epsilon_{(42)1} \cdot \dbasis{1} + \dbasis{2} - \overline{\nu} \cdot \dbasis{3} +  \overline{\nu} \left( \Y - \X \right) \cdot \basis{4}^*],
    $$
    where
    $$
    \X = \frac{\nu \tau_{(12)3} \epsilon_{(12)3}}{\tau_{(12)3}+1}, \quad \Y = \frac{\tau_{(21)4}+1}{\tau_{(21)4}\overline{\nu}},\quad \nu = \epsilon_{(23)1}\epsilon_{(31)2} - \epsilon_{(23)1} + 1, \quad \text{ and } \quad \overline{\nu} = \epsilon_{(14)2}\epsilon_{(42)1} - \epsilon_{(14)2} + 1.
    $$
Since $\tau_{(12)3} > 0$, the cyclically ordered triple of flags $\cyclic{\F}_{(12)3} = \cyclic{V_m,\eta_m}_{m=1}^3$ is a triangle of flags (cf. Lemma~\ref{lem:triple_ratio}). Let $\triangle \subset \rp^3$ be the unique triangle associated to the triangle of flags $\cyclic{\F}_(12)3$.

	As in the proof of Lemma~\ref{lem:positive_coords}, we consider the affine patch $\A := \rp^3 \setminus [\dbasis{1} + \dbasis{2} - \dbasis{3}]$. The three planes $\eta_1,\eta_2,\eta_3$ have a unique common intersection $\eta_1\eta_2\eta_3 = [\basis{4}] \in [\dbasis{1} + \dbasis{2} - \dbasis{3}]$, therefore they form an infinite triangular prism $\Prism$ in $\A$. In particular, the prism $\Prism$ contains the triangle $\triangle$. We want to show that $V_4$ lies inside $\Prism$, and that $\eta_4$ does not intersect the tetrahedron $\tet$ spanned by $\{V_1,V_2,V_3,V_4\}$ in $\Prism$.
	
	First we observe that, by positivity of $(\tau_\star,\epsilon_\star)$, the quadruple of flags $\F$ is non-degenerate:
	\begin{alignat*}{3}
	\eta_1(V_4) &= 1, \quad &&\eta_2(V_4) = \epsilon_{(12)3}, \quad &&\eta_3(V_4) = \tau_{(12)3}\epsilon_{(12)3} + 1 + \nu \tau_{(12)3}\epsilon_{(12)3} = \tau_{(12)3}\epsilon_{(12)3}\epsilon_{(23)1},\\
	\eta_4(V_1) &= 1, \quad &&\eta_4(V_2) = \epsilon_{(14)2}\epsilon_{(42)1}, \quad &&\eta_4(V_3) = \epsilon_{(14)2}\epsilon_{(42)1} + 1 - \overline{\nu} = \epsilon_{(14)2}.
	\end{alignat*}
	In the third equation we used that $(\tau_\star,\epsilon_\star) \in \tetsubvariety$ and therefore $\tau_{(12)3}\epsilon_{(12)3}\epsilon_{(23)1}\epsilon_{(31)2} =1 $.
	
	For every ordered triple $(i,j,k) \in \cyclic{1,2,3}$, let $\ell^*_{ij}$ be the pencil of planes containing $\eta_i$ and $\eta_j$. Let $p^*_{k,ij}$ be the plane in $\ell_{ij}^*$ that contains $V_k$ and let $p^*_{l,ij}$ the plane in $\ell_{ij}^*$ that contains $V_l$. Then we can apply Lemma~\ref{lem:edge_to_cross} to find that
	$$
	\cross{\eta_i,\eta_j,p^*_{k,ij},p^*_{l,ij}} = \overline{\xi}(\F,(ij)k) = \epsilon_{(ij)k}.
	$$
	Once again, given that $\epsilon_{(ij)k}>0$, it follows that $p^*_{k,ij}$ and $p^*_{l,ij}$ are contained in the same connected component of $\ell^*_{ij} \setminus \{\eta_i,\eta_j\}$ (cf. Lemma~\ref{lem:cross_ratio_positivity}). In particular, this implies that the intersection $V_4 = p^*_{4,12} \cap p^*_{4,23} \cap p^*_{4,31}$ is in $\Prism$.
	
	It is only left to show that $\eta_4$ does not intersect the tetrahedron $\tet$ spanned by $\{V_1,V_2,V_3,V_4\}$ in $\Prism$. The argument is analogous, but dual, to the previous one.
	
	For every ordered triple $(i,j,k) \in \cyclic{1,2,3}$, let $\ell_{ij}$ be the line $V_iV_j$. Let $P_{k,ij}=\ell_{ij}\cap \eta_k$ and $P_{l,ij}=\ell_{ij}\cap \eta_l$. Then we can apply Lemma~\ref{lem:edge_to_cross} to find that  
	$$
	\cross{V_i,V_j,P_{k,ij},P_{l,ij}} = \overline{\xi}(\F,(kl)i) = \epsilon_{(kl)i}.
	$$
	Once again, given that $\epsilon_{(kl)i}>0$, it follows that $P_{k,ij}$ and $P_{l,ij}$ are contained in the same connected component of $\ell_{ij} \setminus \{\eta_i,\eta_j\}$ (cf. Lemma~\ref{lem:cross_ratio_positivity}). We remark that $P_{k,ij} \notin \triangle$, therefore $P_{l,ij} \notin \triangle$. In particular, this implies that the spanned plane $\eta_4 = P_{4,12}P_{4,23}P_{4,31}$ misses $\triangle$. But $\triangle$ is a face of $\tet$, and $V_4 \in \eta_4$, therefore $\eta_4$ misses the entire tetrahedron $\tet$.

    In conclusion, $\F \in \tetFL$ is a tetrahedron of flags and $\Phi: \tetsubvariety  \rightarrow \ctetFL$ is well defined. The above argument also shows that $\Phi$ is continuous. Since $\F$ is in $(12)3$--standard position, it follows from Lemma~\ref{lem:std_pos_tetrahedron} that $\Phi = \Psi^{-1}$.
\end{proof}


\subsection{Orientation}\label{subsec:orientation}

We recall that every tetrahedron of flags $\F \in \tetFL$ corresponds to a unique tetrahedron $\tet_{\F}$ in $\rp^3$, with a canonical ordering of the vertices, namely an identification with the standard tetrahedron $\statet$. Then $\tet_{\F}$ is positively oriented if the identification to the standard tetrahedron is orientation preserving. Otherwise it is negatively oriented. 

We remark that projective transformations are not always orientation preserving, thus a tetrahedron may change orientation under projective transformation. However, when we glue tetrahedra of flags together (cf. \S\ref{sec:coords_pair_teta}), it will be convenient to have some control over their orientations to make sure that they are geometrically glued, namely the underlying projective tetrahedra only intersect along the common face.

An immediate consequence of Lemma~\ref{lem:std_pos_tetrahedron} and Lemma~\ref{lem:positive_coords} is that the tetrahedron associated to a $\sigma$--standard representative is always positively oriented.

\begin{lemma}\label{lem:std_tet_is_positively_ori}
    For each edge-face $\sigma=(ij)k$ and for each class $[\F] \in \ctetFL$ of tetrahedra of flags, let $\F$ be the $\sigma$--standard representative of $[\F]$. Then the tetrahedron $\tet_\F \subset \rp^3$ associated to $\F$ is positively oriented.
\end{lemma}

\begin{proof}
    We recall that $\tet_{\F}$ comes with an identification to the standard tetrahedron $\statet$, that determines its orientation.
    
    Let $\sigma = (ij)k$ be an edge-face. Consider the affine patch $\A := \rp^3 \setminus [\dbasis{1} + \dbasis{2} - \dbasis{3}]$. As $\F$ is in $\sigma$--standard position, the three planes $\eta_i,\eta_j,\eta_k$ have a unique common intersection $\eta_i\eta_j\eta_k = [\basis{4}] \in [\dbasis{1} + \dbasis{2} - \dbasis{3}]$, therefore they form an infinite triangular prism $\Prism$ in $\A$. The prism $\Prism$ contains the entire tetrahedron $\tet_{\F}$. The intersection $\triangle' := \Prism \cap V_iV_jV_k$ is a triangle in the plane $V_iV_jV_k$, containing the triangle $\triangle := \tet_{\F} \cap V_iV_jV_k$ associated to the face $\cyclic{\F}_\sigma$. We consider the following identification:
    \begin{align*}
    \A &\rightarrow \RR^3,\\
    [x : y : z : w]^T &\mapsto \left(\frac{x}{x + y - z}, \frac{y}{x + y - z}, \frac{w}{x + y - z}\right)^T.
    \end{align*}
    In this coordinate system
    \begin{align*}
    \triangle' &\subset \{ (a,b,c) \in \A \ | \ c = 0 \},\\
    \Prism &= \{ (a,b,c) \in \A \ | \ (a,b,0) \in \triangle' \},\\
    V_l &= \left(\frac{e_{\sigma}}{e_{\sigma} + 1 - X_\sigma}, \frac{1}{e_{\sigma} + 1 - X_\sigma}, \frac{-1}{e_{\sigma} + 1 -X_\sigma}\right)^T,
    \end{align*}
    where
    $$
    X_\sigma := \frac{\mu_{\sigma}t_\sigma e_\sigma}{t_{\sigma}+1}, \qquad \text{and} \qquad e_{\sigma} + 1 - X_\sigma = \frac{e_\sigma + t_\sigma + t_\sigma e_\sigma e_{\sigma_-}}{t_\sigma +1} > 0.
    $$
    The triangle $\triangle'$ divides the interior of the prism $\Prism$ into two connected components, and $V_l$ always belongs to the one with negative third coordinate:
    $$
    \Prism^- := \{ (a,b,c) \in \Prism \ | \ c<0 \}.
    $$
    It follows that $\tet_\F$ is positively oriented.
\end{proof}


\section{Coordinates on a pair of glued tetrahedra of flags}\label{sec:coords_pair_teta}

We recall that two tetrahedra of flags are glued together when two of their marked faces are matched (cf. \S\ref{subsec:tris_and_tets_flags}). When they are glued geometrically, the corresponding tetrahedra in $\rp^3$ share a common face, and locally only intersect at that face. In other words, gluing tetrahedra of flags corresponds to gluing the underlying tetrahedra.

In this section we show that two tetrahedra of flags can be glued together if and only if they satisfy a simple \emph{face pairing equation} (cf. Lemma~\ref{lem:face_pair_eq}). This will be done in the language of edge-faces (cf. \S\ref{subsec:edge_faces}). Next, we show that two glueable tetrahedra of flags can be glued in a $1$--real parameter family of ways,  which is encoded by the \emph{gluing parameters} (cf. \S\ref{subsec:gluing_param}). These parameters are positive if and only if the gluing is geometric (cf. Lemma~\ref{lem:geometric_pairs}). We conclude by putting edge parameters and gluing parameters together to parametrize the \emph{deformation space of pairs of glued tetrahedra of flags} (cf. \S\ref{subsec:parametrization_pairs_of_tet}).


\subsection{Edge-face glued position}\label{subsec:gluing_and_face_eqs}

We begin by recalling the notion of glued tetrahedra of flags, and set the notation in the language of edge-faces.

Let $\F=(V_m,\eta_m)_{m=1}^4$ and $\E=(W_m,\zeta_m)_{m=1}^4$ be two tetrahedra of flags, and let $\sigma=(ij)k$ and $\tau=(i'j')k'$ be a pair of edge-faces. We say that $\F$ and $\E$ are \emph{glued along $(\sigma,\tau)$}, or \emph{in $(\sigma,\tau)$--glued position}, if
$$
(V_i,\eta_i) = (W_{j'},\zeta_{j'}), \qquad (V_j,\eta_j) = (W_{i'},\zeta_{i'}), \qquad \text{ and } \qquad
(V_k,\eta_k) = (W_{k'},\zeta_{k'}).
$$
If, in addition, $\F$ is the $\sigma$--standard representative of its $\PGL(4)$--class $[\F]$, then we say that $\F$ and $\E$ are in \emph{$(\sigma,\tau)$--standard glued position}. Note that if $\F$ and $\E$ are glued along $(\sigma,\tau)$, then they are also glued along $(\sigma_+,\tau_-)$ and glued along $(\sigma_-,\tau_+)$. In addition, $\E$ and $\F$ are glued along $(\tau,\sigma)$.

{Given a pair of edge-faces $(\sigma,\tau) \in \edgeface \times \edgeface$}, we denote by $\ordFL_{\sigma,\tau}$ the set of pairs $(\F,\E)$ of tetrahedra of flags that are glued along $(\sigma,\tau)$, and by $\cordFL_{\sigma,\tau}$ the quotient of $\ordFL_{\sigma,\tau}$ by the diagonal action of $\PGL(4)$. We remark that $\cordFL_{\sigma,\tau}$ is naturally homeomorphic to the space of pairs of tetrahedra of flags in $(\sigma,\tau)$--standard glued position.

We say that a pair $[\F],[\E] \in \ctetFL$ of $\PGL(4)$--classes of tetrahedra of flags is \emph{$(\sigma,\tau)$--glueable}, if there are representatives $\F \in [\F]$ and $\E \in [\E]$ that are glued along $(\sigma,\tau)$, namely such that $(\F,\E) \in \ordFL_{\sigma,\tau}$. The next result shows that determining $(\sigma,\tau)$--glueable pairs is simple.

\begin{lemma}\label{lem:face_pair_eq}
   A pair $[\F],[\E] \in \ctetFL$ is $(\sigma,\tau)$--glueable if and only if
    \begin{align}\label{eq:face_equation}
        t^\F_{\sigma}t^\E_{\tau} = 1.
    \end{align}
    Moreover, if the pair $[\F]$, $[\E]$ is $(\sigma,\tau)$--glueable, then it is $(\sigma',\tau')$--glueable for all $\sigma' \in \{\sigma,\sigma_-,\sigma_+\}$ and all $\tau' \in \{\tau,\tau_-,\tau_+\}$.
\end{lemma}
Equation~\eqref{eq:face_equation} is called the \emph{face pairing equation} of $[\F]$ and $[\E]$ with respect to the edge-face pair $(\sigma,\tau)$.
\begin{proof}
    Let $\T \in \cycFL^3$ be a cyclically ordered triple of flags, and let $\overline{\T} \in \cycFL^3$ be obtained by applying an odd permutation to the flags in $\T$. Then, from the definition of triple ratio, we have that
    $$
    \3(\T) \3(\overline{\T})=1.
    $$ 
    The first statement then follows from the fact that projective classes of cyclically ordered triples of flags are uniquely determined by the triple ratio (cf. Lemma~\ref{lem:triple_ratio}). The latter instead is a consequence of the cyclic invariance of the triple ratio.
\end{proof}


\subsection{Gluing parameters}\label{subsec:gluing_param}

In Lemma~\ref{lem:face_pair_eq} we underlined that two projective classes of tetrahedra of flags satisfying the face pairing equation~\eqref{eq:face_equation} can be glued along a face in different combinatorial ways. In this section we show that, even for fixed combinatorics, there are different representatives of the same pair that are glueable. They can be parametrized by the following \emph{gluing parameter} (cf. Lemma~\ref{lem:gluing_parameter}). For all edge-faces $\sigma=(ij)k$ and $\tau=(i'j')k'$, we define the continuous function $\overline{g}_\sigma^\tau : \ordFL_{\sigma,\tau} \rightarrow \RR^\times$ via:
\begin{align}\label{eq:gluing_param_def}
(\F,\E) = \left((V_m,\eta_m)_{m=1}^4,(W_m,\zeta_m)_{m=1}^4\right) \mapsto - \frac{\overline{\eta_i}(\overline{V_l})\overline{\eta_j}(\overline{V_k})\overline{\eta_{ijk}}(\overline{W_{l'}})}{\overline{\eta_i}(\overline{V_k})\overline{\eta_j}(\overline{W_{l'}})\overline{\eta_{ijk}}(\overline{V_l})} \in \RR^\times,
\end{align}
where $\eta_{ijk} = V_iV_jV_k$ is the plane spanned by $\{V_i,V_j,V_k\}$, and $\overline{V}_m \in \RR^4$ and $\overline{\eta}_m \in (\RR^4)^*$ are representatives of $V_m$ and $\eta_m$, respectively. Similarly to triple ratios and edge ratios, we remark that $\overline{g}_\sigma^\tau$ is well defined. It is easy to check that it is independent of the choice of representatives. Furthermore, since $\F$ and $\E$ are glued along $(\sigma,\tau)$, then $\eta_{ijk} = V_iV_jV_k = W_{i'}W_{j'}W_{k'} =: \zeta_{i'j'k'}$.
Therefore numerator and denominator in $\overline{g}_\sigma^\tau$ are non-zero by non-degeneracy of the flags.

The \emph{gluing parameter of $(\F,\E)$ with respect to $(\sigma,\tau)$} is
\begin{equation}\label{eq:gluing_param_def2}
g^{\E,\tau}_{\F,\sigma} := \overline{g}_\sigma^\tau\left(\F,\E \right).
\end{equation}
Finally, we underline that the gluing parameter is invariant under the diagonal action of $\PGL(4)$, namely
$$
\overline{g}_\sigma^\tau\left(\F,\E \right) = \overline{g}_\sigma^\tau\left(G \cdot \F, G \cdot \E \right), \qquad \forall G \in \PGL(4).
$$
Therefore $\overline{g}_\sigma^\tau$ descends to a function
$$
g_\sigma^\tau : \cordFL_{\sigma,\tau} \rightarrow \RR^\times.
$$

\begin{remark}\label{rem:gluing_param_geometric_description}
	We now give a geometric description of the gluing parameter. Suppose that $\sigma=(ij)k$ and $\tau=(i'j')k'$ are edge-faces corresponding to the permutations $[ijkl]$ and $[i'j'k'l']$, respectively. Next, suppose that $\F=(V_m,\eta_m)_{m=1}^4$ and $\E=(W_m,\zeta_m)_{m=1}^4$ are in $(\sigma,\tau)$--standard glued position. If we normalize $W_{l'}$ so that it is of the form $[ \basis{1}+b \cdot \basis{2}+c \cdot \basis{3}+d \cdot \basis{4}]$, then it is easy to check using the definition that $g_\sigma^\tau([\F,\E])=d$. In other words, if we think of the inhomogeneous $[\basis{4}]$ coordinate as a ``height parameter'' then the gluing parameter measures the relative differences of these height parameters of $V_l$ and $W_{l'}$.
\end{remark}

The next lemma shows that, given two projective classes of tetrahedra of flags $[\F]$ and $[\E]$, the gluing parameter parametrizes the set of classes of $(\sigma,\tau)$--glued tetrahedra of flags $[\F,\E]$.

\begin{lemma}\label{lem:gluing_parameter}
    Let $[\F],[\E] \in \ctetFL$ be two $\PGL(4)$--classes of tetrahedra of flags that are $(\sigma,\tau)$--glueable. Let $\G = \G([\F],[\E]) \subset \cordFL_{\sigma,\tau}$ be the subset
    $$
    \G := \{ [\F,\E] \in \cordFL_{\sigma,\tau} \mid \F \in [\F] \ \text{and} \  \E \in [\E] \}.
    $$
    Then the restriction map $\restr{g_\sigma^\tau}{\G} :  \G \rightarrow \RR^\times$ is a homeomorphism.
\end{lemma}
\begin{proof}
    We recall that $\cordFL_{\sigma,\tau}$ is naturally homeomorphic to the space of pairs of tetrahedra of flags in $(\sigma,\tau)$--standard glued position. If $\F_0$ is the $\sigma$--standard representative of $[\F]$, then under this identification we have
    $$
    \G \cong \{ (\F_0,\E) \in \ordFL_{\sigma,\tau} \mid \E \in [\E] \}.
    $$
    Let $\E_0$ be the $\tau$--standard representative of $[\E]$. 
    Every $(\F_0,\E) \in \G$ is of the form $(\F_0,G \cdot \E_0)$ for some unique $G \in \PGL(4)$. Since $(\F_0,\E)$ is in $(\sigma,\tau)$--standard glued position and $\E_0$ is in $\tau$--standard position, then $G$ is a projective transformation such that
    \begin{alignat*}{3}
    [\basis{1}] &\mapsto [\basis{2}], \quad [\basis{2}] &&\mapsto [\basis{1}], \quad  [\basis{1} + \basis{2} + \basis{3}] &&\mapsto [\basis{1} + \basis{2} + \basis{3}],\\
    [\dbasis{1}] &\mapsto [\dbasis{2}], \quad [\dbasis{2}] &&\mapsto [\dbasis{1}], \hspace{2cm} [\basis{4}] &&\mapsto [\basis{4}].
    \end{alignat*}
    It follows that $G$ is of the form
    \begin{equation*}
    G:=\begin{bmatrix}
    0 & 1 & 0 & 0\\
    1 & 0 & 0 & 0\\
    0 & 0 & 1 & 0\\
    0 & 0 & 0 & -g
    \end{bmatrix}, \qquad \text{for} \  g\in \RR^\times.
    \end{equation*}
    Using the geometric description in Remark~\ref{rem:gluing_param_geometric_description}, we see that $g = \overline{g}_\sigma^\tau\left(\F,\E\right) =  g_\sigma^\tau\left([\F,\E]\right)$, concluding the proof.
\end{proof}

We recall that, if $\F$ and $\E$ are glued along $(\sigma,\tau)$, then they are also glued along $(\sigma_+,\tau_-)$ and along $(\sigma_-,\tau_+)$. Similarly, $\E$ and $\F$ are glued along $(\tau,\sigma)$. Hence for all pairs $[\F,\E] \in \cordFL_{\sigma,\tau}$, there are six gluing parameters that one can associate to them. They are the ones corresponding to the six pairs of edge-faces:
$$
(\sigma,\tau), \quad (\sigma_+,\tau_-), \quad (\sigma_-,\tau_+), \quad (\tau,\sigma), \quad (\tau_-,\sigma_+), \quad (\tau_+,\sigma_-).
$$
It follows directly from the definition of the gluing parameter~\eqref{eq:gluing_param_def} that these six parameters are related. Their relations are called the \emph{gluing consistency equations}, and are summarized in the following result.

\begin{lemma}\label{lem:gluing_eqs}
   Let $(\sigma,\tau) \in \edgeface \times \edgeface$ and let $[\F,\E] \in \cordFL_{\sigma,\tau}$. Then
    \begin{alignat}{2}
    & &&g^{\E,\tau}_{\F,\sigma} g^{\F,\sigma}_{\E,\tau}  = 1, \label{eq:gluing_eq_1}\\
    &g^{\E,\tau_-}_{\F,\sigma_+} = \frac{g^{\E,\tau}_{\F,\sigma}}{e_{\sigma_+}^{\F}e_{\tau}^{\E}}, &&\qquad \text{ and } \qquad g^{\E,\tau_+}_{\F,\sigma_-}=g^{\E,\tau}_{\F,\sigma}e_{\sigma}^{\F}e_{\tau_+}^{\E}. \label{eq:gluing_eq_2}
\end{alignat}
\end{lemma}

\begin{proof}
These identities can be easily verified using the definition \eqref{eq:gluing_param_def} of the gluing parameter once the pair $[\F,\E]$ has been put into $(\sigma,\tau)$--standard glued position. 
\end{proof}


\subsection{Geometric gluing}\label{subsec:geometric_gluing}

Let $(\F,\E) \in \ordFL_{\sigma,\tau}$ be a pair of tetrahedra of flags glued along $(\sigma,\tau)$. Let $\tet_{\F},\tet_{\E} \subset \rp^3$ be the tetrahedra in $\rp^3$ corresponding to $\F$ and $\E$, respectively. Since the pair $(\F,\E)$ is in $(\sigma,\tau)$--glued position, the tetrahedra $\tet_{\E},\tet_{\F}$ share a face, say $f$. In general $f \subset \tet_{\E} \cap \tet_{\F}$ and we recall that the pair $(\F,\E)$ is \emph{geometric} if $f = \tet_{\E} \cap \tet_{\F}$. 

The pair $(\F,\E)$ is not always geometric, as both tetrahedra $\tet_{\F},\tet_{\E}$ might be ``on the same side'' of $f$. On the other hand, the geometricity condition is invariant under projective transformations, thus a class of pairs $[\F,\E] \in \cordFL_{\sigma,\tau}$ is \emph{geometric} if one representative pair $(\F,\E)$ (and hence all) is geometric. The next result shows that the sign of the gluing parameter determines if a glued pair is geometric (cf. \S\ref{subsec:gluing_param}).

\begin{lemma}\label{lem:geometric_pairs}
    A class of glued pairs $[\F,\E] \in \cordFL_{\sigma,\tau}$ is geometric if and only if
    \begin{equation}
    g^{\E,\tau}_{\F,\sigma} > 0.
    \end{equation}
\end{lemma}
\begin{proof}
    Let $(\F_0,\E) \in [\F,\E]$ be the unique representative pair in $(\sigma,\tau)$--standard glued position, and let $\E_0$ be the $\tau$--standard representative of $[\E]$. From the proof of Lemma \ref{lem:gluing_parameter}, $(\F_0,\E) = (\F_0,G\cdot \E_0)$ for
    \begin{equation*}
    G:=\begin{bmatrix}
    0 & 1 & 0 & 0\\
    1 & 0 & 0 & 0\\
    0 & 0 & 1 & 0\\
    0 & 0 & 0 & -g^{\E,\tau}_{\F,\sigma}
    \end{bmatrix}.
    \end{equation*}
    Let $\tet_{\F_0},\tet_{\E_0} \subset \rp^3$ be the projective tetrahedra associated to $\F_0$ and $\E_0$, respectively. Then both $\tet_{\F_0}$ and $\tet_{\E_0}$ are positively oriented by Lemma~\ref{lem:std_tet_is_positively_ori}. Hence $(\F_0,\E)$ is geometric if and only if $G$ is orientation preserving, that is if and only if $g^{\E,\tau}_{\F,\sigma} > 0$.
\end{proof}

We denote by $\ordFL^+_{\sigma,\tau} \subset \ordFL_{\sigma,\tau}$ the subset of geometric pairs glued along $(\sigma,\tau)$, and by $\cordFL^+_{\sigma,\tau} \subset \cordFL_{\sigma,\tau}$ the corresponding quotient by the diagonal action of $\PGL(4)$. The following result is an immediate consequence of Lemma~\ref{lem:gluing_parameter} and Lemma~\ref{lem:geometric_pairs}.

\begin{corollary}\label{cor:positive_gluing_parameter}
    Let $[\F],[\E] \in \ctetFL$ be two $\PGL(4)$--classes of tetrahedra of flags that are $(\sigma,\tau)$--glueable. Let $\G^+ = \G^+([\F],[\E]) \subset \cordFL^+_{\sigma,\tau}$ be the subset
    $$
    \G^+ := \{ [\F,\E] \in \cordFL^+_{\sigma,\tau} \mid \F \in [\F] \ \text{and} \  \E \in [\E] \}.
    $$
    Then the restriction map $\restr{g_\sigma^\tau}{\G^+} :  \G^+ \rightarrow \RR_{>0}$ is a homeomorphism.
\end{corollary}


\subsection{The parametrizations of \texorpdfstring{$\cordFL_{\sigma,\tau}$}{TEXT} and \texorpdfstring{$\cordFL^+_{\sigma,\tau}$}{TEXT}
} \label{subsec:parametrization_pairs_of_tet}

In this section we combine the homeomorphism $\Psi : \ctetFL \rightarrow \tetsubvariety$ (cf. \S\ref{subsec:parametrization_tet_flags}) and the map $g_\sigma^\tau : \cordFL_{\sigma,\tau} \rightarrow \RR^\times$ (cf. \S\ref{subsec:gluing_param}) to parametrize $\cordFL_{\sigma,\tau}$ and $\cordFL^+_{\sigma,\tau}$.

Let $(\sigma,\tau) \in \edgeface \times \edgeface$ be a fixed pair of edge-faces, and let
$$
    \mathcal{S} := \{ (\sigma,\tau), (\sigma_+,\tau_-), (\sigma_-,\tau_+), (\tau,\sigma), (\tau_-,\sigma_+), (\tau_+,\sigma_-) \}.
$$
Recall that $\RR^{\mathcal{S}}_\times$ is the set of functions from $\mathcal{S}$ to $\RR^\times$. Then we define the map
$$
   \Psist : \cordFL_{\sigma,\tau}\rightarrow \tetsubvariety\times \tetsubvariety\times \RR^{\mathcal{S}}_\times, \qquad \text{such that} \qquad [(\F,\E)] \mapsto \left(\Psi(\F),\Psi(\E),g^{\E,\ast}_{\F,\ast} \right),
$$
where $g^{\E,\ast}_{\F,\ast} \in \RR^{\mathcal{S}}_\times$ is the function
$$
    g^{\E,\ast}_{\F,\ast}(x,y) = g^{\E,y}_{\F,x}, \qquad \text{for all} \quad  (x,y) \in \mathcal{S}.
$$
{The \emph{deformation space of pairs of $(\sigma,\tau)$--glued tetrahedra of flags}} $\mathcal{D}_{\sigma,\tau}$ is the semi-algebraic subset of $\tetsubvariety\times \tetsubvariety\times \RR^{\mathcal{S}}_\times$
such that:
\begin{enumerate}
    \item the first two coordinates satisfy the face pairing equation \eqref{eq:face_equation} from Lemma~\ref{lem:face_pair_eq};
    \item all coordinates satisfy the gluing consistency equations \eqref{eq:gluing_eq_1} and \eqref{eq:gluing_eq_2} from Lemma~\ref{lem:gluing_eqs}.
\end{enumerate}
{The \emph{deformation space of geometric pairs of $(\sigma,\tau)$--glued tetrahedra of flags}} $\mathcal{D}^+_{\sigma,\tau}$ is the semi-algebraic set
$$
\mathcal{D}^+_{\sigma,\tau} := \{ (p_1,p_2,p_3) \in \mathcal{D}_{\sigma,\tau} \mid p_3(s)>0, \ \forall s\in \mathcal{S}\}.
$$
It is a consequence of Lemma~\ref{lem:face_pair_eq} and Lemma~\ref{lem:gluing_eqs} that $\Psist$ has image in $\mathcal{D}_{\sigma,\tau}$. Similarly, it follows from Corollary~\ref{cor:positive_gluing_parameter} that the restriction map
$\Psist^+ := \restr{\Psist}{\cordFL^+_{\sigma,\tau}}$ has image in $\mathcal{D}^+_{\sigma,\tau}$, namely
$$
\Psist^+ := \restr{\Psist}{\cordFL^+_{\sigma,\tau}} : \cordFL^+_{\sigma,\tau} \rightarrow \mathcal{D}^+_{\sigma,\tau}.
$$

Furthermore, it is easy to check that $\mathcal{D}_{\sigma,\tau}$ (resp. $\mathcal{D}^+_{\sigma,\tau}$) is homeomorphic to $\RR^{9}_{>0} \times \RR^\times$ (resp. $\RR^{10}_{>0}$). Indeed $\tetsubvariety\times \tetsubvariety \cong \RR^{10}_{>0}$, and the face pairing equation \eqref{eq:face_equation} eliminates 1 degree of freedom. On the other hand, one can use equations \eqref{eq:gluing_eq_1}--\eqref{eq:gluing_eq_2} to write all gluing parameters in terms of a single one. Hence there is a homeomorphism
$$
\RR^{\mathcal{S}}_\times \rightarrow \RR^\times, \qquad \text{given by} \qquad	g^{\E,\star}_{\F,\star} \mapsto g^{\E,\tau}_{\F,\sigma},
$$
which induces identifications $\mathcal{D}_{\sigma,\tau} \cong \RR^{9}_{>0} \times \RR^\times$ and $\mathcal{D}^+_{\sigma,\tau} \cong \RR^{10}_{>0}$.

\begin{theorem}\label{thm:parametrization_gluing_space}
The maps
$$
\Psist : \cordFL_{\sigma,\tau}\rightarrow \mathcal{D}_{\sigma,\tau}
\qquad \text{and} \qquad \Psist^+ : \cordFL^+_{\sigma,\tau}\rightarrow \mathcal{D}^+_{\sigma,\tau}
$$
are homeomorphisms.
\end{theorem}

\begin{proof}
    The proof is a straightforward application of Theorem~\ref{thm:parametrization_flag_space}, Lemma~\ref{lem:gluing_parameter} and Corollary~\ref{cor:positive_gluing_parameter}. Here is a sketch. The maps $\Psist$ and $\Psist^+$ are continuous by the continuity of $\Psi$ and $g_\sigma^\tau$. Hence it will be enough to find continuous inverses.  For every point $p = (p_1,p_2,p_3) \in \mathcal{D}_{\sigma,\tau}$, let $[\F] = \Psi^{-1}(p_1)$ and $[\E] = \Psi^{-1}(p_2)$. The $\sigma$--standard representative $\F_0 \in [\F]$ and the $\tau$--standard representative $\E_0 \in [\E]$ are $(\sigma,\tau)$--glueable because $p_1$ and $p_2$ satisfy the face pairing equation \eqref{eq:face_equation} (cf. Lemma~\ref{lem:face_pair_eq}). Hence by Lemma~\ref{lem:gluing_parameter} there is a unique representative $\E' \in [\E]$ such that
    $$
    g_{\sigma}^{\tau}\left( [\F_0, \E' ] \right) = p_3(\sigma,\tau).
    $$
    Thus we have defined a map $\mathcal{D}_{\sigma,\tau} \rightarrow \cordFL_{\sigma,\tau}$ via $p \mapsto [\F_0, \E' ]$. This map is clearly continuous and inverse of $\Psist$. The same map can be modified to be the inverse of $\Psist^+$ via Corollary~\ref{cor:positive_gluing_parameter}.
\end{proof}


\section{Coordinates on a triangulation of flags}\label{sec:coords_triangulation}

Let $\Delta$ be a triangulation of a $3$--manifold $M$, and let $\wDelta$ be a lift of $\Delta$ to the universal cover $\widetilde{M}$. We recall that a triangulations of flags of $M$ (with respect to $\Delta$) is a pair $(\Phi,\rho)$ where 
$$
\Phi : \Vertex(\wDelta) \rightarrow \FL, \qquad \text{and} \qquad \rho : \pi_1(M) \rightarrow \PGL(4),
$$
satisfying some compatibility conditions (cf.\ \eqref{item:def_tri_of_flags_1},\eqref{item:def_tri_of_flags_2} and \eqref{item:def_tri_of_flags_3} in \S\ref{subsec:tri_of_flags}). In what follows, we will abuse notation by letting $\Phi(\tilde \tet)$ denote the tetrahedron of flags determined by the $\Phi$--images of the vertices of $\tilde \tet$. We showed in Theorem~\ref{thm:geom_flag_struc_are_developable} that every triangulation of flags can be extended to a (possibly branched) real projective structure on $M$, thus we are interested in parametrizing the space of $\PGL(4)$--classes of triangulation of flags $\ctriFL$. We begin by defining the parametrization.

Let $(\tetsubvariety \times \RR_{>0}^{\edgeface})^{\Tet(\wDelta)}$ be the set of functions from $\Tet(\wDelta)$ to $\tetsubvariety \times \RR_{>0}^{\edgeface}$, and consider the map $\Psi_\Delta : \cTriFL \rightarrow (\tetsubvariety \times \RR_{>0}^{\edgeface})^{\Tet(\wDelta)}$ defined as follows. For all $p=[\Phi,\rho]\in \cTriFL$ and $\wt \tet\in \Tet(\wDelta)$, we let
\begin{align}\label{eq:PsiDelta_def}
\Psi_\Delta(p)(\wt \tet):=\left(\Psi([\Phi(\wt\tet)]),v_{\wt\tet}^p \right),
\end{align}
where $\Psi$ is the parametrization from Theorem \ref{thm:parametrization_flag_space} and $v_{\wt\tet}^p:\edgeface\to \RR_{>0}$ is  defined as follows. For each $\sigma\in \edgeface$, let $\wt \tet'$ and $\tau \in \edgeface$ be the unique tetrahedron and the unique edge-face such that $\Phi(\wt\tet)$ is glued to $\Phi(\wt\tet')$ along the edge-face pair $(\sigma,\tau)$. Then we define
\begin{align}\label{eq:v_def}
v_{\wt \tet}^p(\sigma):=g^{\Phi(\wt\tet),\sigma}_{\Phi(\wt\tet'),\tau}.
\end{align}
It is easy to check that $\Psi([\Phi(\wt\tet)])$ does not depend on the choice of representative $\Phi$, thus it is well defined by property~\eqref{item:def_tri_of_flags_1} of a triangulation of flags (cf. \S\ref{subsec:tri_of_flags}), and belongs to $\tetsubvariety$ by Theorem~\ref{thm:parametrization_flag_space}. Similarly, the function $v_{\wt\tet}^p$ does not depend on a representative pair of $p$, hence it is well defined by property~\eqref{item:def_tri_of_flags_2} of a triangulation of flags. It follows that $\Psi_\Delta$ is well defined. Roughly speaking, the two factors of $\Psi_\Delta(p)(\wt \tet)$ encode the edge parameters and the gluing parameters of the tetrahedron of flags $\Phi(\wt\tet)$ corresponding to $\wt \tet$. 

Finally, we recall that $\Phi$ is $\rho$--equivariant (property~\eqref{item:def_tri_of_flags_3} of a triangulation of flags) and all of the parameters are invariant under projective transformations. Then $\Psi_\Delta(p)(\tet) = \Psi_\Delta(p)(\gamma \cdot \tet)$ for all $\gamma \in \pi_1(M)$, and $\Psi_\Delta$ descends to a well defined function
$$
\Psi_\Delta : \cTriFL \rightarrow (\tetsubvariety \times \RR_{>0}^{\edgeface})^{\Tet(\Delta)}.
$$
To avoid introducing new notation, we make the abuse of using the symbol $\Psi_\Delta$ for both maps.

The goal of this section is to determine the image of $\Psi_\Delta$, and to show that $\Psi_\Delta$ is a homeomorphism onto its image. It is a consequence of Theorem~\ref{thm:parametrization_gluing_space} that we can make the following immediate restriction on the image of $\Psi_\Delta$. Suppose $\tet$ and $\tet'$ are two tetrahedra of $\Delta$ glued along some face $f$, and let $\wt \tet$ and $\wt \tet'$ be two lifts to $\wDelta$ glued along the corresponding lift of $f$. Then there are edge-faces $\sigma$ and $\tau$ such that $\Phi(\wt\tet)$ is geometrically glued to $\Phi(\wt\tet')$ along the edge-face pair $(\sigma,\tau)$. In other words, $[\Phi(\wt\tet),\Phi(\wt\tet')] \in \cordFL^+_{\sigma,\tau}$. Then by Theorem~\ref{thm:parametrization_gluing_space}, the point
$$
\left(\Psi([\Phi(\tet)]),\Psi([\Phi(\tet')]), \left( v_{\tet}^p(\sigma), v_{\tet}^p(\sigma_+),v_{\tet}^p(\sigma_-),v_{\tet'}^p(\tau),v_{\tet'}^p(\tau_+),v_{\tet'}^p(\tau_-)\right) \right) \in \mathcal{D}^+_{\sigma,\tau}.
$$
More precisely:
\begin{enumerate}
    \item \label{eq:pair_cond_1} the pair $\Psi([\Phi(\tet)]),\Psi([\Phi(\tet')])$ satisfies the face pairing equation \eqref{eq:face_equation} from Lemma~\ref{lem:face_pair_eq};
    \item \label{eq:pair_cond_2} each pair $\left(v_{\wt\tet}^p(\sigma),v_{\wt\tet'}^p(\tau) \right),\left(v_{\wt\tet}^p(\sigma_-),v_{\wt\tet'}^p(\tau_+) \right),\left(v_{\wt\tet}^p(\sigma_+),v_{\wt\tet'}^p(\tau_-) \right)$ satisfies the gluing consistency equations \eqref{eq:gluing_eq_1} from Lemma~\ref{lem:gluing_eqs};
    \item \label{eq:pair_cond_3} all coordinates together satisfy the gluing consistency equations \eqref{eq:gluing_eq_2} from Lemma~\ref{lem:gluing_eqs}.
\end{enumerate}
We denote by $\mathcal{D}'_\Delta$ be the subset of $(\tetsubvariety \times \RR_{>0}^{\edgeface})^{\Tet(\Delta)}$ satisfying the above conditions \eqref{eq:pair_cond_1},\eqref{eq:pair_cond_2} and \eqref{eq:pair_cond_3}, for every pair of glued tetrahedra. The above discussion shows that we have a well defined restriction
$$
\Psi_\Delta : \cTriFL \rightarrow \mathcal{D}'_\Delta.
$$
We recall that $\mathcal{D}^+_{\sigma,\tau} \cong \RR^{10}_{>0}$ (cf. \S\ref{subsec:parametrization_pairs_of_tet}), hence $\mathcal{D}'_\Delta \cong \RR^{5|\Tet(\Delta)|}_{>0}$. Roughly speaking, every time we glue two tetrahedra of flags we gain one gluing parameter, but also lose one edge parameter due to the face pairing equation.\\

Next, we are going to define the \emph{deformation space of a triangulation of flags} $\mathcal{D}_\Delta = \mathcal{D}_\rp(M;\Delta)$. We will soon see that $\mathcal{D}_\Delta$ is a semi-algebraic subset of $\mathcal{D}'_\Delta$. To determine the additional equations that cut out $\mathcal{D}_\Delta$, we will introduce the \emph{monodromy complex} $\complex_\Delta$ associated to $\Delta$ (cf. \S\ref{subsec:monodromy_complex}) and define \emph{cochains and cocycles} on $\complex_\Delta$ (cf. \S\ref{subsec:cocycles}). We then show that a point $x \in \mathcal{D}'_\Delta$ determines a cocycle if and only if $x$ satisfies the \emph{edge gluing equations} (cf. Theorem~\ref{thm:deformation_space_to_cocycles}), and define $\mathcal{D}_\Delta$ as the subset of $\mathcal{D}'_\Delta$ where those equations are satisfied. Finally, in \S\ref{subsec:parametrization_tri_of_flags}, we will prove that $\Psi_\Delta$ is a homeomorphism between $\cTriFL$ and $\mathcal{D}_\Delta$ (cf. Theorem~\ref{thm:parametrization_def_space}).


\subsection{The monodromy complex}\label{subsec:monodromy_complex}

Let $\Delta$ be an ideal triangulation of a $3$--manifold $M$. The \emph{monodromy graph} associated to $\Delta$ is the graph $\graph_\Delta$ defined as follows. The vertices of $\graph_\Delta$ are all pairs $(\tet,\sigma)$ where $\tet \in \Tet(\Delta)$ is a tetrahedron and $\sigma \in \edgeface$ is an edge-face, for a total of $12 \cdot \abs{\Tet(\Delta)}$ vertices. Two vertices $(\tet,\sigma)$ and $(\tet',\sigma')$ of $\graph_\Delta$ are connected by an edge if and only if one of the following (mutually exclusive) conditions is satisfied.
\begin{alignat*}{3}
    (1) & (\emph{Red edge}) && \quad \text{ if } \tet = \tet' \text{ and } \sigma' = \sigma_+ \text{ or } \sigma' = \sigma_-.\\
    (2) & (\emph{Blue edge}) && \quad \text{ if } \tet = \tet' \text{ and } \sigma' = \overline{\sigma}.\\
    (3) & (\emph{Green edge}) && \quad \text{ if } \tet \text{ and } \tet' \text{ are glued along } (\sigma,\sigma').
\end{alignat*}
It is easy to check that every pair of vertices is connected by at most one edge, and there are no loops. Moreover, for every tetrahedron $\tet \in \Tet(\Delta)$, the subgraph $\graph_\Delta^\tet$ of $\graph_\Delta$ with vertices labeled with $\tet$ is combinatorially isomorphic to $\Cay(4)$, the Cayley graph of the alternating group $\Alt(4)$ with the standard $3$--cycle and $2$-$2$--cycle generating set. In \S\ref{subsec:edge_faces}, we described a meaningful way to embed $\Cay(4)$ inside a standard tetrahedron. This subgraph is the $1$--skeleton of a \emph{truncated tetrahedron} dual to $\tet$. More precisely, if one takes the tetrahedron dual to $\tet$ and truncates the vertices by planes parallel to the faces of $\tet$ then $\graph_\Delta^\tet$ is the $1$--skeleton of the resulting truncated tetrahedron.  This construction can be extended to embed $\graph_\Delta$ inside $\Delta$.

The \emph{monodromy complex} is the $\CW$--complex $\complex_\Delta$ obtained from attaching four different types of $2$--cells to loops in the monodromy graph $\graph_\Delta$. The loops on the boundary of the first three types of $2$--cells are easy to describe. We denote a loop in $\graph_\Delta$ by a cyclically ordered sequence of its vertices.
\begin{alignat*}{3}
    (1) & (\emph{Triangle}) && \quad \cyclic{(\tet,\sigma),(\tet,\sigma_+),(\tet,\sigma_-)}, \ \forall\ (\tet,\sigma) \in \Tet(\Delta)\times \edgeface
    \\
    (2) & (\emph{Quadrilateral}) && \quad \cyclic{(\tet,\sigma),(\tet',\tau),(\tet',\tau_-),(\tet',\sigma_+)}, \ \forall \tet,\tet' \in \Tet(\Delta) \text{ glued along } (\sigma,\tau).\\
    (3) & (\emph{Hexagon}) && \quad \cyclic{(\tet,\sigma),(\tet,\sigma_-),(\tet,\overline{\sigma_-}),(\tet,(\overline{\sigma_-})_-),(\tet,\overline{\sigma}_+),(\tet,\overline{\sigma})}, \
    \forall\ (\tet,\sigma) \in \Tet(\Delta)\times \edgeface.
\end{alignat*}
Triangular and hexagonal faces can be seen in Figure \ref{fig:EdgeFace_labels} and quadrilateral faces can been seen in Figure \ref{fig:Monodromy_graph}.  To described the fourth type of cell we need the following notation. For every edge $s\in \Edge(\Delta)$, fix once and for all an orientation on $s$, and let $\cyclic{\tet_1^s,\ldots, \tet_{k_s}^s}$ be the cyclically ordered ${k_s}$--tuple of tetrahedra that abut $s$, cyclically ordered to follow the ``right hand rule'' by placing the thumb in the direction of $s$. The number $k_s$ is called the \emph{valence} of $s$. For each $1\leq i\leq {k_s}$, the tetrahedra $\tet_i^s$ and $\tet_{i+1}^s$ are glued along a pair $(\sigma_i^s,\tau_{i+1}^s)$ of edge-faces (here all indices are taken mod ${k_s}$). The edge-face $\tau_{i}^s$ (resp.\ $\sigma_i^s$) is called the \emph{incoming} (resp.\ \emph{outgoing}) edge-face of $\tet_i^s$ around $s$. We remark that the two faces of $\tet^s_i$ glued to $\tet_{i-1}^s$ and $\tet_{i+1}^s$ share the edge $s$, thus  $\overline{\sigma_{i}^s}=\tau_i^s$.

\begin{figure}[t]
    \centering
    \includegraphics[width=\textwidth]{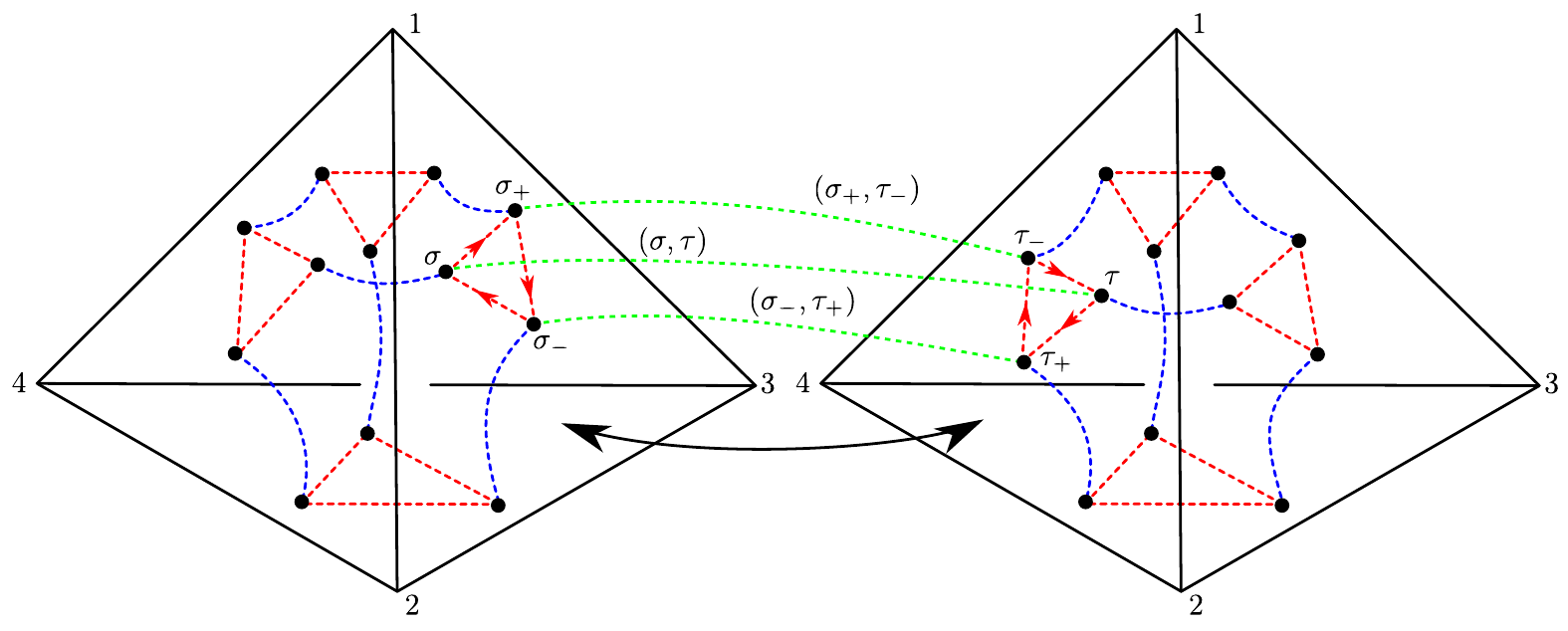}
    \caption{A gluing of two truncated tetrahedra in the monodromy graph.}    \label{fig:Monodromy_graph}
\end{figure}

The last type of $2$--cell in $\complex_\Delta$ is the $2{k_s}$--gon attached to the following loop around $s$, which alternates between green and blue edges (cf. Figure~\ref{fig:2kgon}).
\begin{alignat*}{3}
    (4) & (2{k_s}\emph{--gon}) && \quad \cyclic{(\tet_1^s,\sigma_1^s),(\tet_2^s,\tau_2^s),(\tet_2^s,\sigma_2^s),\dots,(\tet_{{k_s}}^s,\sigma_{{k_s}}^s),(\tet_1^s,\tau_1^s)} \text{ for all } s \in \Edge(\Delta).
\end{alignat*}

\begin{figure}[t]
    \centering
    \includegraphics[height=5cm]{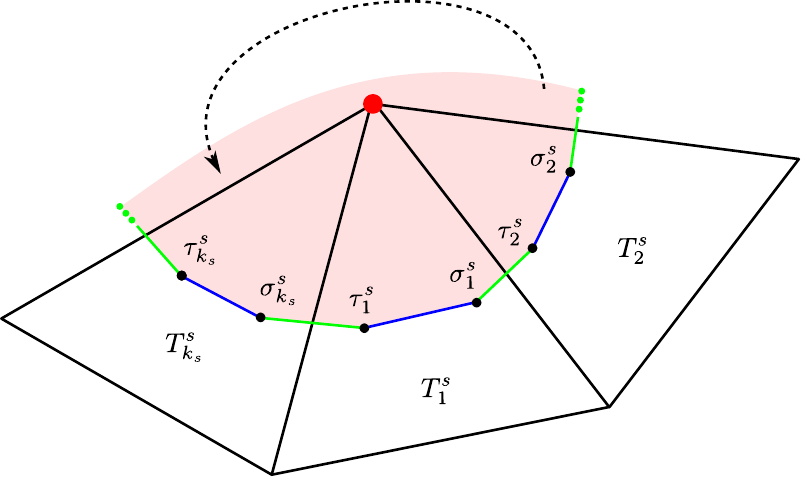}
    \caption{Every edge in $\Delta$ is dual to a $2k_s$--gon in the monodromy complex.}    \label{fig:2kgon}
\end{figure}

The embedding of $\graph_\Delta$ into $\Delta$ naturally extends to an embedding of $\complex_\Delta$ into $\Delta$, thus we can regard the monodromy complex as a subset of $M$ (via its ideal triangulation). Furthermore, if $\wt \Delta$ is the ideal triangulation of the universal cover of $M$ obtained by lifting the triangulation $\Delta$, then it is easy to show that $\wt \complex_\Delta=\complex_{\wt \Delta}$. In other words, the universal cover of the monodromy complex is the monodromy complex of the triangulation $\wt \Delta$. The next result shows that $\complex_\Delta$ carries the entire fundamental group of $M$ (and of $\Delta$).

\begin{lemma}\label{lem:complex_fund_group}
If $\Delta$ is an ideal triangulation of $M$ and $\iota:\complex_\Delta \hookrightarrow M$ is the inclusion of its monodromy complex, then $\iota_*:\pi_1(\complex_\Delta)\to \pi_1(M)$ 
is an isomorphism. 
\end{lemma}

\begin{proof}
We recall that, for each tetrahedron $\tet \in \Tet(\Delta)$, the subcomplex of $\complex_\Delta$ with vertices labeled with $\tet$ is a truncated tetrahedron contained in $\tet$, thus it bounds a topological $3$--ball. Its boundary is made out of four triangular $2$--cells and four hexagonal $2$--cells. Similarly, for each pair of tetrahedra $\tet,\tet' \in \Tet(\Delta)$ that are glued along a pair of edge-faces $(\sigma,\tau)$, the subcomplex of $\complex_\Delta$ with vertices
$$
\{(\tet,\sigma),(\tet,\sigma_+),(\tet,\sigma_-),(\tet',\tau),(\tet',\tau_+),(\tet',\tau_-)\}
$$
is a triangular prism that bounds a $3$--ball inside $\tet \cup \tet'$. Its boundary is made out of two triangular $2$--cells and three quadrilateral $2$--cells.

By adding all of these two types of $3$--cells, the monodromy complex $\complex_\Delta$ can be extended to a $3$--dimensional $\CW$--complex $\complex'_\Delta$ with the same fundamental group. The complex $\complex'_\Delta$ deformation retracts onto the \emph{dual spine} of $\Delta$, by collapsing every truncated tetrahedron into a single point and every prism to a line segment between two such points. We refer to~\cite[\S $1.1.2$]{MAT} for the definition of a dual spine of a triangulation, and to~\cite[\S $2.7$]{HRST15} for the more specific setting of an ideal triangulation. Since $M$ deformation retracts onto its dual spine (cf.~\cite[Thm $1.1.7$]{MAT}), the result follows.
\end{proof}


\subsection{Cochains and cocycles}\label{subsec:cocycles}
	 
Let $H$ be a group and $\complex$ be a finite dimensional $\CW$--complex.  A \emph{$H$--cochain} on $\complex$ is a function from the set of oriented $1$--cells of $\complex$ to $H$.  A \emph{$H$--cocycle} on $\complex$ is a $H$--cochain with the following properties:
\begin{itemize}
    \item oppositely oriented edges of $\complex$ with the same vertices are mapped to inverse elements of $H$;
	\item for each oriented $2$--cell $D$, in $\complex$, the product of the elements along the boundary of $D$ is the identity in $H$. 
\end{itemize}
We denote by $\coch(\complex,H)$ and $\cocy(\complex,H)$ the \emph{sets of all $H$--cochains and $H$--cocycles on $\complex$}, respectively. 

An oriented path $\alpha$ in $\complex$ is \emph{simplicial} if it is contained in the $1$--skeleton of $\complex$. For each simplicial path $\alpha$ and each  $H$--cochain $f$ on $\complex$, we define $f(\alpha)$ as the product of the elements of $H$ along $\alpha$ determined by $f$. If $f$ is a $H$--cocycle on $\complex$ and $\alpha$ and $\alpha'$ are homotopic (rel.\ endpoints) simplicial paths, then it is easy to check that $f(\alpha)=f(\alpha')$. Furthermore, every path in $\complex$ with endpoints in the $1$--skeleton is homotopic (rel.\ endpoints) to the $1$--skeleton of $\complex$, which allows us to extend the previous definition to any path in $\complex$ based at a point in the $1$--skeleton.

\begin{lemma}\label{lem:cocyle_gives_rep}
	 Let $u \in \complex$ be a vertex, then each cocycle $c\in \cocy(\complex,H)$ determines a representation $\rho_c:\pi_1(\complex,u) \to H$.
\end{lemma}
\begin{proof}
	 Let $[\gamma]\in \pi_1(\complex,u)$ and let $\gamma$ be a homotopic representative lying in the $1$--skeleton of $\complex$. Every time $\gamma$ traverses an oriented edge $e$ of $\complex$, the cocycle $c$ determines an element $c(e)\in H$. Multiplying the resulting group elements determines the value $\rho_c([\gamma])$. Since products along the boundaries of $2$--cells are trivial, the map $\rho_c$ does not depend on the representative $\gamma$ and is well defined. Moreover, since multiplication in $\pi_1(\complex,u)$ is given by concatenating paths, the map $\rho_c$ is a homomorphism. 
\end{proof}


\subsection{Edge gluing equations}\label{subsec:gluing_eq}

We recall that $\mathcal{D}'_\Delta$ is the codomain of $\Psi_\Delta$, defined at the beginning of \S\ref{sec:coords_triangulation}. Here we are going to determine equations that cut out a semi-algebraic subset $\mathcal{D}_\Delta = \mathcal{D}_\rp(M;\Delta)$ of $\mathcal{D}'_\Delta$, which will be shown to be homeomorphic to the space of triangulations of flags $\cTriFL$ via $\Psi_\Delta$ in \S\ref{subsec:parametrization_tri_of_flags}.

We begin by defining a map
$$
\Co : \mathcal{D}'_\Delta \rightarrow \coch(\complex_\Delta, \PGL(4)).
$$

Let $x \in \mathcal{D}'_\Delta$. We fix the following notation to describe the components of $x$. The space $\mathcal{D}'_\Delta$ is a subset of $(\tetsubvariety \times \RR_{>0}^{\edgeface})^{\Tet(\Delta)}$, thus for every tetrahedron $\tet \in \Tet(\Delta)$, we can write $x(\tet) = \left(x_1(\tet),x_2(\tet)\right)$, where $x_1(\tet) \in \tetsubvariety$ and $x_2(\tet) \in \RR_{>0}^{\edgeface}$. Since $\tetsubvariety \subset \RR_{>0}^{\edgeface} \times \RR_{>0}^{\edgeface}$, we will adopt the notation 
$$
x(\tet) = \left(x_1(\tet),x_2(\tet)\right) = \left( (\tau_*^\tet,\epsilon_*^\tet),\kappa_*^\tet \right) \in \left( \RR_{>0}^{\edgeface} \times\RR_{>0}^{\edgeface} \right) \times\RR_{>0}^{\edgeface}.
$$
Now we can describe the cochain $\Co(x)$. We recall that the $1$--cells of $\complex_\Delta$ are of three types: red, blue and green (cf. \S\ref{subsec:monodromy_complex}). The vertices of $\complex_\Delta$ are pairs $(\tet,\sigma)$ of a tetrahedron $\tet \in \Tet(\Delta)$ and an edge-face $\sigma \in \edgeface$. We are going to encode an oriented $1$--cell with an ordered pair $(u_1,u_2)$ of its starting vertex $u_1$ and final vertex $u_2$.

\begin{enumerate}
    \item (\emph{Red edges}) For every $\tet \in \Tet(\Delta)$ and $\sigma \in \edgeface$, we define the cochain of the oriented red edge $\left((\tet,\sigma_+),(\tet,\sigma)\right)$ as
    \begin{equation} \label{eq:rot_mat}
    \Co(x)\left((\tet,\sigma_+),(\tet,\sigma)\right) :=\Rot_\sigma(x(\tet))=
    \begin{bmatrix}
    0 & 1 & 0 & 0\\
	\tau_{\sigma}^\tet & 1 & -$1$--\tau_{\sigma}^\tet & 0\\
	0 & 1 & -1 & 0 \\
	0 & 0 & 0 & \frac{1}{\epsilon_{\sigma_+}^\tet}
	\end{bmatrix}.
    \end{equation}
    Then we set $\Co(x)\left((\tet,\sigma),(\tet,\sigma_+)\right) := \left(\Co(x)\left((\tet,\sigma_+),(\tet,\sigma)\right) \right)^{-1}$.
    
    \item (\emph{Blue edges}) For every $\tet \in \Tet(\Delta)$ and $\sigma \in \edgeface$, we define the cochain of the oriented blue edge $\left((\tet,\overline{\sigma}),(\tet,\sigma)\right)$ as
    \begin{equation} \label{eq:flip_mat}
    \Co(x)\left((\tet,\overline{\sigma}),(\tet,\sigma)\right) :=\Flip_\sigma(x(\tet))=
    \begin{bmatrix}
    0 & \epsilon_\sigma^\tet & 0 & 0\\
    1 & 0 & 0 & 0\\
    0 & 0 & \X_{\overline{\sigma}}^\tet & \X_\sigma^\tet \X_{\overline{\sigma}}^\tet-\epsilon_\sigma^\tet\\
    0 & 0 & -1 & -\X_\sigma^\tet
	\end{bmatrix},
	\end{equation}
	where
	$$
    \X_{\sigma'}^\tet := \frac{\nu_{\sigma'}}{\epsilon_{\sigma'_-}^\tet \epsilon_{\sigma'_+}^\tet (\tau_{\sigma'}^\tet+1)},\quad\text{and} \quad \nu_{\sigma'}^\tet := \epsilon_{\sigma'_-}^\tet \epsilon_{\sigma'_+}^\tet - \epsilon_{\sigma'_-}^\tet + 1, \qquad \text{for } \ \sigma' \in \{\sigma,\overline{\sigma}\}.
    $$
   
    \item (\emph{Green edges}) 
    For every pair of tetrahedra $\tet,\tet' \in \Tet(\Delta)$ and edge-faces $\sigma,\tau \in \edgeface$ such that $\tet$ and $\tet'$ are glued along $(\sigma,\tau)$, we define the cochain of the oriented green edge $\left((\tet,\sigma),(\tet',\tau)\right)$ as
    \begin{equation} \label{eq:glue_mat}
    \Co(x)\left((\tet',\tau),(\tet,\sigma)\right) :=\Glue_\sigma(x(\tet))=
    \begin{bmatrix}
    0 & 1 & 0 & 0\\
    1 & 0 & 0 & 0\\
    0 & 0 & 1 & 0\\
    0 & 0 & 0 & -\kappa^\tet_\sigma
	\end{bmatrix}.
	\end{equation}
\end{enumerate}

All of the above transformations depend continuously on the parameters of $x$ and so it is easy to check that $\Co$ is continuous and injective, but not surjective. At first sight, the above transformations seem mysterious, but they can each be described geometrically. Loosely speaking, each of the above transformations maps the standard position of its terminal vertex to the standard position of its initial vertex. This is made precise in the following lemma.

\begin{lemma}\label{lem:std_pos_action}
Let $x\in \mathcal{D}'_\Delta$ and let $\tet$ and $\tet'$ be two tetrahedra in $\Delta$ that are glued via the edge-face pair $(\sigma,\tau)$. Next, let $\F$ be the $\sigma$--standard representative of $\Psi^{-1}(x(\tet)_1)$ and let $\E$ be the representative of $\Psi^{-1}(x(\tet')_1)$ so that $(\F,\E)$ is in $(\sigma,\tau)$--standard glued position. Then
\begin{enumerate}
\item $\Flip_\sigma(x(\tet))\cdot \F$ is the $\overline{\sigma}$--standard representative of $[\F]$,
    \item $(\Rot_\sigma(x(\tet))\cdot \F,\Rot_\sigma(x(\tet))\cdot \E)$ is in $(\sigma_+,\tau_-)$--standard glued position, and 
    \item $(\Glue_\sigma(x(\tet))\cdot \F,\Glue_\sigma(x(\tet))\cdot \E)$ is in $(\tau,\sigma)$--standard glued position. 
\end{enumerate}
\end{lemma}

\begin{proof}
All three of these properties follow from direct computations using the definitions of standard and standard glued positions, and the fact that the entries of $x(\tet)$ and $x(\tet')$ satisfy the conditions \eqref{eq:pair_cond_1}, \eqref{eq:pair_cond_2}, \eqref{eq:pair_cond_3} defining $\mathcal{D}_\Delta'$.
\end{proof}

The previous lemma also provides motivation for the suggestive naming of the above transformations: the transformation $\Rot_\sigma(x(\tet))$ \emph{rotates} the edge-face $\sigma=(ij)k$ to the edge-face $\sigma_+=(ki)j$, $\Flip_\sigma(x(\tet))$ \emph{flips} the edge-face $\sigma=(ij)k$ to the edge-face $\overline{\sigma}=(ji)l$, and $\Glue_\sigma(x(\tet))$ \emph{glues} the edge-face $\sigma$ to the edge-face $\tau$.

Now we are going to determine necessary and sufficient conditions for the cochain $\Co(x)$ to be a cocycle. The next lemma shows that $\Co(x)$ already satisfies most of the conditions necessary to be a cocycle.

\begin{lemma}\label{lem:cochain_almost_cocycle}
    Let $x\in \mathcal{D}'_\Delta$. The cochain $\Co(x)$ maps oppositely oriented edges of $\complex_\Delta$ to inverse elements of $\PGL(4)$. Furthermore, the product of the matrices associated by $\Co(x)$ to the oriented edges along the boundary of a triangular, quadrilateral and hexagonal $2$--cell of $\complex_\Delta$ is trivial.
\end{lemma}
\begin{proof}
    First, we remark that the cochain $\Co(x)$ is defined to map oppositely oriented red edges of $\complex_\Delta$ to inverse elements of $\PGL(4)$. Using that $\epsilon_\sigma^\tet = \epsilon_{\overline \sigma}^\tet$ and $\kappa^\tet_\sigma \kappa^{\tet'}_\tau = 1$, it is easy to check that the same property holds for blue and green edges.
    
    Next we take care of the $2$--cells. The proof consists in checking that products of matrices of the form \eqref{eq:rot_mat},\eqref{eq:flip_mat} and \eqref{eq:glue_mat} along the boundary of triangular, quadrilateral and hexagonal $2$--cells of $\complex_\Delta$ is trivial. This is a straightforward but tedious computation that requires the use of the relations among the coordinates of $x$ in $\mathcal{D}'_\Delta$. In particular, for triangular $2$--cells, it is enough to use the relations
    $$
    \tau_\sigma^\tet(\epsilon_\sigma^\tet \epsilon_{\sigma_+}^\tet \epsilon_{\sigma_-}^\tet) =1, \qquad \forall \sigma \in \edgeface, \quad \text{and} \quad \forall \tet \in \Tet(\Delta).
    $$
    For the quadrilateral $2$--cells, one also needs
    $$
    \tau_\sigma^\tet \tau_\tau^{\tet'} = 1, \qquad \text{and} \qquad \kappa_{\sigma_-}^\tet = \kappa_\sigma^\tet \epsilon_\sigma^\tet \epsilon_{\tau_+}^{\tet'},
    $$
    for all tetrahedra $\tet,\tet' \in \Tet(\Delta)$ glued along the pair of edge-faces $(\sigma,\tau))$.
    Finally, for the hexagonal $2$--cells, it is enough to use the relations
    $$
    \tau_\sigma^\tet(\epsilon_\sigma^\tet \epsilon_{\sigma_+}^\tet \epsilon_{\sigma_-}^\tet) =1, \quad \text{and} \quad \epsilon_\sigma^\tet \epsilon_{\sigma_+}^\tet \epsilon_{\sigma_-}^\tet \epsilon_{\op{\sigma}}^\tet \epsilon_{\op{(\sigma_+)}}^\tet \epsilon_{\op{(\sigma_-)}}^\tet =1, \quad \qquad \forall \sigma \in \edgeface, \quad \text{and} \quad \forall \tet \in \Tet(\Delta).
    $$
    
\end{proof}

\begin{remark}
Note that Lemma~\ref{lem:std_pos_action} provides an alternative proof of Lemma~\ref{lem:cochain_almost_cocycle}. Specifically, Lemma~\ref{lem:std_pos_action} can be inductively used to show that the product of the elements determined by $\Co(x)$ along the triangular, quadrilateral, and hexagonal faces each fix a $\sigma$--standard representative of some tetrahedron of flags, and is thus trivial. 
\end{remark}

We recall that $\complex_\Delta$ contains a fourth type of $2$--cells dual to the edges in $\Edge(\Delta)$. For every edge $s \in \Edge(\Delta)$, let $k_s$ be the valence of $s$. Then there is a $2{k_s}$--gon in $\complex_\Delta$ attached to a loop around $s$ which is an alternating sequence of green and glue edges (cf. \S\ref{subsec:monodromy_complex}):
$$\cyclic{(\tet_1^s,\sigma_1^s),(\tet_2^s,\tau_2^s),(\tet_2^s,\sigma_2^s),\dots,(\tet_{{k_s}}^s,\sigma_{{k_s}}^s),(\tet_1^s,\tau_1^s)}.$$
Let 
{\begin{equation}\label{eq:edge_product}
    G^s:=\prod_{i=1}^{k_s}\Glue_{\tau_i^s}(x(\tet_i^s))\Flip_{\sigma_i^s}(x(\tet_i^s))
\end{equation}
}
be the product of the matrices associated by $\Co(x)$ to the boundary of this $2{k_s}$--gon. We remark that $G^s$ is a product of matrices of the form \eqref{eq:flip_mat} and \eqref{eq:glue_mat}, therefore
\begin{equation}\label{eq:prod_matrix_around_edge}
    G^s=\begin{bmatrix}
    G_{11}^s & 0 & 0 & 0 \\
    0 & 1 & 0 & 0 \\
    0 & 0 & G_{33}^s & G_{34}^s\\
    0 & 0 & G_{43}^s & G_{44}^s
    \end{bmatrix}.
\end{equation}
It follows from the definition of cocycles (cf. \S\ref{subsec:cocycles}) and Lemma~\ref{lem:cochain_almost_cocycle} that the cochain $\Co(x)$ is a cocycle if and only if the matrix $G^s$ is trivial for each $s\in \Edge(\Delta)$. This is equivalent to the following \emph{edge gluing equations} (of $\Delta$ with respect to $s \in \Edge(\Delta)$):
\begin{align}\label{eq:edge_gluing_equations}
    G_{11}^s&=G_{33}^s=G_{44}^s=1,\\ \nonumber
    G_{34}^s&=G_{43}^s=0.
\end{align}

\begin{remark}\label{rem:easy_edge_gluing_equations}
In general, the entries of $G^s$ are complicated expressions of the coordinates in $x$, however the $(1,1)$--entry is simple: 
\begin{align}\label{eq:simple_entry_gluing_equations}
    G_{11}^s=\prod_{i=1}^{k_s} \epsilon^{\tet_i^s}_{\sigma_i^s},
\end{align}
Informally, this says that the product of the edge-ratios of the edges identified with $s$ must be equal to $1$. Similarly, the determinant of $G^s$ is simple to determine:
\begin{equation}\label{eq:simple_det_gluing_equations}
\det(G^s) = G_{11}^s \left(G_{33}^s G_{44}^s-G_{43}^s G_{34}^s\right) = \prod_{i=1}^{k_s} \left(\epsilon^{\tet_i^s}_{\sigma_i^s}\right)^2\kappa^{\tet_i^s}_{\sigma_i^s}.
\end{equation}
\end{remark}

We denote by $\mathcal{D}_\Delta = \mathcal{D}_\rp(M;\Delta)$ the semi-algebraic affine subset of $\mathcal{D}'_\Delta$ consisting of those points satisfying the edge gluing equations \eqref{eq:edge_gluing_equations}, for each edge in $\Edge(\Delta)$. The set $\mathcal{D}_\Delta$ is called the \emph{deformation space of a triangulation of flags}. We summarize the above discussion in the following result.

\begin{theorem}\label{thm:deformation_space_to_cocycles}
    Let $x \in \mathcal{D}'_\Delta$, then $\Co(x)$ is a cocycle if and only if $x \in \mathcal{D}_\Delta$. Namely
    $$
    \Co^{-1}(\cocy(\complex_\Delta,\PGL(4))) = \mathcal{D}_\Delta,\qquad \text{and} \qquad \restr{\Co}{\mathcal{D}_\Delta}: \mathcal{D}_\Delta \hookrightarrow \cocy(\complex_\Delta,\PGL(4)).
    $$
\end{theorem}

We conclude this section with a technical result that will be needed to prove the main Theorem~\ref{thm:parametrization_def_space} in the next section.

Let $v$ be an (ideal) vertex in $\Delta$ and let $\alpha$ be an oriented simplicial path in $\complex_\Delta$ with the following property. If $\left(\tet_i,\sigma_i\right)_{i=0}^k$ is the ordered list of vertices of $\complex_\Delta$ crossed by $\alpha$, then
\begin{itemize}
    \item $\alpha$ crosses an even number of edges (i.e. $k$ is even);
    \item and for every even $j$, the vertex $v$ is the initial endpoint of the underlying oriented edge of $\sigma_j$.
\end{itemize}

A path that is homotopic to a path with the above properties is a \emph{peripheral path around $v$}. A more geometric description is that a path is a peripheral path around $v$ if it is homotopic (rel.\ endpoints) into a neighborhood of the vertex $v$ in $\Delta$. The next lemma shows that peripheral paths always preserve an incomplete flag. 

\begin{lemma}\label{lem:peripheral_fixes_flag}
    Let $x \in \mathcal{D}_\Delta$ and let $\alpha$ be an oriented peripheral path around a vertex $v$ of $\Delta$. Then $\Co(x)(\alpha)$ fixes the incomplete flag $([\basis{1}],[\dbasis{2}])$.
\end{lemma}

\begin{proof}
    Since $\Co(x)$ is a cocycle (cf. Theorem~\ref{thm:deformation_space_to_cocycles}), the quantity $\Co(x)(\alpha)$ does not depend on the homotopy class of $\alpha$. Hence we can assume that $\alpha$ is a path of the form
    $$
    \left((\tet_1,\sigma_1),(\tet_2,\sigma_2),\dots, (\tet_{2k},\sigma_{2k}) \right),
    $$
    where $k \in \NN$ and for every even $j$, $v$ is the starting endpoint of the underlying oriented edge $\sigma_j$.
    
    We observe that $\alpha$ is the concatenation of $k$ oriented peripheral paths around $v$ of length two. If each of these paths satisfies the conclusion of the lemma, then $\alpha$ also satisfies the conclusion of the lemma. It is therefore sufficient to prove the statement for $\alpha$ of length two (i.e. $k = 1$).

    There are only four possibilities for $\alpha$, namely
    \begin{alignat*}{3}
    &(i) \quad \big((\tet_0,\sigma_0),(\tet_0,\overline{\sigma_0}), (\tet_{0},(\overline{\sigma_0})_-) \big),  \qquad &&(ii) \quad  \big((\tet_0,\sigma_0),(\tet_0,(\sigma_0)_+), (\tet_{0},\overline{(\sigma_0)_+} \big),\\
    &(iii) \quad \big((\tet_0,\sigma_0),(\tet_1,\sigma_1), (\tet_1,\overline{\sigma_1}) \big), \qquad &&(iv) \quad  \big((\tet_0,\sigma_0),(\tet_1,\sigma_1), (\tet_1,(\sigma_1)_-) \big),
    \end{alignat*}

    where in $(iii)$ and $(iv)$, the tetrahedron $\tet_1$ is the unique tetrahedron of $\Delta$ glued to $\tet_0$ along $(\sigma_0,\sigma_1)$. Using the matrices in \eqref{eq:rot_mat},\eqref{eq:flip_mat} and \eqref{eq:glue_mat} it is easy to check that the cocycle associated to these paths fixes the incomplete flag $([\basis{1}],[\dbasis{2}])$.
\end{proof}


\subsection{The parametrization of \texorpdfstring{$\cTriFL$}{TEXT} }\label{subsec:parametrization_tri_of_flags}

In this section we show that the image of the map $\Psi_\Delta : \cTriFL \rightarrow \mathcal{D}'_\Delta$ is contained in the deformation space $\mathcal{D}_\Delta$ (cf. Corollary~\ref{cor:image_of_param_tri_of_flags}), allowing us to write
$$
\Psi_\Delta : \cTriFL \rightarrow \mathcal{D}_\Delta.
$$
Finally we will prove that $\Psi_\Delta$ is a homeomorphism onto $\mathcal{D}_\Delta$.

To prove that $\Psi_\Delta(\cTriFL) \subset \mathcal{D}_\Delta$, it is enough to show that $(\Co \circ \Psi_\Delta)(p) \in \cocy(\complex_\Delta,\PGL(4))$ for all $p \in \cTriFL$, and conclude using Theorem~\ref{thm:deformation_space_to_cocycles}. This is easy to see once we understand how $(\Co \circ \Psi_\Delta)(p)$ acts on the flags of $p$. To avoid introducing additional notation, we will employ the following abuse of notation. If $\alpha$ is an oriented simplicial path in the monodromy complex $\complex_{\wt \Delta}$ of $\wt \Delta$, then there is a unique oriented simplicial path $\overline{\alpha}$ in $\complex_{ \Delta}$ that lifts to $\alpha$. Then if $c \in \cocy(\complex_\Delta,\PGL(4))$, we will denote by $c(\alpha) := c(\overline{\alpha})$.

\begin{lemma}\label{lem:cochain_moves_standard_positions}
Let $p = [\Phi,\rho] \in \cTriFL$ and let $\alpha$ be an oriented simplicial path in the monodromy complex $\complex_{\wt \Delta}$ of $\wt \Delta$. Let $(\wt \tet_1,\sigma_1)$ (resp. $(\wt \tet_2,\sigma_2)$) be the starting (resp. ending) vertex of $\alpha$, and set $G_\alpha := (\Co \circ \Psi_\Delta)(p)(\alpha)$. If $(\Phi,\rho)$ is the representative pair of $p$ such that $\Phi(\tet_2)$ is in $\sigma_2$--standard position, then $G_\alpha \cdot \Phi(\tet_1)$ is in $\sigma_1$--standard position.
\end{lemma}

\begin{proof}
    We remark that if $\alpha$ is the concatenation of finitely many oriented simplicial paths satisfying the lemma, then $\alpha$ also satisfies the lemma. Since $\alpha$ is a sequence of oriented edges in $\complex_{\wt \Delta}$, it will be enough to prove the statement for paths of length one.
    
    Thus suppose $\alpha$ consists of a single oriented edge $\left((\wt \tet_1,\sigma_1),(\wt \tet_2,\sigma_2)\right)$. This edge is either red, blue or green, and for these edges, the result follows from Lemma~\ref{lem:std_pos_action}. For instance, if the edge is blue the edge above is $\left((\wt \tet_2,\overline{\sigma_2}),(\wt \tet_2,\sigma_2)\right)$ and $G_\alpha$ is $\Flip_{\sigma_2}(\Psi_\Delta(p)(\tet_2))$, where $\tet_2$ is the tetrahedron in $\Delta$ covered by $\wt\tet_2$. In this case the statement of the lemma is that $\Flip_{\sigma_2}(\Psi_\Delta(p)(\tet_2))$ maps the $\sigma_2$--standard position of $\Phi(\tet_2)$ to the $\overline{\sigma_2}$--standard position of $\Phi(\tet_2)$, which is guaranteed by Lemma~\ref{lem:std_pos_action}. The proof for the red and green edges follows from a similar argument.
\end{proof}

\begin{corollary}\label{cor:tri_of_flags_give_cocycles}
    For all $p=[\Phi,\rho]\in \cTriFL$, the cochain $(\Co \circ \Psi_\Delta)(p)$ is a cocycle.
\end{corollary}

\begin{proof}
    Let $\alpha$ be an oriented simplicial loop in $\complex_\Delta$ based at $(\tet,\sigma)$ that bounds a $2$--cell in $\complex_\Delta$ and let ${G_\alpha=(\Co \circ \Psi_\Delta)(p)(\alpha)}$. Since $\alpha$ is homotopically trivial and $\complex_{\wt\Delta}=\wt\complex_{\Delta}$, it follows that $\alpha$ lifts to an oriented simplicial loop in $\complex_{\wt \Delta}$ based at $(\wt\tet,\sigma)$ covering $(\tet,\sigma)$. By Lemma~\ref{lem:cochain_moves_standard_positions} it follows that $G_\alpha$ fixes the $\sigma$--standard representative of $[\Phi(\wt\tet)]$. However, the stabilizer of a tetrahedron of flags is trivial (cf. Lemma~\ref{lem:tet_stab_is_trivial}), and so $G_\alpha$ is the identity in $\PGL(4)$. In particular, this implies that $(\Co \circ \Psi_\Delta)(p)$ is a cocycle.
\end{proof}

\begin{corollary}\label{cor:image_of_param_tri_of_flags}
    The image of the map $\Psi_\Delta : \cTriFL \rightarrow \mathcal{D}'_\Delta$ is contained in $\mathcal{D}_\Delta$, and thus
$$
\Psi_\Delta : \cTriFL \rightarrow \mathcal{D}_\Delta
$$
is well defined.
\end{corollary}

\begin{proof}
    Let $p \in \cTriFL$. By Corollary~\ref{cor:tri_of_flags_give_cocycles}, $(\Co \circ \Psi_\Delta)(p)$ is a cocycle. Therefore $\Psi_\Delta(p) \in \mathcal{D}_\Delta$, by Theorem~\ref{thm:deformation_space_to_cocycles}.
\end{proof}

We close this section by showing that $\Psi_\Delta$ is a homeomorphism. This is done by constructing an explicit inverse. Let $x \in \mathcal{D}_\Delta$. We define a triangulation of flags $p_x = [\Phi_x,\rho_x] \in \cTriFL$ as follows. First we fix a tetrahedron $\wt \tet_0 \in \wt \Delta$ lifting a tetrahedron $\tet_0 \in \Delta$, and an edge-face $\sigma_0 \in \edgeface$. As $\Co(x)$ is a cocycle on the monodromy complex $\complex_\Delta$ (cf. Theorem~\ref{thm:deformation_space_to_cocycles}), we define
    $$
    \rho_x : \pi_1(\complex_{\Delta}, (\tet_0, \sigma_0)) \rightarrow \PGL(4)
    $$
    to be the unique representation determined by $\Co(x)$ based at $(\tet_0, \sigma_0)$ (cf. Lemma~\ref{lem:cocyle_gives_rep}). We remark that $\pi_1(\complex_{\Delta}, (\tet_0, \sigma_0))$ is isomorphic to the fundamental group of $M$ (cf. Lemma~\ref{lem:complex_fund_group}).
    
    Next we define $\Phi_x:\Vertex(\wt\Delta)\to \FL$. For every vertex $v \in \Vertex(\wt \Delta)$, let $\wt \tet \in \wt \Delta$ be a tetrahedron with vertex $v$, and let $\sigma \in \edgeface$ be an edge-face of $\wt \tet$ such that $v$ is the initial vertex of the underlying edge in $\sigma$. Let $\alpha$ be an oriented simplicial path in $\complex_{\wt \Delta}$ from $(\wt \tet_0,\sigma_0)$ to $(\wt \tet, \sigma)$. Henceforth we adopt the notation $G_\gamma := \Co(x)(\gamma)$ for all oriented simplicial paths in $\complex_{\wt \Delta}$. Then we define
    $$
    \Phi_x(v) := G_\alpha \cdot ([\basis{1}],[\dbasis{2}]).
    $$
    
\begin{lemma}\label{lem:Phix_well_defined}
If $x\in \mathcal{D}_\Delta$, then $\Phi_x$ is well defined.
\end{lemma}

\begin{proof}
We must show that $\Phi_x$ is independent of the choice of both the vertex $(\wt \tet,\sigma)\in \complex_{\wt \Delta}$ and the path $\alpha$ from $(\wt \tet_0,\sigma_0)$ to $(\wt \tet, \sigma)$. First, suppose that $\alpha'$ is another path from $(\wt \tet_0,\sigma_0)$ to $(\wt \tet, \sigma)$. Let $\alpha^{-}$ be the path from $(\wt \tet,\sigma)$ to $(\wt \tet_0, \sigma_0)$ obtained by traversing $\alpha$ backwards. Then $\alpha'\cdot \alpha^{-}$ is a loop in $\complex_{\wt\Delta}$ based at $(\wt \tet_0,\sigma_0)$. Since $\complex_{\wt \Delta}$ is simply connected this loop is homotopically trivial and thus covers a homotopically trivial simplicial loop $\gamma$ in $\complex_{\Delta}$ based at $(\tet_0,\sigma_0)$. Since $\Co(x)$ is a cocycle (cf. Theorem~\ref{thm:deformation_space_to_cocycles}) it follows that $\Co(x)(\gamma)$ is trivial, and so $G_\alpha=G_{\alpha'}$. Hence $\Phi_x(v)$ is independent of $\alpha$. 

Next, suppose $(\wt \tet',\sigma')$ is another pair satisfying the same properties as $(\wt \tet,\sigma)$. Then there is an oriented peripheral path $\beta$ from $(\wt \tet,\sigma)$ to $(\wt \tet',\sigma')$, whose concatenation $\alpha \cdot \beta$ with $\alpha$ is an oriented simplicial path from $(\wt \tet_0,\sigma_0)$ to $(\wt \tet', \sigma')$. But then by Lemma~\ref{lem:peripheral_fixes_flag},
    $$
    G_{\alpha \cdot \beta} \cdot ([\basis{1}],[\dbasis{2}]) = G_{\alpha}G_{\beta} \cdot ([\basis{1}],[\dbasis{2}]) = G_{\alpha}\cdot ([\basis{1}],[\dbasis{2}]).
    $$
It follows that $\Phi_x(v)$ is also independent of the choice $(\wt\tet,\sigma)$, and so $\Phi_x$ is well defined.
\end{proof}

Next, we show that the pair $(\Phi_x,\rho_x)$ is a triangulation of flags.

\begin{lemma}\label{lem:Phix_triang_flags}
If $x\in \mathcal{D}_\Delta$, then $(\Phi_x,\rho_x)\in \FL_\Delta$.
\end{lemma}

\begin{proof}
In order to prove the lemma we need to show that conditions \eqref{item:def_tri_of_flags_1}, \eqref{item:def_tri_of_flags_2}, and \eqref{item:def_tri_of_flags_3} from \S\ref{subsec:tri_of_flags} are satisfied. First, suppose $\{v_i\}_{i=1}^4$ are the vertices of a tetrahedron $\wt \tet \in \wt \Delta$, and let $\sigma_i \in \edgeface$ be an edge-face of $\wt \tet$ such that $v_i$ is the starting vertex of the underlying edge of $\sigma_i$. Since $x_1(\wt \tet) \in \tetsubvariety$ then $\Psi^{-1}( x_1(\wt \tet) ) \in \ctetFL$ (cf. Theorem~\ref{thm:parametrization_flag_space}). We let $\F_{\wt \tet}^{\sigma_i}$ to be the $\sigma_i$--standard representative of $\Psi^{-1}( x_1(\wt \tet) )$. Let $\alpha_i$ be an oriented simplicial path in $\complex_{\wt \Delta}$ from $(\wt \tet_0,\sigma_0)$ to $(\wt \tet, \sigma_i)$.

In $\sigma_i$--standard position, the $i$--th flag of $\F_{\wt \tet}^{\sigma_i}$ is $([\basis{1}],[\dbasis{2}])$, and so by definition $\Phi_x(v_i)$ is the $i$--th flag of $G_{\alpha_i} \cdot \F_{\wt \tet}^{\sigma_i}$. We claim that
    \begin{equation}\label{eq:constructing_flags_is_well_defined}
    G_{\alpha_i} \cdot \F_{\wt \tet}^{\sigma_i} = G_{\alpha_j} \cdot \F_{\wt \tet}^{\sigma_j},\qquad  \forall i,j \in \{1,2,3,4\},
    \end{equation}
    from which it follows that
    \begin{equation}\label{eq:vertices_make_tet_of_flags}
    \Phi_x(\wt\tet):=\left( \Phi_x(v_1),\Phi_x(v_2),\Phi_x(v_3),\Phi_x(v_4) \right) = G_{\alpha_i} \cdot \F_{\wt \tet}^{\sigma_i} \in \ordFL_\tet.
    \end{equation}
    To prove~\eqref{eq:constructing_flags_is_well_defined}, let $\beta$ be an oriented simplicial path from $(\wt \tet, \sigma_i)$ to $(\wt \tet, \sigma_j)$. Then we claim that
    $$
    G_{\alpha_i} G_{\beta} = G_{\alpha_i \cdot \beta} = G_{\alpha_j}, \qquad \text{and} \qquad G_{\beta} \cdot \F_{\wt \tet}^{\sigma_j} = \F_{\wt \tet}^{\sigma_i}.
    $$
    The first point follows from the fact that $\Co(x)$ is a cocycle and that $\alpha_i\cdot\beta$ and $\alpha_j$ are homotopic paths in $\wt\complex_\Delta$. For the second point, observe that since $\F_{\wt \tet}^{\sigma_j}$ is in $\sigma_j$--standard position, Lemma~\ref{lem:cochain_moves_standard_positions} says that the transformation $G_{\beta}$ maps $\F_{\wt \tet}^{\sigma_j}$ into $\sigma_i$--standard position, which is $\F_{\wt \tet}^{\sigma_i}$. This proves that \eqref{item:def_tri_of_flags_1} is satisfied.
    
    Next, suppose $\wt \tet_1$ and $\wt \tet_2$ are tetrahedra in $\wt \Delta$ glued along a face $f$, and let $(\sigma_1,\sigma_2)$ be a pair of edge-faces encoding this gluing. For $i\in\{1,2\}$, let $\F_{\wt \tet_i}^{\sigma_i}$ to be the $\sigma_i$--standard representative of $\Psi^{-1}( x_1(\wt \tet_i) )$, and let $\alpha_i$ be an oriented simplicial path in $\complex_{\wt \Delta}$ from $(\wt \tet_0,\sigma_0)$ to $(\wt \tet_i, \sigma_i)$. Then by~\eqref{eq:vertices_make_tet_of_flags} we have that $G_{\alpha_1} \cdot \F_{\wt \tet_1}^{\sigma_1}$ and $G_{\alpha_2} \cdot \F_{\wt \tet_2}^{\sigma_2}$ are the tetrahedra of flags $\Phi_x(\wt \tet_1)$ and $\Phi_x(\wt \tet_2)$, respectively. Finally, let $\beta$ be the green edge in $\complex_{\wt \Delta}$ from $(\wt \tet_1,\sigma_1)$ to $(\wt \tet_2, \sigma_2)$. It follows from the definition of $\Co(x)$ that $G_\beta$ is $\Glue_{\sigma_2}(x(\wt\tet_2))$. Furthermore, since $x\in \mathcal{D}_\Delta$, the classes $[\F_{\wt \tet_1}^{\sigma_1}]$ and $[\F_{\wt \tet_2}^{\sigma_2}]$ are $(\sigma_1,\sigma_2)$--glueable and so $(\F_{\wt\tet_1}^{\sigma_1},G_{\beta}\cdot \F_{\wt \tet_2}^{\sigma_2})$ are in $(\sigma_1,\sigma_2)$--standard glued position. Furthermore, arguing as before, we see that $G_{\alpha_2}=G_{\alpha_1}G_{\beta}$ and so
    $$
    (G_{\alpha_1} \cdot \F_{\wt \tet_1}^{\sigma_1},G_{\alpha_2} \cdot \F_{\wt \tet_2}^{\sigma_2})=G_{\alpha_1}\cdot (\F_{\wt\tet_1}^{\sigma_1},G_{\beta}\cdot \F_{\wt \tet_2}^{\sigma_2})$$
    are in $(\sigma_1,\sigma_2)$--glued position. Since $x\in \mathcal{D}_\Delta$, it follows that the gluing parameter for this pair is positive and so this gluing is geometric.  This proves that \eqref{item:def_tri_of_flags_2} is satisfied.
    
    The fact that $\Phi_x$ is $\rho_x$--equivariant follows directly from the fact that $\rho_x$ is defined through $\Co(x)$. Indeed let $v \in \Vertex(\wt \Delta)$ and let $\gamma \in \pi_1(\complex_{\Delta}, (\tet_0, \sigma_0))$. Let $\sigma \in \edgeface$ be an edge-face of a tetrahedron $\wt \tet$ such that $v$ is the starting vertex of the underlying edge in $\sigma$ with respect to $\wt \tet$. Then $\gamma$ lifts to an oriented simplicial path $\alpha \cdot \wt \gamma \cdot \beta$ in $\complex_{\wt \Delta}$ which is a concatenation of paths starting at $(\wt \tet_0,\sigma_0)$, through $(\wt \tet, \sigma)$ and $(\gamma \cdot \wt \tet, \sigma)$, and ending at $(\gamma \cdot \wt \tet_0, \sigma_0)$. We remark that $\alpha$ and $\beta$ can be chosen to project to the same simplicial path in $\complex_{\Delta}$, with opposite orientations, thus $G_\beta= G_\alpha^{-1}$. By definition
    $$
    \rho_x(\gamma) = \Co(x)(\alpha \cdot \wt \gamma \cdot \beta)
    = G_\alpha  G_{\wt \gamma}  G_\alpha^{-1}.
    $$
    We conclude that
    $$
    \Phi_x(\gamma \cdot v) = G_{\alpha \cdot \wt \gamma} \cdot ([\basis{1}],[\dbasis{2}]) = G_{\alpha} G_{\wt \gamma} \cdot ([\basis{1}],[\dbasis{2}]) = G_{\alpha} G_{\wt \gamma}  G_\alpha^{-1} \cdot \Phi_x(v) = \rho_x(v) \cdot \Phi_x(v),
    $$
    which shows that condition \eqref{item:def_tri_of_flags_3} is satisfied.
    
\end{proof}
{
The $\PGL(4)$--class of $(\Phi_x,\rho_x)$ does not depend on the initial choice of pair $(\wt \tet_0,\sigma_0)$, therefore by Lemmas \ref{lem:Phix_well_defined} and \ref{lem:Phix_triang_flags} we have constructed a well defined map
$$
\Theta:\mathcal{D}_\Delta \rightarrow \cTriFL, \qquad \text{where} \qquad \Theta(x)= [\Phi_x,\rho_x].
$$
It is easy to see that $[\Phi_x,\rho_x]$ depends continuously on the coordinates in $x$, namely $\Theta$ is continuous.
}
\begin{theorem}\label{thm:parametrization_def_space}
The map $\Psi_\Delta : \cTriFL \rightarrow \mathcal{D}_\Delta$ is a homeomorphism.
\end{theorem}

\begin{proof}
The map $\Psi_\Delta$ is continuous because it is defined in terms of $\Psi$ and $\Phi$ that are continuous, thus it will be enough to show that $\Theta = \Psi^{-1}_\Delta$.

For every $\wt \tet \in \Tet(\wt \Delta)$
\begin{align*}
\Psi_\Delta(\Theta(x))(\wt \tet)=\Psi_\Delta([\Phi_x,\rho_x])(\wt\tet) = \left(\Psi([\Phi_x(\wt\tet)]),v_{\wt\tet}^{[\Phi_x,\rho_x]} \right)\\ = \left(\Psi(\Psi^{-1}( x_1(\wt \tet) )),x_2(\wt \tet) \right) = \left( x_1(\wt \tet) ,x_2(\wt \tet) \right) = x(\wt\tet).
\end{align*}
Furthermore, since the projective class of each tetrahedron of flags and each gluing is encoded by the parameters in $\mathcal{D}_\Delta$, it follows that the map $\Psi_\Delta$ is also injective. Hence the map $\Psi_\Delta$ is invertible with inverse $\Theta$.
\end{proof}


\section{Relationship with Thurston's equations}\label{sec:thu_equations}

In this section we discuss how our deformation space of a triangulation of flags is related to Thurston's deformation space for hyperbolic structures. The material is described in a more general setting, that also includes results of Danciger~\cite{DAN2} generalizing Thurston's technique to the case of Anti-de Sitter structures and half-pipe structures.

We begin by recalling most of the background, in a unified framework that includes all of these three geometries (hyperbolic, Anti-de Sitter and half-pipe) as subgeometries $(\G_\star,\XX_\star)$ of real projective geometry $(\PGL(4),\rp^3)$ (cf. \S\ref{subsec:X_spaces}). In \S\ref{subsec:thu_equations}, we review \emph{Thurston's deformation space} $\mathcal{D}_{\XX_{\star}}(M;\Delta)$ and the \emph{extension map} $\mathrm{Ext}_{\XX_\star} : \mathcal{D}_{\XX_{\star}}(M;\Delta) \rightarrow \XX_\star(M;\Delta)$, which associates \emph{Thurston's parameters} to $(\G_\star,\XX_\star)$--structures. Next, we show that there is a straightforward dictionary between our projective parameters and Thurston's parameters, which makes it easy to determine when a projective structure is a $(\G_\star,\XX_\star)$--structures (cf. \S\ref{subsec:X_tet_of_flags}). This leads to the construction of the map $\varphi_\star : \mathcal{D}_{\XX_{\star}}(M;\Delta) \rightarrow \mathcal{D}_{\rp}(M;\Delta)$ (cf. \S\ref{subsec:proof_main_thm}), which was mentioned in the statement of Theorem~\ref{thm:main_thm} in the preface. We give a complete description of $\varphi_\star$ (cf. Theorem~\ref{thm:phi_image}), and show that the maps $\varphi_\star$, $\Ext$, and $\mathrm{Ext}_{\XX_\star}$ are compatible, in the sense that they fit into a certain commutative diagram (cf. Theorem~\ref{thm:X_comm_diag}).


\subsection{Hyperbolic, Anti-de Sitter and half-pipe spaces}\label{subsec:X_spaces}

Let $\star\in \{-1,1,0\}$ and let $\mathcal{B}_\star$ be the $2$--dimensional $\RR$--algebra with basis $\{1,\iota\}$, such that $\iota^2=\star$. It is easy to check that $\mathcal{B}_{-1}$ is the algebra of complex numbers $\CC$, while $\mathcal{B}_{1}$ is the algebra of \emph{pseudo complex numbers}. If $z=a+\iota b\in \mathcal{B}_{\star}$, then $a =: \Re(z)$ and $b =: \Im(z)$ are called the \emph{real} part and \emph{imaginary} part of $z$, respectively. Each $z=a+\iota b$ has a \emph{conjugate} $\overline{z}:=a-\iota b$. Thus

\begin{equation}\label{eq:real_and_imaginary_parts}
2\Re(z)=z+\overline{z} \qquad \text{and} \qquad 2\iota\Im(z)=z-\overline{z}.
\end{equation}

Using conjugation we can define a (pseudo) norm $\abs{z}^2:=z\overline{z}=a^2-\iota^2 b^2$. Elements of $\mathcal{B}_\star$ are \emph{space-like}, \emph{time-like}, or \emph{light-like} if their norm squared is positive, negative, or zero. Note that if $\star=-1$, then all non-zero elements are space-like. On the other hand, if $\star=0$ then $\abs{z}^2=\Re(z)^2$ and all purely imaginary elements are light-like, while the remaining ones are space-like. Let $\mathcal{B}_\star^\times$ denote the invertible elements of $\mathcal{B}_\star$. It is easy to check that $z \in \mathcal{B}_\star^\times$ if and only if $z$ is not light-like, in which case
\begin{align}\label{eq:inv_formula}
z^{-1}=\frac{\overline{z}}{\abs{z}^2}.
\end{align}
In particular $\iota$ is not a unit in $\mathcal{B}_0$, hence one \emph{cannot} calculate $\Im(z)$ of $z\in \mathcal{B}_0$ using the second formula in \eqref{eq:real_and_imaginary_parts}.

There is a natural (injective) $\RR$--algebra homomorphism from $\mathcal{B}_\star$ to the space of real $2\times 2$ matrices, given by
$$
z\mapsto \begin{pmatrix}
\Re(z) & \iota^2\Im(z)\\
\Im(z) & \Re(z)
\end{pmatrix}.
$$
It follows that multiplication in $\mathcal{B}_\star$ can be encoded using $2\times 2$ matrices, by
\begin{align}\label{eq:B_mat_mult}
\begin{pmatrix}
\Re(z) & \iota^2\Im(z)\\
\Im(z) & \Re(z)
\end{pmatrix}
\begin{pmatrix}
\Re(w) & \iota^2\Im(w)\\
\Im(w) & \Re(w)
\end{pmatrix}=\begin{pmatrix}
\Re(zw) & \iota^2\Im(zw)\\
\Im(zw) & \Re(zw)
\end{pmatrix}, \qquad \forall z,w\in \mathcal{B}_\star.
\end{align}

A matrix with coefficients in $\mathcal{B}_\star$ is $\star$--\emph{Hermitian} if it is equal to its conjugate transpose (where here conjugate refers to conjugation in $\mathcal{B}_\star$). Let $\H_\star$ be the space of $2 \times 2$ $\star$--Hermitian matrices with coefficients in $\mathcal{B}_\star$. The space $\H_\star$ is a real $4$--dimensional vector space and $A\mapsto -\det(A)$ defines a quadratic form on $\H_\star$. Depending on $\star$, the corresponding symmetric bilinear form $D_\star$ has signature $(3,1)$ ($\star = -1$), $(2,2)$ ($\star = 1$), or it is degenerate with signature $(2,1,1)$ ($\star = 0$). Let $$
\mathcal{C}_\star := \{u\in \H_\star \mid D_\star(u,u)<0\},
$$
and let $\XX_\star:=\PP(\mathcal{C}_\star)$ be the image of $\mathcal{C}_\star$ in the projective space $\PP(\H_\star)$. The \emph{boundary} $\partial \XX_\star$ consists of projective classes of $D_\star$--isotropic vectors, namely projective classes of rank one matrices in $\PP(\H_\star)$. The space $\XX_\star \cup \partial \XX_\star$ is homeomorphic to a $3$--ball (if $\star = -1$), a solid torus (if $\star = 1$), or a ``pinched'' solid torus (if $\star = 0$). In all cases, $\XX_\star$ is orientable and we fix an orientation.

The group of orthogonal 
transformations $\mathrm{O}(D_\star)$ preserves $\mathcal{C}_\star$, thus its projectivization $\PO(D_\star)$ acts on $\XX_\star$. In particular, the pair $(\PO(D_\star),\XX_\star)$ is a $(G,X)$--\emph{geometry} in the sense of Klein~\cite{KL}. It is easy to check that
$$
(\PO(D_\star),\XX_\star) =
\begin{cases}
		(\PO(3,1),\HH^3), & \text{if} \quad \star = -1,\\
		(\PO(2,2),\ADS), & \text{if} \quad \star = 1.
\end{cases}
$$
In his thesis \cite{DANTHESIS}, J.\ Danciger defines a transitional geometry called \emph{half-pipe geometry} that captures the geometry of collapsing hyperbolic and Anti-de Sitter structures.  Half-pipe geometry turns out to be a subgeometry of $(\PO(D_0),\XX_0)$--geometry with the same model space $\XX_0$, but a strictly smaller group of symmetries. To define the symmetry group of half-pipe geometry, we consider the following construction. Let
$$
\mathcal{B}_\star\PP^1 :=\{(x,y)\in \mathcal{B}_\star\times\mathcal{B}_\star\mid (\alpha x , \alpha y) \not= (0,0), \ \forall\ \alpha\in \mathcal{B}_\star\backslash\{0\}\}_{/\sim},
$$ 
where the equivalence relation $\sim$ is scalar multiplication by $\mathcal{B}_\star$. We denote by $[x,y]$ the equivalence class containing $(x,y)$. There is an embedding $\mathcal{B}_\star \cup \{ \infty\} \rightarrow \mathcal{B}_\star\PP^1$ given by
\begin{align}\label{eq:B_star_embed}
v \mapsto [v,1], \qquad \forall v \in \mathcal{B}_\star, \qquad \text{and} \qquad \infty \mapsto [1,0].
\end{align}
When $\star=-1$ this map is a bijection and this is the standard identification $\CC\cup\{\infty\}\cong \cp^1$. Next, consider the bijection
\begin{align}\label{eq:X_boundary_identification}
	\Lambda_{\star} : \mathcal{B}_\star\PP^1 & \xrightarrow{\cong} \partial \XX_\star,\\ \nonumber
	[v_1,v_2] &\mapsto \begin{bmatrix}
	\abs{v_1}^2 & v_1\overline{v_2}\\
	\overline{v_1}v_2 & \abs{v_2}^2
	\end{bmatrix}.
\end{align}

The group $\mathrm{PSL}_2(\mathcal{B}_\star)$ acts both on $\mathcal{B}_\star\PP^1$ by linear fractional transformations, and on $\PP(\H_\star)$ via
\begin{equation}\label{eq:sl2c_action}
A\cdot G := A \ G \ \overline{A^T}, \qquad \text{for all} \quad A \in \mathrm{PSL}_2(\mathcal{B}_\star), \ G \in \PP(\H_\star).
\end{equation}
Next, let $\mathrm{C}:\mathcal{B}_\star\PP^1\to \mathcal{B}_\ast\PP^1$ be the involution given by componentwise conjugation. 

This action of both $\mathrm{PSL}_2(\mathcal{B}_\star)$ and $\mathrm{C}$ on $\PP(\H_\star)$ is faithful, $\RR$--linear, and preserves determinants. It gives rise to a subgroup of $\PO(D_\star)$ isomorphic to $\mathrm{PSL}_2(\mathcal{B}_\star)\rtimes \ZZ_2$, which we henceforth call the \emph{group of $\XX_\star$--transformations} and denote $\G_\star$. The index two subgroup of $\G_\star$ isomorphic to $\mathrm{PSL}_2(\mathcal{B}_\star)$ is called the \emph{group of orientation preserving $\XX_\star$--transformations} and is denoted $\G_\star^+$. In this paper we will refer to the $(\G_0,\XX_0)$--geometry as \emph{half-pipe geometry}, although Danciger's half-pipe geometry is modeled on $(\G_0^+,\XX_0)$.

For $\star \in \{-1,1\}$, one finds that $\G_\star=\PO(D_\star)$, thus $(\G_\star,\XX_\star)$ is isomorphic to the geometry $(\PO(D_\star),\XX_\star)$. On the other hand, we will shortly see that when $\star=0$, the group $\G_0$ is a proper subgroup of $\PO(D_0)$. 

Note that $(\G_\star,\XX_\star)$ is a subgeometry of $(\PGL(4),\rp^3)$, in the sense of \S\ref{subsec:proj_structures}. For example, there is an isomorphism of quadratic spaces
\begin{align}\label{eq:quad_space_map}
    \J_\star : \H_\star \rightarrow \RR^4, \qquad \text{where} \qquad
	\begin{pmatrix}
	x & v+\iota w\\
	v-\iota w & y
	\end{pmatrix} &\mapsto 
	\frac{1}{\sqrt{2}}\left(x,y,v,w\right),
\end{align}
between $\H_\star$ with $D_\star$ and $\RR^4$ with the \emph{standard symmetric bilinear form}
\begin{equation}\label{eq:standard_bilinear_form}
J_{\star}=\begin{pmatrix}
0 & -1 & 0 & 0\\
-1 & 0 & 0 & 0\\
0 & 0 & 2 & 0\\
0 & 0 & 0 & -2\iota^2
\end{pmatrix}.
\end{equation}
The map $\J_\star$ descends to equivariant embeddings $\G_\star \hookrightarrow \PGL(4)$ and $\XX_\star \hookrightarrow \rp^3$.

\begin{lemma}\label{lem:half_pipe_subgroup}
    For each $A\in \PO(D_0)\subset \PGL(4)$ there is a unique $k \in \RR$ and a unique $B\in \G_0$ so that $A=D_kB$, where $D_k=\mathrm{Diag}(1,1,1,e^k)$. It follows that $\G_0$ is a normal subgroup of $\PO(D_0)$ and that $\PO(D_0)/\G_0 \cong \RR$. 
\end{lemma}

\begin{proof}
Using \eqref{eq:quad_space_map} we can regard $\XX_0$ as a subset of $\rp^3$ and $\PO(D_0)$ as a subgroup of $\PGL(4)$ . The points $V_1=[\basis{1}]$, $V_2=[\basis{2}]$, and $V_3=[\basis{1}+\basis{2}+\basis{3}]$ are contained in $\partial \XX_0$. Let $W_i=A^{-1}V_i$ for $1\leq i\leq 3$. By \cite[Prop.\ $2$]{DAN2} there is a unique element $B\in \G_0^+$ mapping $V_i$ to $W_i$ for $1\leq i\leq 3$. As a result, the element $D:=AB\in \PO(D_0)$ fixes $V_i$ for $1\leq i\leq 3$.

The point $[\basis{4}]$ is the unique non--manifold point of $\partial \XX_0$ and so $D$ also fixes $[\basis{4}]$. Furthermore, since $D$ fixes $V_1$, $V_2$, and $V_3$ it preserves the plane $V_1V_2V_3$ spanned by these points. $D$ also preserves the dual hyperplanes tangent to $\partial \XX_0$ at $V_1$ and $V_2$. These planes are $[\dbasis{2}]$ and $[\dbasis{1}]$, respectively, and so $D$ preserves the line $\ell$ containing $[\basis{4}]$ and $[\basis{3}]$. It follows that $D$ fixes $[\basis{3}]=\ell\cap V_1V_2V_3$. These properties imply that $D=\mathrm{Diag}(1,1,1,t)$ for some $t\in \RR^\times$. However, if $t<0$ then $\mathrm{Diag}(1,1,1,t)=D_{\log(-t)}\mathrm{Diag}(1,1,1,-1)$ and $\mathrm{Diag}(1,1,$1$--1)\in \G_0$. Furthermore, $D_k\in \G_0$ if and only if $k=0$, and so uniqueness of the decomposition follows. 

Next, since $[\basis{4}]$ is preserved by each element of $\PO(D_\star)$, it follows that every element $B\in \G_0^+$ can be written in block form as
$$
\begin{bmatrix}
A & 0\\ 
c^T & d
\end{bmatrix},
$$
where $A\in \mathrm{SL}(3,\RR)$, $c\in \RR^3$, and $d\in \RR^\times$. Furthermore, using \eqref{eq:sl2c_action} it is easy to check that the elements of $\G_0$ are precisely those elements for which $d=\pm 1$. It follows that $D_k$ normalizes $\G_0$, and so $\G_0$ is normal in $\PO(D_\star)$ and $\PO(D_\star)/\G_0\cong \RR$. 
\end{proof}

For every $3$--manifold $M$ with ideal triangulation $\Delta$, the map $\J_\star$ induces a ``forgetful'' map
$$
\mathrm{f}_\star : \XX_\star(M;\Delta) \rightarrow \rp^3(M;\Delta)
$$
that maps a branched $(\G_\star,\XX_\star)$--structure on $M$ (with respect to $\Delta$) to its underlying branched projective structure (cf. \S\ref{subsec:proj_structures}).
\begin{lemma}\label{lem:forget_structures_almost_injective}
	The map $\mathrm{f}_\star : \XX_\star(M;\Delta) \rightarrow \rp^3(M;\Delta)$ is injective for $\star \in \{-1,1\}$. For $\star=0$, each non-empty fiber of $\mathrm{f}_0$ is homeomorphic to $\RR$. 
\end{lemma}
\begin{proof}
Using \eqref{eq:quad_space_map} we can assume that $\XX_\star\subset \rp^3$ and $\PO(D_\star) \subset \PGL(4)$. In each case, the group of projective transformations that preserves $\XX_\star$ is $\PO(D_\star)$. When $\star=\pm 1$, $\G_\star=\PO(D_\star)$, and the result follows.

When $\star=0$, suppose $\mathcal{M}_0\in \XX_0(M;\Delta)$ and let $(\dev_0,
\hol_0)$ be a representative pair for $\mathrm{f}_0(\mathcal{M}_0)$. For every $\mathcal{M}
\in \XX_0(M;\Delta)$ in the same $\mathrm{f}_0$--fiber as$\mathcal{M}_0$ , let $(\dev,
\hol)$ be a representative pair. Then there is $A\in \PO(D_0)$ such that $(\dev,\hol)=A\circ (\dev_0,\hol_0)$ (up to precomposing the developing maps with an isotopy). By Lemma~\ref{lem:half_pipe_subgroup}, there is a unique way to write $A=D_kB$, for $B\in \G_0$ and $D_k=\mathrm{Diag}(1,1,1,e^k)$, thus there is a map $g$ from the fiber to $\RR$ given by $g(\mathcal{M})=k$.  By Lemma~\ref{lem:half_pipe_subgroup}, the map $g$ is well defined and injective. To see that that $g$ is surjective, observe that if $k\in \RR$ and $\mathcal{M}_k$ is the half-pipe structure with developing map $D_{k}\circ \dev_0$, then $g(\mathcal{M}_k)=k$.
\end{proof}
{
Suppose that $\J : \H_\star \rightarrow \RR^4$ is another isomorphism of quadratic spaces between $\H_\star$ with $D_\star$ and $\RR^4$ with another quadratic form $J$, of the same signature as $J_\star$ (cf. \eqref{eq:standard_bilinear_form}). Then it is easy to check that $\J$ gives rise to equivariant embeddings $\G_\star \hookrightarrow \PGL(4)$ and $\XX_\star \hookrightarrow \rp^3$. Let $\G_\J$ and $\XX_\J$ be the respective images of these embeddings, then $(\G_\J,\XX_\J)$ is a projectively equivalent model of $(\G_\star,\XX_\star)$--geometry inside of $(\PGL(4),\rp^3)$--geometry.  The triple $(\G_\J,\XX_\J,\J)$ is a \emph{marked projective model} for $(\G_\star,\XX_\star)$--geometry. If the marking if forgotten, then $(\G_\J,\XX_\J)$ is an \emph{unmarked projective model} for $(\G_\star,\XX_\star)$--geometry. We will generally not specify if a model is marked or unmarked when it is obvious from the context. There are two important notions of equivalence for marked projective models that we now discuss. Two marked projective models $(\G_\J,\XX_\J,{\J})$ and $(\G_{\J'},\XX_{\J'},\J')$ are \emph{projectively equivalent} if there is a transformation $A \in \PGL(4)$ such that 
$$
(\G_{\J'},\XX_{\J'},\J' ) = (A \cdot \G_\J \cdot A^{-1}, A \cdot \XX_\J,A\circ {\J})
$$
and are $(\G_\star,\XX_\star)$--equivalent if $(\G_{\J},\XX_\J)=(\G_{\J'},\XX_{\J'})$ and there is $B\in \G_\star$ so that $\J'=\J\circ B$.
Since quadratic forms of the same signature are isometric it follows that different marked projective models of $(\G_\star,\XX_\star)$--geometry are projectively equivalent. The following lemma describes when different marked models are $(\G_\star,\XX_\star)$--equivalent.
}
\begin{lemma}\label{lem:res_for_marked_models}
Let $(\G_\J,\XX_\J,\J)$ and $(\G_{\J'},\XX_{\J'},\J')$ be two marked projective models for $(\G_\star,\XX_\star)$--geometry and let $A\in \PGL(4)$ be a projective isomorphism between the marked models $(\G_{\J},\XX_\J,\J)$ and $(\G_{\J'},\XX_{\J'},\J')$. Then there is a map $B:\XX_\star\to \XX_\star$ that makes the following diagram commute
$$\begin{tikzcd}
(\G_\star,\XX_\star)\ar[d,"B"'] \ar[r,"\J"]& (\G_{\J},\XX_\J)\ar[d,"A"]\\
(\G_\star,\XX_\star)\ar[r,"\J'"]& (\G_{\J'},\XX_{\J'})
\end{tikzcd}
$$

Moreover, if $\star=\pm1$, then $B\in \G_\star$. If $\star=0$, then there is $k\in \RR$ so that $B=C_kB'$, where $B'\in\G_0$ and $C_k$ is given by 
$$\begin{bmatrix}
x & v+\iota w\\
v-\iota w &y
\end{bmatrix}\mapsto
\begin{bmatrix}
x & v+\iota e^k w\\
v-\iota e^k w &y
\end{bmatrix}.
$$
\end{lemma}
\begin{proof}
When $\star=\pm1$ the lemma is a consequence of the fact that $\G_\star=\PO(D_\star)$. When $\star=0$, the result follows from Lemma \ref{lem:half_pipe_subgroup} and thinking about how the map $D_k$ act on $\XX_\star$.
\end{proof}

{
An immediate consequence of Lemma~\ref{lem:res_for_marked_models} is that
if $\star=\pm 1$ then any two marked projective models are $(\G_\star,\XX_\star)$--equivalent. However, if $\star = 0$, not there are marked projective models that are not $(\G_\star,\XX_\star)$--equivalent. This fact will have the consequence that Thurston's parameters are not projectively invariant for this definition of half-pipe geometry (see Remark~\ref{rem:thu_param_invt} for more details).
}

\begin{remark}\label{rem:hyp_is_strictly_convex}
	For every projective model $(\G_\J,\XX_\J)$ for hyperbolic geometry, there is an affine patch $\A$ of $\rp^3$ where $\XX_\J$ is the unit ball in $\A$. In particular, every projective model for hyperbolic space is \emph{strictly convex} in $\rp^3$, namely for each $V\in \partial \XX_\J$, every tangent plane to $\partial \XX_\J$ at $V$ locally intersects $\Sigma$ only at $V$. This is a special case of proper convexity, which will be discussed in \S\ref{sec:prop_convex_projective_st}.
\end{remark}


\subsection{Thurston's parameters and equations}\label{subsec:thu_equations}

In~\cite{THUNOTES}, Thurston uses ideal triangulations to construct (branched) hyperbolic structures from solutions to a collection of complex equations. This technique was extended to Anti-de Sitter space and half-pipe space by Danciger in~\cite{DAN2}. In this section we recall this construction, adapted to the framework of real projective structures introduced in \S\ref{subsec:X_spaces}.

A \emph{projective tetrahedron} is a region in $\rp^3$ that is projectively equivalent to the projectivization of the positive orthant in $\RR^4$. Note that a projective tetrahedron is the convex hull of its vertices in any affine patch that fully contains it. Let $\star \in \{-1,1,0\}$ and let $(\G_\J,\XX_\J)$ be a projective model for the $(\G_\star,\XX_\star)$--geometry. An \emph{ideal $\XX_\J$--tetrahedron} is a projective tetrahedron whose vertices are contained in $\partial \XX_\J$, and whose interior is contained in $\XX_\J$. This notion is projectively invariant, therefore it does not depend on the projective model. In particular we can define an \emph{ideal $\XX_{\star}$--tetrahedron} as a region of $\XX_{\star}$ that corresponds to an ideal $\XX_\J$--tetrahedron in some, and hence all, of the projective models.
\begin{remark}\label{rem:X_tet_in_different_spaces}
	If $\XX_{\star}$ is projectively equivalent to $\HH^3$, then four points in $\partial \XX_{\star}$ are always the vertices of a unique ideal $\XX_{\star}$--tetrahedron. This is a consequence of the fact that every projective model for $\HH^3$ is strictly convex in $\rp^3$ (cf. Remark~\ref{rem:hyp_is_strictly_convex}). On the other hand, in Anti-de Sitter space and half-pipe space there are quadruples of points on the boundary that are not the vertices of any ideal tetrahedron, in the respective geometries (cf. Lemma~\ref{lem:thurston_parameters}).
\end{remark}
Let $\tet \subset \XX_{\star}$ be an ideal $\XX_{\star}$--tetrahedron, with a fixed ordering of its vertices $(V_1,V_2,V_3,V_4)$, and let $\sigma = (ij)k \in \edgeface$. Henceforth we use the identification $\mathcal{B}_\star \cup \{\infty\} \hookrightarrow \mathcal{B}_\star\PP^1 \cong \partial \XX_{\star}$ described in~\eqref{eq:X_boundary_identification}, to identify the vertices of $\tet$ with elements in $\mathcal{B}_\star\PP^1$. Since $\mathrm{PSL}_2(\mathcal{B}_\star)$ acts simply transitively on ordered triples of points of $\mathcal{B}_\star\PP^1$, there is a unique element of $\G_\star^+$ mapping
$$
V_i \mapsto \infty = \Lambda_{\star}^{-1}\left(\begin{bmatrix}1 & 0\\ 0 & 0\end{bmatrix}\right) ,\quad  V_j \mapsto 0= \Lambda_{\star}^{-1}\left(\begin{bmatrix}0 & 0\\ 0 & 1\end{bmatrix}\right) , \quad \text{and} \quad V_k \mapsto 1=\Lambda_{\star}^{-1}\left(\begin{bmatrix}1 & 1\\ 1 & 1\end{bmatrix}\right) .
$$
Under this transformation, the last vertex $V_l$ is mapped to some $z_{\sigma}^\tet \in \mathcal{B}_\star\setminus\{0,1\}  \subset \mathcal{B}_\star\PP^1$. The number $z_\sigma^\tet$ is called the \emph{Thurston's parameter} of $\tet$ with respect to $\sigma$. We recall that an ideal $\XX_{\star}$--tetrahedron $\tet$ with ordered vertices $(\infty,0,1,z)$ is \emph{positively oriented} if the orientation induced on $\tet$ by the ordering of its vertices agrees with the orientation on $\XX_{\star}$. By construction, $z_\sigma^\tet$ is a $\G_\star^+$--invariant of $\tet$, but if $\star=0$ then it is not a $\PO(D_0)$--invariant of $\tet$ (see Remark~\ref{rem:thu_param_invt} for more details). 

Each ideal $\XX_\star$--tetrahedron has twelve Thurston's parameters, one for each edge-face, but they are not independent. They obey to the following \emph{internal relations}:
\begin{align}\label{eq:Thurston_param_relations}
z_\sigma=z_{\overline{\sigma}},\qquad
z_{\sigma}=z_{\op{\sigma}}, \qquad \text{and} \qquad
z_{\sigma_+}=\frac{1}{1-z_{\sigma}}, \qquad \forall \sigma \in \edgeface.
\end{align}
On the other hand, not every element of $\mathcal{B}_\star\PP^1$ is the Thurston's parameter of an $\XX_\star$--tetrahedron.
\begin{lemma}[\cite{DANTHESIS} Prop.\ 33]\label{lem:thurston_parameters}
	Let $\star \in \{-1,1,0\}$. An element $z \in \mathcal{B}_\star\subset \mathcal{B}_\star\PP^1$ is the Thurston's parameter of an ideal $\XX_{\star}$--tetrahedron if and only if $z$ and $1-z$ are space-like. In that case, the ideal $\XX_{\star}$--tetrahedron with vertices $(\infty,0,1,z)$ is positively oriented if and only if $\Im(z)>0$.
\end{lemma}

 Now let $M$ be an orientable $3$--manifold with an ideal triangulation $\Delta$. We recall that $\Tet(\Delta)$ and $\Edge(\Delta)$ are the set of tetrahedra and edges of $\Delta$. We denote by $\mathcal{B}_\star^{\Tet(\Delta)\times \edgeface}$ the set of functions $z^{\ast}_{\ast} : \Tet(\Delta)\times \edgeface \rightarrow \mathcal{B}_\star$ and by $z_\sigma^\tet := z^{\ast}_{\ast}(\tet,\sigma)$. The \emph{$\XX_{\star}$--deformation space} of $M$ (with respect to $\Delta$) is the set $\mathcal{D}_{\XX_{\star}}(M;\Delta)$ of points $z^{\ast}_{\ast} \in \mathcal{B}_\star^{\Tet(\Delta)\times \edgeface}$ such that:
 \begin{enumerate}
 	\item\label{item:X_defspace_cond_1} for all $\tet\in \Tet(\Delta)$ and all $\sigma \in \edgeface$, the numbers $z_\sigma^\tet$ and $1-z_\sigma^\tet$ are space-like, and $\Im(z_\sigma^\tet)>0$;
 	\item\label{item:X_defspace_cond_2} for all $\tet\in \Tet(\Delta)$, the numbers $\{z_\sigma^\tet\} _{\sigma \in \edgeface}$ satisfy the internal relations~\eqref{eq:Thurston_param_relations};
 	\item\label{item:X_defspace_cond_3} for every edge $s\in \Edge(\Delta)$, fix once and for all an orientation on $s$, and define $\Tet(\Delta)_s$ to be the cyclically ordered ${k_s}$--tuple $\cyclic{\tet_1^s,\ldots, \tet_{k_s}^s}$ of tetrahedra that abut $s$, cyclically ordered to follow the ``right hand rule'' by placing the thumb in the direction of $s$. Let $\sigma_i^s$ be the outgoing edge-face of $\tet_i^s$ around $s$. Then $z^{\ast}_{\ast} $ satisfies
 	\begin{equation}\label{eq:Thurston_gluing_equation}
 	\prod_{i=1}^{k_s}z^{\tet_i}_{\sigma_i}=1.
 	\end{equation}
 \end{enumerate}
Equation~\eqref{eq:Thurston_gluing_equation} is called the \emph{Thurston's gluing equation} of $s$.

Intuitively, \eqref{item:X_defspace_cond_1} and \eqref{item:X_defspace_cond_2} ensure that the parameters arise as the Thurston's parameters of positively oriented ideal $\XX_\star$--tetrahedra. As $\mathrm{PSL}_2(\mathcal{B}_\star)$ acts simply transitively on ordered triples of points of $\mathcal{B}_\star\PP^1$, there is always a unique orientation preserving $\XX_\star$--transformation that realizes each face pairing.
 
Finally, \eqref{item:X_defspace_cond_3} ensures that, when we glue all $\XX_\star$--tetrahedra around an edge, the last and the first match to close up the cycle (with possible branching along the edge). A more formal analysis of this discussion leads to the following result. It is a consequence of the work of Thurston~\cite{THUNOTES} in the hyperbolic setting, and of Danciger~\cite{DAN2} for Anti-de Sitter and half-pipe settings.
\begin{lemma}[\cite{THUNOTES},\S$3$ \cite{DAN2}]\label{lem:develop_X_structures}
	Let $\star \in \{-1,1,0\}$. For every $z^{\ast}_{\ast} \in \mathcal{D}_{\XX_{\star}}(M;\Delta)$ there is a unique branched $(\G_\star,\XX_\star)$--structure on $M$ (with respect to $\Delta$) made up of ideal $\XX_\star$--tetrahedra, with Thurston's parameters $z^{\ast}_{\ast}$. In particular, this induces a continuous injective map
	$$
	\mathrm{Ext}_{\XX_\star} : \mathcal{D}_{\XX_{\star}}(M;\Delta) \rightarrow \XX_\star(M;\Delta).
	$$
\end{lemma}
 
 
\subsection{Hyperbolic, Anti-de Sitter and half-pipe tetrahedra of flags}\label{subsec:X_tet_of_flags}

Let $\star \in \{-1,1,0\}$ and let $(\G_\J,\XX_\J, \J)$ be a marked projective model for the $(\G_\star,\XX_\star)$--geometry. For each ideal $\XX_\J$--tetrahedron $\tet$, there is a preferred tetrahedron of flags $\F_\tet =(V_m,\eta_m)_{m=1}^4$ associated to it, where
\begin{enumerate}
    \item $\{V_m\}_{m=1}^4$ is the set of vertices of $\tet$;
    \item each $\eta_m$ is the unique plane tangent to $\XX_\J$ at $V_m$.
\end{enumerate}
{
Conversely, we say that a tetrahedron of flags $\F=(V_m,\eta_m)_{m=1}^4$ is an \emph{$\XX_\star$--tetrahedron of flags} if there is a marked projective model $(\G_\J,\XX_\J, \J)$ and an ideal $\XX_\J$--tetrahedron whose associated tetrahedron of flags is $\F$. In that case we say that $(\G_\J,\XX_\J, \J)$ is a \emph{marked osculating model} for $\F$. If $\F$ is an $\XX_\star$--tetrahedron of flags with osculating model $(\G_\J,\XX_\J, \J)$, then we can use $\J$ to pull the vertices of $\F$ (i.e.\ the points of each of the flags) back to $\partial \XX_\star$. Thus given an edge-face $\sigma$ we can compute the \emph{Thurston's parameter (with respect to $\J$)} of this tetrahedron of flags, which we denote $z_{\sigma}^{\F,\J}$. 
}
\begin{remark}\label{rem:thu_param_invt}
By Lemma \ref{lem:res_for_marked_models}, when $\star=\pm 1$, the Thurston's parameters are independent of the marking and are hence projective invariants. When $\star=0$ different markings give rise to Thurston's parameters with the same real part, but different imaginary parts, and so the real part is a projective invariant, but the imaginary part is not. 

Furthermore, Lemma \ref{lem:res_for_marked_models} also shows that for half-pipe tetrahedra the imaginary parts of Thurston's parameters are uniformly distorted by different markings. Specifically, if $\E$ and $\F$ are $\XX_0$--tetrahedra with osculating models $(\G,\X, \J)$ and $(\G,\X, \J')$, then
$$
\frac{\Im(z_\sigma^{\F,\J})}{\Im(z_{\tau}^{\E,\J})}=\frac{\Im(z_\sigma^{\F,\J'})}{\Im(z_{\tau}^{\E,\J'})}, \qquad \forall \sigma,\tau\in \edgeface.
$$
\end{remark}

The next result shows that determining when a tetrahedron of flags is an $\XX_\star$--tetrahedron of flags is simple, using edge ratios and triple ratios. We recall that given a tetrahedron of flags $\F=(V_m,\eta_m)_{m=1}^4$ and an edge-face $\sigma$, we denote by $e_\sigma^\F$ (resp. $t_\sigma^\F$) the edge ratio (resp. triple ratio) of $\F$ with respect to $\sigma$ (cf.~\S\ref{subsec:triple_edge_ratios}). We also recall that
$$
    X_\sigma^\F := \frac{\mu_{\sigma}^\F t_\sigma^\F e_\sigma^\F}{t_{\sigma}^\F+1} \quad \text{ and } \quad  Y_{\sigma}=\frac{t_{\sigma}^\F+1}{t_{\sigma}^\F \mu_{\sigma}^\F}, \qquad \text{ where } \qquad \mu_\sigma^\F := e^\F_{\sigma_-}e^\F_{\sigma_+} -e_{\sigma_-}^\F + 1,
$$
and further define
\begin{equation}\label{eq:j_definition}
    j_\sigma^\F:= e_\sigma^\F-(X_\sigma^\F)^2.
\end{equation}

\begin{theorem}\label{thm:X_tets_of_flags}
Let $\star \in \{-1,1,0\}$ and let $\F=(V_m,\eta_m)_{m=1}^4$ be a tetrahedron of flags. Then the following are equivalent:
\begin{enumerate}
    \item \label{item:tet_of_flag_is_X_1} $\F$ is an $\XX_\star$--tetrahedron of flags; 
    \item \label{item:tet_of_flag_is_X_2} $t_\sigma^\F=1$ for all $\sigma\in \edgeface$; 
    \item \label{item:tet_of_flag_is_X_3} $e_{\sigma}^\F=e_{\op{\sigma}}^\F$ for all $\sigma\in \edgeface$.
\end{enumerate}
Furthermore, the osculating model of an $\XX_\star$--tetrahedron of flags is hyperbolic, Anti-de Sitter, or half pipe if $j_\sigma^\F>0$, $j_\sigma^\F<0$, or $j_\sigma^\F=0$, respectively. In each case, if $z_\sigma^\F$ is a Thurston's parameter (with respect to some marking) then we have the following relations:
\begin{equation}\label{eq:edge_ratios_thurston_param}
    e_\sigma^\F=\abs{z_\sigma^\F}^2, \qquad X_\sigma^\F=\Re(z_\sigma^\F), \qquad \text{and} \qquad j_\sigma = -\iota^2 \Im(z_\sigma^\F)^2, \qquad \forall \sigma\in \edgeface.
\end{equation}
\end{theorem}

\begin{proof}
Before embarking on the proof we observe that if the relations \eqref{eq:edge_ratios_thurston_param} are satisfied for the Thurston's parameters with respect to some marking then they are satisfied for the Thurston's parameters with respect to any other marking. This is a consequence of Remark \ref{rem:thu_param_invt} and the fact that the last relation is trivially satisfied when $\star=0$. As a result we are free to compute Thurston's parameters with respect to any marking we wish when proving the theorem.

Henceforth we are going to drop the superscript in the parameters of $\F$.

(\eqref{item:tet_of_flag_is_X_2} $\Leftrightarrow$ \eqref{item:tet_of_flag_is_X_3}): This is a consequence of the internal consistency equations (cf. Lemma~\ref{lem:internal_eqs}) and the relations
$$
\frac{e_{\sigma}}{e_{\op{\sigma}}}=
\frac{e_{\sigma}e_{\sigma_-}e_{\sigma_+}}{e_{\op{\sigma}}e_{\sigma_-}e_{\sigma_+}}=
\frac{e_{\sigma}e_{\sigma_-}e_{\sigma_+}}{e_{\overline{\op{\sigma}}}e_{\overline{\sigma_-}}e_{\sigma_+}}=
\frac{1}{t_{\sigma}t_{\overline{\sigma}}}, \qquad \text{and} \qquad t_\sigma^2=\frac{e_{\op{\sigma}}e_{\op{(\sigma_+)}}e_{\op{(\sigma_-)}}}{e_{\sigma}e_{\sigma_+}e_{\sigma_-}}.
$$

(\eqref{item:tet_of_flag_is_X_1} $\Leftrightarrow$ \eqref{item:tet_of_flag_is_X_2} $+$  \eqref{item:tet_of_flag_is_X_3}): We begin by noticing that the property of being an $\XX_\star$--tetrahedron of flags is $\PGL(4)$--invariant, thus we can assume without loss of generality that $\F$ is in $\sigma$--standard position for any $\sigma = (ij)k \in \edgeface$. Henceforth suppose $\F$ is in $\sigma$--standard position (cf. Lemma~\ref{lem:std_pos_tetrahedron}) and fix $\delta \in \RR_{>0}$. We consider the isomorphism of quadratic spaces
\begin{align}\label{eq:osculating_model}
    \J_{\delta} : \H_\star \rightarrow \RR^4, \qquad \text{where} \qquad
	\begin{pmatrix}
	x & v+\iota w\\
	v-\iota w & y
	\end{pmatrix} &\mapsto 
	\frac{1}{\sqrt{2}}\left(x,y,v,-\frac{w}{\delta}\right),
\end{align}
between $\H_\star$ with $D_\star$ and $\RR^4$ with the symmetric bilinear form
\begin{align}\label{eq:osc_form}
J_{\delta}=\begin{pmatrix}
0 & -1 & 0 & 0\\
-1 & 0 & 0 & 0\\
0 & 0 & 2 & 0\\
0 & 0 & 0 & -2 \iota^2 \delta^2
\end{pmatrix}.
\end{align}
It is easy to check that the underlying projective quadric $\partial \XX_{\J_{\delta}}$ of $\rp^3$ corresponding to $J_\sigma$ contains the points $V_i, V_j,$ and $V_k$ and is tangent to the planes $\eta_i$ and $\eta_j$ at $V_i$ and $V_j$. The point $V_l \in \partial \XX_{\J_\delta}$ if and only if
\begin{equation}\label{eq:delta_sigma_j_sigma}
-2e_\sigma + 2 X_\sigma^2 - 2 \iota^2 \delta^2 = 0 \qquad \Longleftrightarrow \qquad -\iota^2 \delta^2 = j_\sigma. 
\end{equation}
We remark that this condition is satisfied in the half-pipe setting if and only if $j_\sigma=0$, while in the other cases we get the condition that $\delta = \sqrt{\abs{j_\sigma}}$ 

Furthermore, the planes tangent to $\partial \XX_{\J_\delta}$ at $V_3$ and $V_4$ are
$$
[\dbasis{1} + \dbasis{2} - 2 \cdot \dbasis{3}], \qquad \text{and} \qquad \left[\frac{1}{e_\sigma} \cdot \dbasis{1} + \dbasis{2} - \frac{2X_\sigma}{e_\sigma} \cdot \dbasis{3} -  \frac{2j_\sigma}{e_\sigma} \cdot \basis{4}^*\right].
$$
From Lemma \ref{lem:std_pos_tetrahedron} it follows that $\F$ is an $\XX_\star$--tetrahedron of flags, with osculating model $(\G_{\J_\delta},\XX_{\J_\delta},\J_\delta)$, if and only if
$$
t_\sigma = 1, \qquad e_{\overline{\sigma}_-}e_{\overline{\sigma}_+} = \frac{1}{e_\sigma}, \qquad  \mu_{\overline{\sigma}} = \frac{2X_\sigma}{e_\sigma}, \qquad \text{and} \qquad  \mu_{\overline{\sigma}}\left( Y_{\overline{\sigma}} - X_\sigma \right) = \frac{2j_\sigma}{e_\sigma}.
$$
Using the internal consistency equations (cf. Lemma~\ref{lem:internal_eqs}), one can check that these equations are equivalent to
$$
t_\sigma = 1, \qquad t_{\overline{\sigma}} = 1, \qquad \text{and} \qquad  \mu_{\sigma}=\mu_{\overline{\sigma}},
$$
which, in turn, are equivalent to statements \eqref{item:tet_of_flag_is_X_2} $+$  \eqref{item:tet_of_flag_is_X_3}, once one observes that \eqref{item:tet_of_flag_is_X_2} implies that $e_{\overline{\sigma}_-}=e_{\sigma_-}$. 

For the final part of the proof, suppose that $\F$ is an $\XX_\star$--tetrahedron of flags, with osculating model $(\G_{\J},\XX_{\J},\J)$. Once again we can assume that $\F$ is in $\sigma$--standard position, and by the above discussion, we may take $\J = \J_\delta$. Then the fact that $\XX_{\J_\delta}$ is hyperbolic, Anti-de Sitter, or half pipe if $j_\sigma>0$, $j_\sigma<0$, or $j_\sigma=0$, follows directly from the form of $J_\delta$ in \eqref{eq:osc_form}.

Furthermore, one can check that
$$
V_i = \left[\J_\delta ( \Lambda_\star(\infty) )
\right], \qquad V_j = \left[\J_\delta ( \Lambda_\star(0) ) 
\right], \qquad \text{and} \qquad V_k = \left[\J_\delta ( \Lambda_\star(1) ) 
\right].
$$
Hence, by definition of Thurston's parameter,
$$
[e_{\sigma} \cdot \basis{1} + \basis{2} + X_\sigma \cdot \basis{3} - \basis{4}] = V_l = \left[\J_\sigma ( \Lambda_\star(z_\sigma) ) \right] = \left[|z_\sigma|^2 \cdot \basis{1} + \basis{2} + \Re(z_\sigma) \cdot \basis{3} - \frac{\Im(z_\sigma)}{\delta} \cdot \basis{4} \right],
$$
and the last part of the lemma follows.
\end{proof}

\begin{remark}\label{rem:delta_sigma_j_sigma}
{
In the above proof we remarked that equation~\eqref{eq:delta_sigma_j_sigma} is always satisfied in the half-pipe setting, while in the other cases we get the condition that $\delta_\sigma = \sqrt{\abs{j_\sigma}}$. This has the following implications. For $\star \in \{-1,1\}$, if $\F$ is an $\XX_\star$--tetrahedron of flags in $\sigma$--standard position, then there is a unique (marked) osculating model $(\G_{\J_\sigma}, \XX_{\J_\sigma},\J_\sigma)$ for $\F$, obtained by replacing $\delta_\sigma = \sqrt{\abs{j_\sigma}}$ in \eqref{eq:osculating_model}. In particular, there is a unique Thurston's parameter $z_\sigma^\F$ associated to $\F$ which is the Thurston's parameter of the underlying ideal $\XX_{\J_\sigma}$--tetrahedron, with respect to $\sigma$. This can be explicitly computed from the edge ratios using equations~\eqref{eq:edge_ratios_thurston_param}.
}

{
On the other hand, for $\star = 0$, if $\F$ is an $\XX_\star$--tetrahedron of flags in $\sigma$--standard position, then there is a unique unmarked osculating model but a $1$--real parameter family of marked osculating models $(\G_{\J_\sigma}, \XX_{\J_\sigma},\J_\sigma)$, for all $\delta_\sigma \in \RR_{>0}$. Consequentially, for each of these models there is a Thurston's parameter $z_\sigma^\F$ that can be associated to $\F$. The real part of $z_\sigma^\F$ (which is equal to the norm) is invariant and can be determined by equations~\eqref{eq:edge_ratios_thurston_param}, while the imaginary part depends on the model, according to the relation $\Im(z_\sigma^\F) = \delta_\sigma$.
}
\end{remark}

\begin{remark}\label{rem:j_sigma_geom}
If $\sigma=(ij)k$, then the parameter $j_\sigma$ has a geometric interpretation. Let $P^*$ be the pencil of lines of $\rp^3$ through the point of intersection $P := \eta_i\eta_j\eta_k$. Then $P^*$ can be identified with the projective plane $\eta := V_i V_j V_k$ by sending each line to its intersection with $\eta$. The restriction of the form $J_\sigma$ with this plane (for any possible value of $\delta_\sigma$) has signature $(2,1)$. The fourth point $V_l$ determines a line in the pencil $P^*$ and $j_\sigma$ is positive, negative, or zero, depending on whether the corresponding point in $\eta$ is time-like, space-like, or light like (see Figure~\ref{fig:inscribed_quadric}). 
\end{remark}

\begin{figure}[t]
  \centering
  \def\svgscale{.4}
  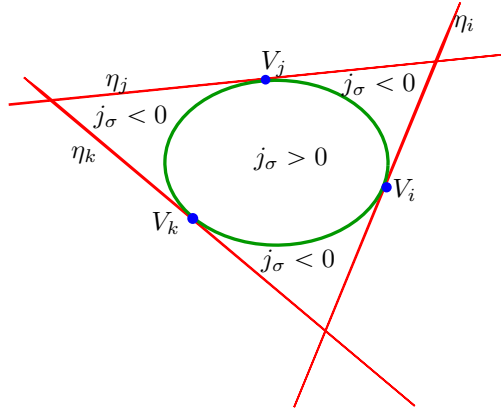
  \caption{A geometric interpretation of the parameter $j_\sigma$.}\label{fig:inscribed_quadric}
\end{figure}

Next, we turn our attention to gluings of $\XX_\star$--tetrahedra of flags. Given any two $\XX_\star$--tetrahedra of flags $\F$ and $\E$, and edge-faces $\sigma,\tau\in \edgeface$, it follows from Theorem~\ref{thm:X_tets_of_flags} that $[\F]$ and $[\E]$ are $(\sigma,\tau)$--glueable since all triple ratios are equal to $1$. On the other hand, even when $\E$ and $\F$ are in $(\sigma,\tau)$--standard glued position, the osculating model for $\F$ is typically different from the osculating model for $\E$. In other words, while the individual tetrahedra of flags can be inscribed in a projective model of $(\G_\star,\XX_\star)$, the glued tetrahedra might not be inscribed in a \emph{common} model of $(\G_\star,\XX_\star)$. The following theorem shows that there is a pair of values for the gluing parameter of $(\F,\E)$ for which the two osculating models coincide, one of which corresponds to a geometric gluing.

%
%
%
%
%
%
{
\begin{theorem}\label{thm:X_gluing_param}
Let $\star \in\{-1,1,0\}$. Let $\sigma,\tau\in \edgeface$, and let $\F$ and $\E$ be two $\XX_\star$--tetrahedra of flags in $(\sigma,\tau)$--glued position. Consider a marked osculating model for each of them, and let $z_\sigma^\F$ (resp. $z_\tau^\E$) be the corresponding Thurston's parameter of $\F$ with respect to $\sigma$ (resp. of $\E$ with respect to $\tau$). Then the marked osculating model for $\F$ is equal to the marked osculating model for $\E$ if and only if
\begin{align}\label{eq:X_gluing_param}
g_{\F,\sigma}^{\E,\tau}= \pm \frac{\Im(z_{\tau}^\E)}{\Im(z_{\sigma}^\F)}.
\end{align}
Furthermore, the gluing is geometric if and only if $g_{\F,\sigma}^{\E,\tau} =  \frac{\Im(z_{\tau}^\E)}{\Im(z_{\sigma}^\F)}$.
\end{theorem}
}

\begin{proof}
Once again, without loss of generality we can assume that $\F$ and $\E$ are in $(\sigma,\tau)$--standard glued position. Let $\sigma = (ij)k$ and $\tau = (i'j')k'$. For each $\XX_\star$--tetrahedron of flags we fix a marked osculating model. Let $z_\sigma^\F,z_\tau^\E$ be the Thurston's parameters of $\F$ and $\E$, with respect to $\sigma$ and $\tau$ in these models. We recall that each model is unique in the hyperbolic and Anti-de Sitter setting, while there is a $1$--parameter family of choices for half-pipe (cf. Remark~\ref{rem:delta_sigma_j_sigma}). Since $\F$ is in $\sigma$--standard position, the marked osculating model is $(\G_{\J_\sigma^\F}, \XX_{\J_\sigma^\F},\J_\sigma^\F)$, as defined in~\eqref{eq:osculating_model}, for some non-zero $\delta_\sigma^\F = \Im(z_\sigma^\F)$. If $\E = (W_m,\zeta_m)_{m=1}^4$, then
\begin{align}\label{eq:glue_pos_pt}
    W_{l'}=\left[\basis{1}+\abs{z_\tau^\E}^2 \cdot \basis{2}+\Re(z_\tau^\E)\cdot \basis{3}+g_{\F,\sigma}^{\E,\tau} \cdot \basis{4}\right].
\end{align}
To compare the osculating models, we consider the unique $\XX_{\J_\sigma^\F}$--tetrahedron of flags $\E^0 := (W^0_m,\zeta^0_m)_{m=1}^4$, in $\sigma$--standard position, with Thurston's parameter $z_\tau^\E$. Then
$$
W^0_l = \left[|z_\tau^\E|^2 \cdot \basis{1} + \basis{2} + \Re(z_\tau^\E) \cdot \basis{3} - \frac{\Im(z_\tau^\E)}{\Im(z_\sigma^\F)} \cdot \basis{4} \right].
$$
 of both $(\G_{\J_\sigma^\F}, \XX_{\J_\sigma^\F})$ and $(\G_{\J_\tau^\G}, \XX_{\J_\tau^\G})$
There are exactly two transformations in $\G_{\J_\sigma^\F} \cap \G_{\J_\tau^\G}$ that map
$$
(W^0_i,\zeta^0_i) \mapsto (W_{j'},\zeta_{j'}), \qquad (W^0_j,\zeta^0_j) \mapsto (W_{i'},\zeta_{i'}), \qquad \text{and} \qquad (W^0_k,\zeta^0_k)\mapsto (W_{k'},\zeta_{k'}).
$$
They are
$$
G^-:=\begin{bmatrix}
0 & 1 & 0 & 0\\
1 & 0 & 0 & 0\\
0 & 0 & 1 & 0\\
0 & 0 & 0 & 1
\end{bmatrix}, \qquad \text{and} \qquad G^+:=\begin{bmatrix}
0 & 1 & 0 & 0\\
1 & 0 & 0 & 0\\
0 & 0 & 1 & 0\\
0 & 0 & 0 & -1
\end{bmatrix}.
$$
It follows that $(\G_{\J_\sigma^\F}, \XX_{\J_\sigma^\F},\J_\sigma^\F) = (\G_{\J_\tau^\E}, \XX_{\J_\tau^\E},\J_\tau^\E)$ if and only if either
$$
G^+ \cdot (W'_4,\zeta'_4) = (W_{l'},\zeta'_{l'}), \qquad \text{or} \qquad G^- \cdot (W'_4,\zeta'_4) = (W_{l'},\zeta'_{l'}).
$$
This is equivalent to~\eqref{eq:X_gluing_param}. The final statement follows from the fact that $G^+$ is orientation preserving, while $G^-$ is orientation reversing.
\end{proof}


\subsection{Proof of Theorem~\ref{thm:main_thm}}\label{subsec:proof_main_thm}

We are now ready to construct the map $\varphi_\star : \mathcal{D}_{\XX_{\star}}(M;\Delta) \rightarrow \mathcal{D}_\rp(M;\Delta)$ mentioned in the statement of Theorem~\ref{thm:main_thm}. First, we recall that $\mathcal{D}_{\Delta} = \mathcal{D}_\rp(M;\Delta)$ is a subset of
$$
Y := \left(\left( \RR_{>0}^{\edgeface} \times\RR_{>0}^{\edgeface} \right) \times\RR_{>0}^{\edgeface} \right)^{\Tet(\Delta)},
$$
and that we can describe every $x \in Y$ as
$$
x(\tet) = \left(x_1(\tet),x_2(\tet)\right) = \left( (\tau_*^\tet,\epsilon_*^\tet),\kappa_*^\tet \right) \in \left( \RR_{>0}^{\edgeface} \times\RR_{>0}^{\edgeface} \right) \times\RR_{>0}^{\edgeface}, \qquad \forall \tet \in \Tet(\Delta).
$$
For $\star \in \{-1,1,0\}$, we define a map $\varphi_\star : \mathcal{D}_{\XX_{\star}}(M;\Delta) \rightarrow Y$. For every $z_\ast^\ast \in \mathcal{D}_{\XX_{\star}}(M;\Delta)$ and every $\tet \in \Tet(\Delta)$, let
$$
\varphi_\star(z_\ast^\ast)(\tet) := \left( (\tau_*^\tet,\epsilon_*^\tet),\kappa_*^\tet \right), \qquad \text{such that} \qquad
\tau_\sigma^\tet =1, \quad \epsilon_\sigma^\tet =\abs{z^\tet_\sigma}^2, \quad \text{and} \quad \kappa_\sigma^\tet =  \frac{\Im(z^\tet_\sigma)}{\Im(z^{\tet'}_{\tau})},
$$
where $\tet'$ is the unique tetrahedron glued to $\tet$ along the edge-face pair $(\sigma,\tau)$. We claim that the image of $\varphi_\star$ is contained in $\mathcal{D}_{\Delta}$, and that it can be explicitly described. We recall that the coordinates of points in $\mathcal{D}_{\Delta}$ are edge ratios, triple ratios and gluing parameters of some triangulation of flags, thus we henceforth refer to them with their usual symbols $e_\sigma^\tet,t_\sigma^\tet$ and $g_\sigma^\tet$. Let $\mathcal{S}_\star(M;\Delta) \subset \mathcal{D}_{\Delta}$ be the subset of points satisfying the following additional equations
\begin{enumerate}
    \item for each $\tet\in \Tet(\Delta)$ and each $\sigma \in \edgeface$,
    $$
    e_\sigma^\tet=e_{\op{\sigma}}^\tet;
    $$
    \item \label{item:j_sigma_and_gluing_equation} for each $\tet,\tet'\in \Tet(\Delta)$ glued along the edge-face pair $(\sigma,\tau)$ (see \eqref{eq:j_definition} for the definition),
    $$
    j_\sigma^\tet=(g_\sigma^\tet)^2 j_\tau^{\tet'};
    $$ 
    \item for each $\tet\in \Tet(\Delta)$ and each $\sigma \in \edgeface$,
    $$
    j_\sigma^\tet
    \begin{cases}
        >0, & \quad \text{if} \quad \star = -1,\\
        <0, & \quad \text{if} \quad \star = 1,\\
        =0, & \quad \text{if} \quad \star = 0.
        \end{cases}
    $$
\end{enumerate}
\begin{remark}\label{rem:sigma_j_relations}
    For every $x \in \mathcal{S}_\star(M;\Delta)$, the coordinates of $x$ satisfy the additional relations
    \begin{equation}\label{eq:thurston_additional_relations}
        j_{\sigma_+}^\tet = e_{\sigma_+}^2 j_{\sigma}^\tet, \qquad \forall \tet \in \Tet(\Delta), \quad \sigma \in \edgeface.
    \end{equation}
    This is a consequence of the fact that all triple ratios in $x$ are $1$ (cf. Theorem~\ref{thm:X_tets_of_flags} part \eqref{item:tet_of_flag_is_X_3} $\Rightarrow$ \eqref{item:tet_of_flag_is_X_2}). It follows that, if $\star \not= 0$, all of the $j_{\sigma_+}^\tet$'s are determined by any one of them, via equations~\eqref{item:j_sigma_and_gluing_equation} and \eqref{eq:thurston_additional_relations}.
\end{remark}

\begin{theorem}\label{thm:phi_image} Let $M$ be an orientable $3$--manifold with an ideal triangulation $\Delta$, then
$$
\varphi_\star(\mathcal{D}_{\XX_{\star}}(M;\Delta)) = \mathcal{S}_\star(M;\Delta), \qquad \forall \star \in \{-1,1,0\}.
$$
Furthermore, if $\star \in \{-1,1\}$, then $\varphi_\star$ is injective. Otherwise, $\mathcal{D}_{\XX_{0}}(M;\Delta)$ is a trivial line bundle over $\mathcal{S}_0(M;\Delta)$, where two points $z_\ast^\ast,w_\ast^\ast \in \mathcal{D}_{\XX_{0}}(M;\Delta)$ belong to the same fiber if and only if there is a constant $k \in \RR$ such that
$$
\Re(z_\sigma^\tet) = \Re(w_\sigma^\tet), \quad \text{and} \quad \Im(z_\sigma^\tet) = e^k \cdot \Im(w_\sigma^\tet), \qquad \forall \tet \in \Tet(\Delta), \ \sigma \in \edgeface.
$$
\end{theorem}

\begin{proof}
    We begin by showing that $\varphi_\star(\mathcal{D}_{\XX_{\star}}(M;\Delta)) \subset \mathcal{S}_\star(M;\Delta)$. Suppose $z_\ast^\ast \in \mathcal{D}_{\XX_{\star}}(M;\Delta)$. It follows from the definition of $\varphi_\star$ and the internal relations~\eqref{eq:Thurston_param_relations} that $\varphi_\star(z_\ast^\ast)$ satisfies the internal consistency equations~\eqref{eq:triple_rat_eq}--\eqref{eq:fe2} (cf. Lemma~\eqref{lem:internal_eqs}), and the face pairing equations~\eqref{eq:face_equation} (cf. Lemma~\ref{lem:face_pair_eq}). To show that the gluing consistency equations \eqref{eq:gluing_eq_1} and \eqref{eq:gluing_eq_2} (cf. Lemma~\ref{lem:gluing_eqs}) are also satisfied, it is enough to notice that every Thurston's parameter $z_\sigma$ satisfies the additional relation
    $$
    \Im(z_{\sigma_+}) = \abs{z_{\sigma_+}}^2 \Im(z_{\sigma}).
    $$
    Next we show that $\varphi_\star(z_\ast^\ast)$ satisfies the edge gluing equations \eqref{eq:edge_gluing_equations}. For every edge $s\in \Edge(\Delta)$, fix an orientation on $s$, and let $\cyclic{\tet_1^s,\ldots, \tet_{k_s}^s}$ be the cyclically ordered ${k_s}$--tuple  of tetrahedra that abut $s$, cyclically ordered to follow the ``right hand rule'' by placing the thumb in the direction of $s$. Let $\tau_i^s$ and $\sigma_i^s$ be the respective incoming and outgoing edge-faces of $\tet_i^s$, around $s$. Henceforth we are going to drop the superscript $s$. Using that the coordinates of $\varphi_\star(z_\ast^\ast)(\tet)$ satisfy
    $$
    z^{\tet_{i}}_{\sigma_{i}} = z^{\tet_{i}}_{\tau_{i}}, \qquad \text{and} \qquad
    g_{\tau_i}^{\tet_i} =  \frac{\Im\left(z^{\tet_{i}}_{\sigma_{i}}\right)}{\Im(z^{\tet_{i-1}}_{\tau_{i-1}})} =  \frac{\Im\left(z^{\tet_{i}}_{\sigma_{i}}\right)}{\Im(z^{\tet_{i-1}}_{\sigma_{i-1}})},
    $$
    we deduce that each gluing map \eqref{eq:glue_mat} can be factored as $\Glue_{\tau_i}(\varphi_\star(z_\ast^\ast)(\tet))=G_{i-1}'G_{i}$, where 
    $$
    G_{i}'=\begin{bmatrix}
    1 & 0 & 0 & 0\\
    0 & 1 & 0 & 0\\
    0 & 0 & 1 & 0\\
    0 & 0 & 0 & 1/\Im(z^{\tet_{i}}_{\sigma_{i}})
    \end{bmatrix},
    \qquad \text{and} \qquad
    G_{i}=\begin{bmatrix}
    0 & 1 & 0 & 0\\
    1 & 0 & 0 & 0 \\
    0 & 0 & 1 & 0\\
    0 & 0 & 0 & -\Im(z^{\tet_i}_{\sigma_i})
    \end{bmatrix}.
    $$
    Thus we can rewrite the product in~\eqref{eq:edge_gluing_equations} as
    \begin{align}\label{eq:X_edge_cycle}
    G^s = G'_{k_s}G_1\Flip_{\sigma_1}(\varphi_\star(z_\ast^\ast)(\tet))G_1'G_2\ldots G'_{k_s-1}G_{k_s}\Flip_{\sigma_{k_s}}(\varphi_\star(z_\ast^\ast)(\tet)).
    \end{align}
    On the other hand, from the definition~\eqref{eq:flip_mat},
    $$
    \Flip_{\sigma_i}(\varphi_\star(z_\ast^\ast)(\tet))=\begin{bmatrix}
    0 & \abs{z^{\tet_i}_{\sigma_i}}^2 & 0 & 0\\[3pt]
    1 & 0 & 0 & 0\\[3pt]
    0 & 0 & \Re(z^{\tet_i}_{\sigma_i}) & \iota^2\Im(z^{\tet_i}_{\sigma_i})^2\\[3pt]
    0 & 0 & -1 & -\Re(z^{\tet_i}_{\sigma_i})
    \end{bmatrix},
    $$
    and so 
    $$
    G_i\Flip_{\sigma_i}(\varphi_\star(z_\ast^\ast)(\tet))G_i'=
    \begin{bmatrix}
    1 & 0 & 0 & 0\\[3pt]
    0 & \abs{z^{\tet_i}_{\sigma_i}}^2 & 0 & 0\\[3pt]
    0 & 0 & \Re(z^{\tet_i}_{\sigma_i})  & \iota^2\Im(z^{\tet_i}_{\sigma_i})\\[3pt]
    0 & 0 & \Im(z^{\tet_i}_{\sigma_i}) & \Re(z^{\tet_i}_{\sigma_i})
    \end{bmatrix}.
    $$
    It follows from~\eqref{eq:B_mat_mult} that $G^s$ is conjugate to 
    \begin{equation}\label{eq:thurston_gluing_matrix}
    \begin{bmatrix}
    1 & 0 & 0 & 0 \\[3pt]
    0 & \prod_{i=1}^{k_s}\abs{z^{\tet_i}_{\sigma_i}}^2 & 0 & 0\\[3pt]
    0 & 0 & \Re\left(\prod_{i=1}^{k_s} z^{\tet_i}_{\sigma_i}\right) & \iota^2\Im \left(\prod_{i=1}^{k_s} z^{\tet_i}_{\sigma_i}\right)\\[3pt]
    0 & 0 & \Im \left(\prod_{i=1}^{k_s} z^{\tet_i}_{\sigma_i}\right) & \Re \left(\prod_{i=1}^{k_s} z^{\tet_i}_{\sigma_i}\right)
    \end{bmatrix}.
    \end{equation}
    However, $z_\ast^\ast \in \mathcal{D}_{\XX_{\star}}(M;\Delta)$  satisfies Thurston's gluing equations~\eqref{eq:Thurston_gluing_equation}, thus $G^s$ is the identity matrix. This concludes the proof that $\varphi_\star(z_\ast^\ast) \in \mathcal{D}_{\Delta}$. The fact that $\varphi_\star(z_\ast^\ast) \in \mathcal{S}_\star(M;\Delta)$ is a direct consequence of Theorem~\ref{thm:X_tets_of_flags} and Theorem~\ref{thm:X_gluing_param}.\\
    
    To prove that $\varphi_\star$ is surjective, let $x \in \mathcal{S}_\star(M;\Delta)$. Recall that we refer to the coordinates of $x$ with the usual symbols $e_\sigma^\tet,t_\sigma^\tet$ and $g_\sigma^\tet$. Let $\delta_\ast^\ast \in \RR_{>0}^{\Tet(\Delta) \times \edgeface}$ be a function satisfying
    \begin{equation}\label{eq:delta_sigma_equations}
        \delta_{\sigma_+}^\tet = e_{\sigma_+} \delta_{\sigma}^\tet, \qquad  \delta_{\sigma}^\tet = g_\sigma^\tet \delta_{\tau}^{\tet'}, \quad \text{and} \quad -\iota^2 (\delta_{\sigma}^\tet)^2 = j_{\sigma}^\tet, \qquad \tet \in \Tet(\Delta), \ \sigma \in \edgeface,
    \end{equation}
    where $\tet'$ is the unique tetrahedron glued to $\tet$ along $(\sigma,\tau)$. If $\iota^2 \in \{-1,1\}$, then $\delta_\ast^\ast$ is uniquely determined by $\delta_\sigma^\tet = \sqrt{\abs{j_\sigma^\tet}}$. In this case, the first two equations of~\eqref{eq:delta_sigma_equations} are consequences of the assumption that $x \in \mathcal{S}_\star(M;\Delta)$ (cf. Remark~\ref{rem:sigma_j_relations}). On the other hand, if $\iota^2 =0$, then $j_{\sigma}^\tet = 0$ by definition of $\mathcal{S}_0(M;\Delta)$, and the last equation of~\eqref{eq:delta_sigma_equations} is trivial. In that case $\delta_\ast^\ast$ is only determined up to positive rescaling.
    
    Next we consider the function $z_\ast^\ast \in \mathcal{B}_\star^{\Tet(\Delta)\times \edgeface}$ defined by
    \begin{equation}\label{eq:inverse_map_of_varphi_star}
        z_\sigma^\tet := X_\sigma^\tet + \iota \delta_\sigma^\tet \in \mathcal{B}_\star ,\qquad  \forall \tet \in \Tet(\Delta), \quad \sigma \in \edgeface.
    \end{equation}
    
    It is easy to check that if $z_\ast^\ast \in \mathcal{D}_{\XX_{\star}}(M;\Delta)$, then $\varphi_\star(z_\ast^\ast) = x$. Indeed it is enough to notice that
    \begin{equation}\label{eq:norm_and_imaginary_parts_are_coorect}
    \abs{z_\sigma^\tet}^2 = (X_\sigma^\tet)^2 - \iota^2 (\delta_\sigma^\tet)^2 = (X_\sigma^\tet)^2 + j_\sigma^\tet = e_\sigma^\tet, \qquad \text{and} \qquad \frac{\Im(z_\sigma^\tet)}{\Im(z_\tau^{\tet'})} = \frac{\delta_\sigma^\tet}{\delta_\tau^{\tet'}} = g_\sigma^\tet.
    \end{equation}
    
    Therefore we are left to show that $z_\ast^\ast \in \mathcal{D}_{\XX_{\star}}(M;\Delta)$. It follows from the definitions that each number $z_\sigma^\tet$ is the Thurston's parameter of some positively oriented ideal $\XX_\star$--tetrahedron (cf. Lemma~\ref{lem:thurston_parameters}), and that they satisfy the internal relations~\eqref{eq:Thurston_param_relations}. Now let $s \in \Edge(\Delta)$ be an oriented edge. Using~\eqref{eq:norm_and_imaginary_parts_are_coorect} and the definition of $\mathcal{S}_\star(M;\Delta)$, one can rewrite every edge ratio, triple ratios and gluing parameter in terms of $z_\ast^\ast$ in the edge gluing equations, to show that $G^s$ is conjugate to the matrix in~\eqref{eq:thurston_gluing_matrix}. But since $x \in \mathcal{S}_\star(M;\Delta)$, $G^s$ is the identity matrix, and therefore $z_\ast^\ast$ satisfies the Thurston's gluing equation~\eqref{eq:Thurston_gluing_equation} of $s$. This concludes the proof of the first part of the lemma.\\
    
    For the second part, it is enough to notice that every $\delta_\ast^\ast \in \RR_{>0}^{\Tet(\Delta) \times \edgeface}$ defines a map
    $$
    \psi_\star : \mathcal{S}_0(M;\Delta) \rightarrow \mathcal{D}_{\XX_{\star}}(M;\Delta),
    $$
    via~\eqref{eq:inverse_map_of_varphi_star}, such that $\varphi_\star \circ \psi_\star = 1$. If $\star \in \{-1,1\}$, then $\delta_\ast^\ast$ is uniquely determined and $\psi_\star \circ \varphi_\star = 1$, implying that $\varphi_\star$ is injective. On the other hand, if $\star = 0$ then $\delta_\ast^\ast$ is only determined up to positive rescaling, and each $\psi_0$ is a section of $\varphi_0$. 
\end{proof}

We conclude this section by showing that the maps $\varphi_\star,\mathrm{f}_\star, \Ext$ and $\mathrm{Ext}_{\XX_\star}$ fit nicely in a commutative diagram. We recall that $\Ext$ is the extension map from Theorem~\ref{thm:geom_flag_struc_are_developable}, which associates each triangulation of flags to a branched real projective structure. Similarly, the map $\mathrm{Ext}_{\XX_\star}$ associates points in the $\XX_\star$--deformation space to branched $(\G_\star,\XX_\star)$--structures (cf. Lemma~\ref{lem:develop_X_structures}). Finally, $\mathrm{f}_\star$ is the ``forgetful'' map that associates a $(\G_\star,\XX_\star)$--structure to the underlying projective structure (cf. \S\ref{subsec:X_spaces}).

\begin{theorem}\label{thm:X_comm_diag}
    Let $\star \in \{-1,1,0\}$, and let $M$ be an orientable $3$--manifold with an ideal triangulation $\Delta$. Then the following diagram commutes
    \begin{equation}\label{cd:X_cd}
    \begin{tikzcd}
    \mathcal{D}_{\XX_\star}(M;\Delta)\ar[d,rightarrow,"\varphi_\star"'] \ar[r,"\mathrm{Ext}_{\XX_\star}"] & \XX_\star(M;\Delta)\ar[d,rightarrow,"\mathrm{f}_\star"]\\
    \hspace{2cm} \mathcal{S}_\star(M;\Delta) \subset \mathcal{D}_{\rp}(M;\Delta)\ar[r,"\Ext"] & \rp(M;\Delta)
    \end{tikzcd}
    \end{equation}
    
\end{theorem}
\begin{proof}
    Let $z_\ast^\ast \in \mathcal{D}_{\XX_\star}(M;\Delta)$, and consider the branched projective structures
    $$
    p_1 :=  \mathrm{f}_\star\left( \mathrm{Ext}_{\XX_\star}(z_\ast^\ast) \right), \qquad \text{and} \qquad p_2 := \Ext\left( \varphi_\star(z_\ast^\ast) \right).
    $$
    Let $\dev_1,\dev_2$ be developing maps for $p_1$ and $p_2$, respectively. It is clear from the definitions that $\dev_1$ maps all tetrahedra of $\wDelta$ to ideal $\XX_\J$--tetrahedra, for some osculating model $(\G_\J,\XX_\J,\J)$. The same is true for $\dev_2$, for some other model $(\G_{\J'},\XX_{\J'},\J')$, by Theorem~\ref{thm:X_tets_of_flags} and Theorem~\ref{thm:X_gluing_param}. Thus $p_1$ is a $(\G_\J,\XX_\J)$--structure, while $p_2$ is a $(\G_\J',\XX_\J')$--structure. It is easy to check that, in both cases, the Thurston's parameters of all these tetrahedra are exactly $z_\ast^\ast$. It follows that these structures are projectively equivalent, hence $p_1 = p_2$.
\end{proof}

\begin{proof}[Proof of Theorem \ref{thm:main_thm}]
Theorem \ref{thm:X_comm_diag} establishes that the maps in Theorem \ref{thm:main_thm} fit into the appropriate commutative diagram. The map $\mathrm{f} = \mathrm{f} _{-1}$ is injective by Lemma \ref{lem:forget_structures_almost_injective} and $\mathrm{Ext}_\HH = \mathrm{Ext}_{\XX_{-1}}$ is injective by Lemma \ref{lem:develop_X_structures}. However, commutativity of the diagram then implies that the map $\varphi = \varphi_{-1}$ must also be injective.  

\end{proof}
	 

\section{Properly convex projective structures}\label{sec:prop_convex_projective_st}

A specific type of projective structures of interest in this work are \emph{properly convex projective structures}. A subset $\Omega \subset \rp^n$ is \emph{properly convex} if it has non-empty interior, its closure is disjoint from some hyperplane $\eta \subset \rp^n$, and it is convex in the \emph{affine patch} determined by $\rp^n \setminus \eta$. In this context, a projective structure on an $n$--manifold $M$ is \emph{properly convex} if its developing map is a diffeomorphism onto some properly convex set $\Omega$. In that case the holonomy representation is an isomorphism onto a discrete group $\Gamma\subset \PGL(n+1)$. The quotient manifold $\Omega/\Gamma$ is a \emph{properly convex manifold} and the developing map descends to a diffeomorphism $M\to \Omega/\Gamma$, called \emph{marking}. We let $\Tt(M)$  denote the subset of $\rp(M)$ consisting of (marked) properly convex projective structures.\\

This section is concerned with those properly convex structures whose geometry near the non-compact part is particularly well behaved. This is the concept of a \emph{generalized cusp}, which will be introduced in \S\ref{subsec:gen_cusp}. It is conjectured that, under some mild condition on the fundamental group, every cusp of a properly convex manifold is a generalized cusp (cf. Conjecture~\ref{conj:gen_cusps}). The main goal of this section will be to understand how to decorate generalized cusps with flags (cf. Lemma~\ref{lem:framing_end}), and to prove Theorem~\ref{thm:main_thm2} (cf. \S\ref{subsec:proo_main_thm_2}).  


\subsection{Generalized cusps}\label{subsec:gen_cusp}

Generalized cusps are properly convex generalizations of cusps of finite volume hyperbolic manifolds. A \emph{generalized cusp} is a properly convex manifold with virtually abelian fundamental group that admits a codimension $1$ foliation by compact, strictly convex hypersurfaces. In this context, an embedded codimension $1$ submanifold $\Sigma\subset M$ is \emph{strictly convex} if for each $V\in \Sigma$, every tangent plane to $\Sigma$ at $V$ locally intersects $\Sigma$ in a single point (namely $V$). In other words, $\Sigma$ is ``locally on one side'' of each tangent plane. A good example to keep in mind is a cusp of a finite volume hyperbolic $3$--manifold. Such a cusp is foliated by tori, each of which is covered by a (strictly convex) horosphere. 

Generalized cusps are classified in \cite{BCL}, and we now describe the classification in the setting of $3$--manifolds. In dimension three, every generalized cusp is of the form $C=\Omega/\Gamma$, where $C\cong T^2\times (0,\infty)$ and $T^2$ is a $2$--torus. For each generalized cusp we can define a projective invariant called the \emph{type}, which is an integer $0\leq t\leq 3$ (see \cite[\S 1]{BCL}. Loosely speaking, the type gives the dimension of the largest subgroup of $\Gamma$ that is diagonalizable over $\RR$ (a slightly different definition applies when $t=3)$. Furthermore, for each $C$ there is a non-negative vector $u=(\lambda_1,\lambda_2,\lambda_3)\in \RR_{\geq0}^ 3$, where $\lambda_i=0$ if and only if $i> t$. When $t=3$, define $\phi_u:\RR_{\geq0}^3\to \RR$ by $\phi(x_1,x_2,x_3)=x^{\lambda_1}_1x_2^{\lambda_2}x_3^{\lambda_3}$ and $\Omega_u=\phi_u^{-1}((0,\infty))$. Next let $H_u$ be the $2$--dimensional abelian Lie subgroup of $\PGL(4)$ consisting of elements of the form
\begin{align}\label{eq:type3}
\begin{bmatrix}
e^{y_1} & 0 & 0 & 0\\
0 & e^{y_2} & 0 & 0\\
0 & 0 & e^{y_3} & 0\\
0 & 0 & 0 & 1
\end{bmatrix},
\end{align}
where $\sum_{i=1}^3\lambda_i y_i=0$. The set $\Omega_u$ is properly convex, and is foliated by the (strictly convex) level sets of $\phi_u$, called \emph{horospheres}. Both $\Omega_u$ and each leaf of this foliation is $H_u$--invariant.  
{Furthermore, every generalized cusp of type $3$ is projectively equivalent to $\Omega_u/\Gamma_u$, for some $u$ and some lattice $\Gamma_u$ in $H_u$} (see \cite[Thm 0.1]{BCL}).

When $t<3$, define $\phi_u: \RR_{\geq 0}^t\times \RR^{3-t}\to \RR$ by $\phi_u(x_1,x_2,x_3)=x_{t+1}+\sum_{i=1}^t\lambda_i x_i-\frac{1}{2}\sum_{t+2}^3x_i^2$. Note that with this definition, it is possible for some of the sums in the definition of $\phi_u$ to be empty depending on the value of $t$. For each $t$, let $\Omega_u=\phi_u^{-1}((0,\infty))$ and define $H_u$ to be the $2$--dimensional abelian Lie subgroup of $\PGL(4)$ consisting of elements of the form
\begin{align}\label{eq:typeleq2}
&\begin{bmatrix}
1 & x_1 & x_2 & \frac{x_1^2 + x_2^2}{2} \\
0 & 1 & 0 & x_1\\
0 & 0 & 1 & x_2\\
0 & 0 & 0 & 1
\end{bmatrix}, \quad
&&\begin{bmatrix}
e^{y_1} & 0 & 0 & 0 \\
0 & 1 & x_1 & \frac{x_1^2}{2} + \lambda_1 y_1\\
0 & 0 & 1 & x_1\\
0 & 0 & 0 & 1
\end{bmatrix}, \quad
&&\begin{bmatrix}
e^{y_1} & 0 & 0 & 0 \\
0 & e^{y_2} & 0 & 0\\
0 & 0 & 1 & -\lambda_1y_1-\lambda_2y_2\\
0 & 0 & 0 & 1
\end{bmatrix},\\ \nonumber
& \qquad \qquad t = 0, && \qquad \qquad t = 1, && \qquad \qquad t = 2.
\end{align}
Again, the set $\Omega_u$ is properly convex and foliated by the strictly convex level sets of $\phi_u$. Both the domain and the foliation are $H_u$--invariant and each type $t<3$ generalized cusp is projectively equivalent to $\Omega_u/\Gamma_u$ for some $u$ and some lattice $\Gamma_u$ in $H_u$. We depicted one example of $\Omega_u$ for each type $0\leq t \leq 3$ in Figure~\ref{fig:cusp_domains}.

\begin{figure}
    \centering
    \includegraphics[scale=.5]{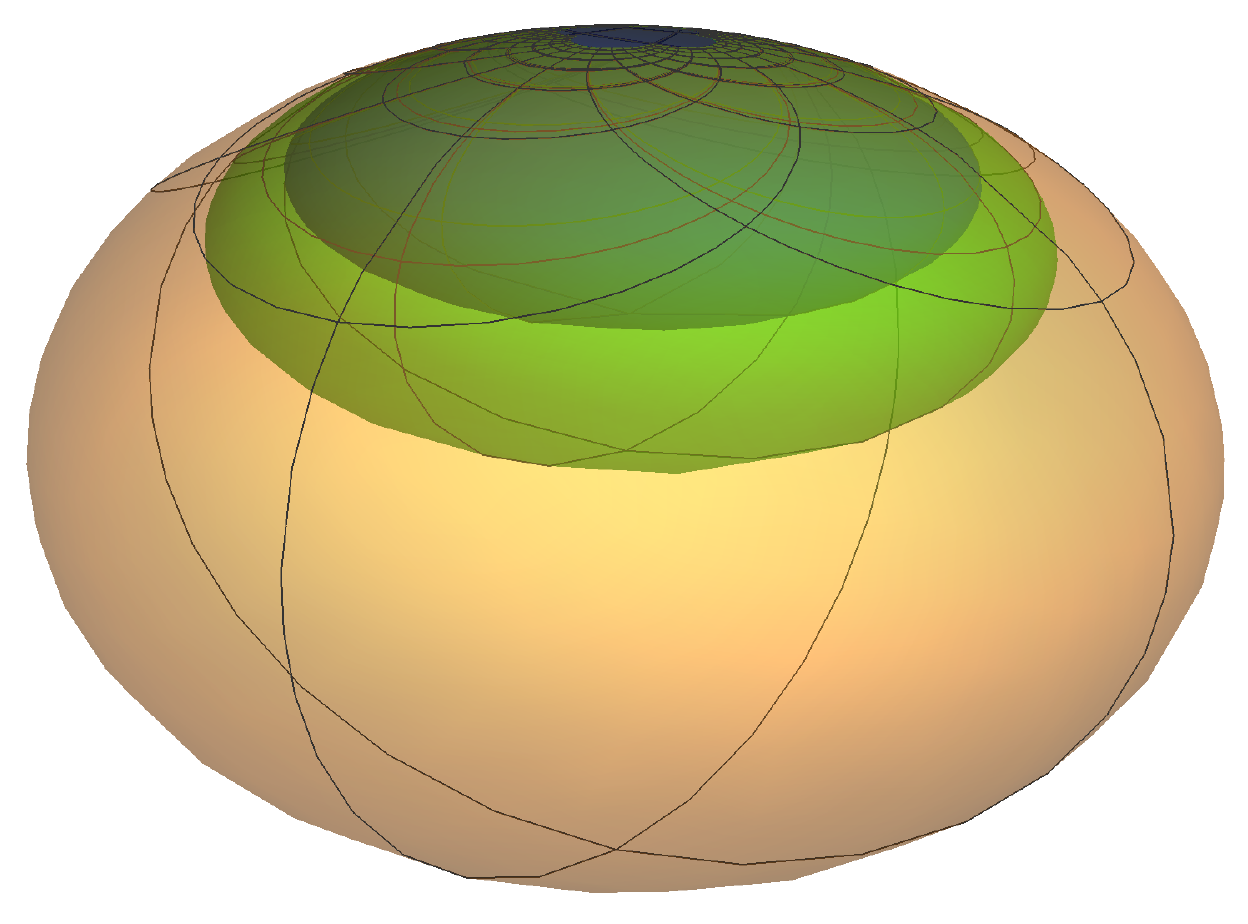}
    \includegraphics[scale=.5]{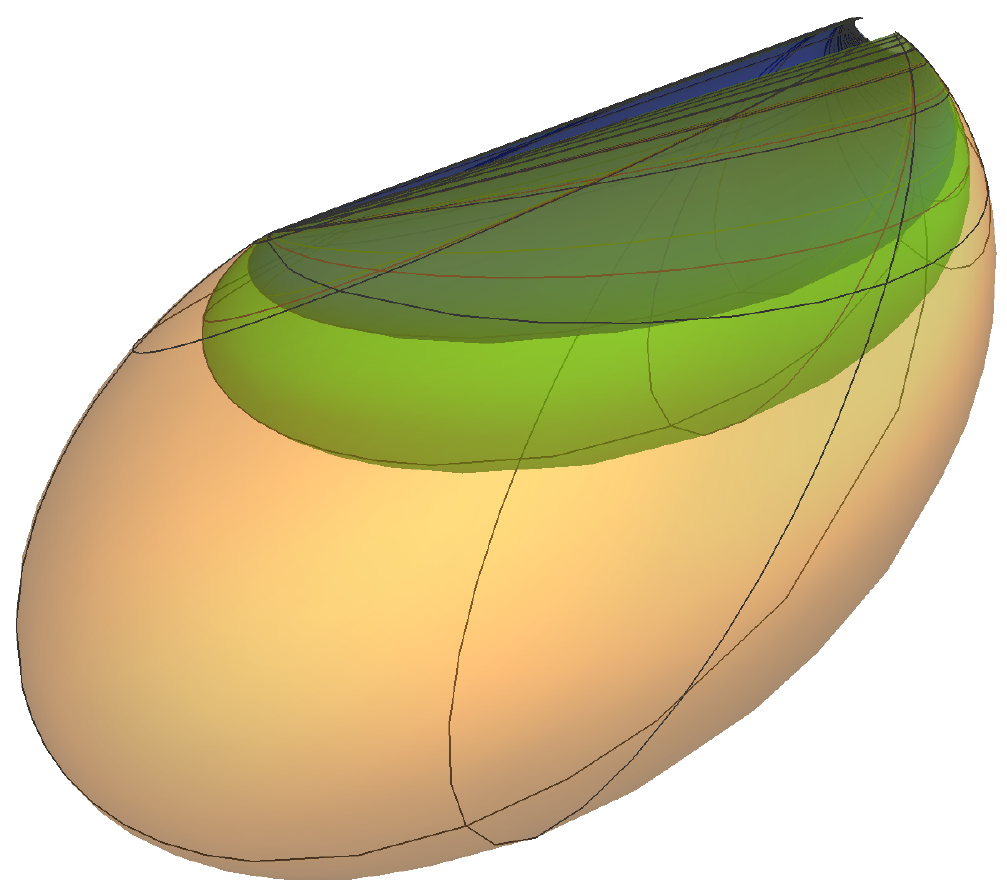}
     \includegraphics[scale=.5]{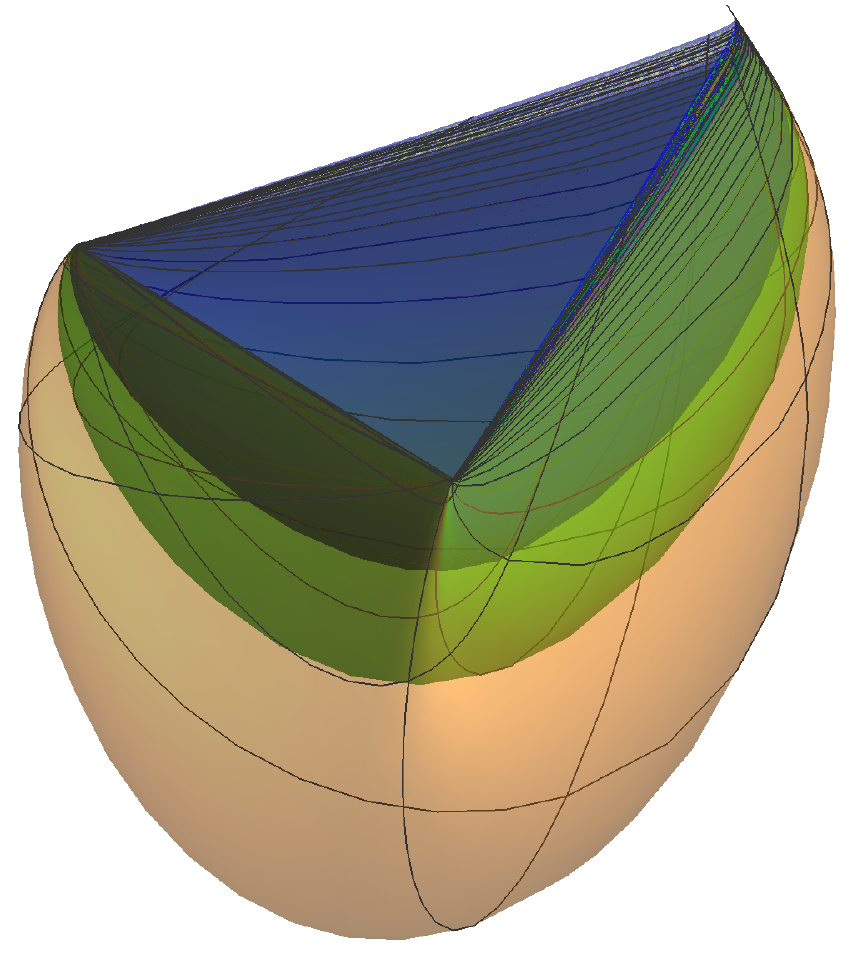}
      \includegraphics[scale=.45]{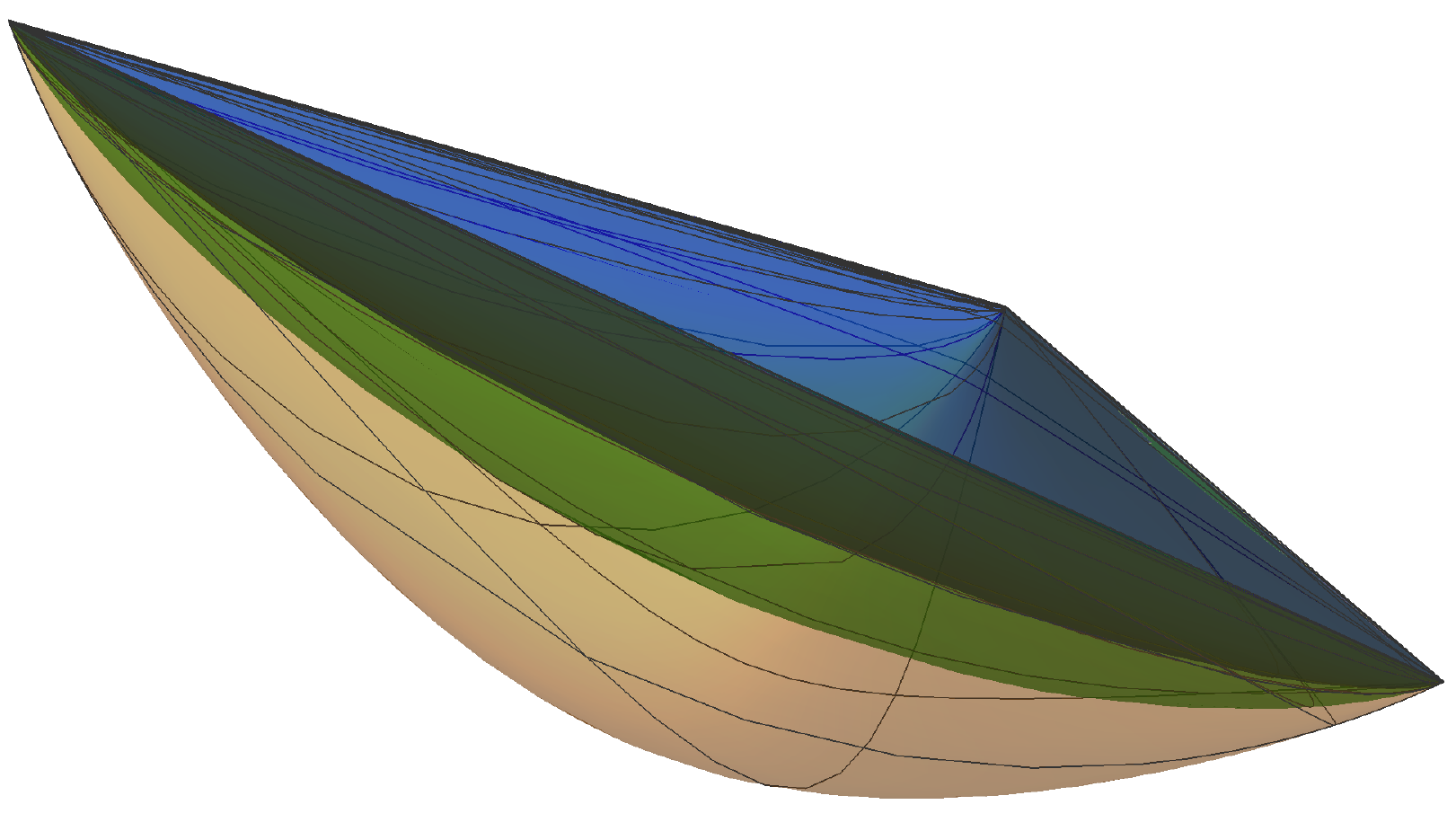}
    \caption{The domains $\Omega_u$. The type increases as one goes left to right and top to bottom.}
    \label{fig:cusp_domains}
\end{figure}

Given a properly convex domain $\Omega$, a flag $(V,\eta)$ is a \emph{supporting flag} if $V\in \partial \overline{\Omega}$, and $\eta$ is a supporting plane for $\Omega$. If $C=\Omega/\Gamma$ is a generalized cusp then $(V,\eta)$ is a \emph{supporting flag for $C$} if $(V,\eta)$ is a supporting flag for $\Omega$ and it is $\Gamma$--invariant. Generalized cusps always admit supporting flags, and most of the time (when $ t \neq 3$) there is a canonical one. They can be described using the concept of radial flow which we briefly recall here (for more details refer to \cite[\S 1.1]{BCL}.

A \emph{radial flow} $\varphi_s$ on $\rp^3$ is a $1$--parameter flow by projective transformations with the property that there is a unique plane $\eta \subset \rp^3$ and a unique point $V\in \rp^3$ so that $\varphi_s$ acts trivially on $\eta$ for each $s$, and $\displaystyle \lim_{s\to -\infty}\varphi_s(x)=V$ for any $x\notin \eta$. Up to reparametrization and applying projective transformations there are exactly two radial flows distinguished by whether or not $V\in \eta$. In the former case we call the radial flow \emph{parabolic}, and in the latter case it is \emph{hyperbolic}. The point $V$ (resp. \ plane $\eta$) is called the \emph{center} (resp.\ \emph{dual center}) of the radial flow.

It turns out that for each generalized cusp $C = \Omega_u/\Gamma_u$ there is a distinguished radial flow that centralizes the group $\Gamma_u$ and for which $\varphi_s(\Omega_u)\subset \Omega_u$ for $s\leq 0$. These radial flows can be described in a ``coordinate free'' way and hence they are (up to reparametrization) a projective invariant of the generalized cusp (see \cite[Prop.\ 1.29]{BCL}. If $0\leq t<3$, then the radial flow is parabolic and has center $V=[\basis{t+1}]$ and dual center $\eta=[\dbasis{4}]$. In this case we see that $(V,\eta)$ is an incomplete flag that is $\Gamma_u$ invariant. Furthermore, the point $V\in \partial \overline{\Omega}_u$ and the plane $\eta$ is a supporting plane for $\Omega_u$. If $t=3$, then the radial flow is hyperbolic and has center $V=[\basis{4}]$ and dual center $\eta=[\dbasis{4}]$. In this case the center and dual center do not form a flag, however $\eta$ is a supporting plane for $\Omega_u$. The intersection of $\eta$ with $\partial \overline{\Omega}_u$ is a triangle and if we let $V'$ be one of the vertices of this triangle, then $(V',\eta)$ is a supporting flag. 

For future reference we record the above discussion in the following lemma. In addition, we show that these flags are some of only finitely many possible supporting flags of $C$.

\begin{lemma}\label{lem:framing_end}
If $C=\Omega/\Gamma$ is a generalized cusp, then $\Gamma$ preserves at least one and at most finitely many (incomplete) flags in $\rp^3$. Furthermore, at least one of these flags can be chosen to be a supporting flag for $C$.
\end{lemma}

\begin{proof}
    As mentioned above, existence follows from the above discussion concerning radial flows.

    Recall that, depending on the type $0\leq t\leq 3$ of $C$, the group $\Gamma$ is a lattice in one of the Lie groups $H_u$ defined in \eqref{eq:type3} and \eqref{eq:typeleq2}. When $t>0$, it is easy to check using \eqref{eq:type3} and \eqref{eq:typeleq2} that a generic element $G\in \Gamma$ will preserve only finitely many incomplete flags.
    
    On the other hand, if $t=0$, then each $G$ preserves infinitely many incomplete flags. Specifically, we have that
    $$
    G=\begin{bmatrix}
    1 & y_1 & y_2 & \frac{1}{2}(y_1^2+y_2^2)\\
    0 & 1 & 0 & y_1\\
    0 & 0 & 1 & y_2\\
    0 & 0 & 0 & 1
    \end{bmatrix}, \qquad \text{for some} \quad y_1,y_2 \in \RR,
    $$
    thus $G$ pointwise fixes the line $\ell=[\basis{1}][-y_2\basis{2}+y_1\basis{3}]$, and preserves each plane in the pencil of planes through $\ell$. It follows that $G$ preserves any incomplete flag where the point lies on $\ell$ and the plane comes from the pencil through $\ell$. Furthermore, every $G$--invariant incomplete flag arises in this way.  However, if $G'\in \Gamma$ is such that $\langle G,G'\rangle\cong \Gamma$, then the fixed line $\ell'$ for $G'$ is such that $\ell\cap \ell'=[\basis{1}]$ and $\ell' \subset [\dbasis{4}]$. It follows that $([\basis{1}],[\dbasis{4}])$ is the unique incomplete flag preserved by $\Gamma$.
\end{proof}

\begin{remark}\label{rem:cplt_flags_bad}
Notice that if one instead uses complete flags, then Lemma~\ref{lem:framing_end} is false. For example, a type $t=0$ generalized cusp preserves the complete flags $([\basis{1}],[\basis{1}][v],[\basis{1}][\basis{2}][\basis{3}])$, for any point $[v]$ contained in the line $[\basis{2}][\basis{3}]$.
\end{remark}


\subsection{Proof of Theorem~\ref{thm:main_thm2}}\label{subsec:proo_main_thm_2}

Let $M$ be an orientable non-compact $3$--manifold, which is the interior of a compact manifold $\hat{M}$ whose boundary is a finite union of disjoint $2$--tori $T_1, \ldots, T_n$. Let $\mathcal{N}_1, \ldots, \mathcal{N}_n$ be open neighbourhoods of $T_1, \ldots, T_n$ in $\hat{M}$, with pairwise disjoint closures. We call $\End_k = \overline{\mathcal{N}_k} \setminus T_k$ an \emph{end} of $M$.

For each end $\End_k$ of $M$, there is a conjugacy class of \emph{peripheral subgroups} corresponding to $\im (\pi_1(\End_k) \to \pi_1(M))$, where different choices of basepoint give different conjugacy classes.  Let $\Gamma_k \le \pi_1(M)$ be a representative for this class, and let $(\dev,\hol)$ be a developing pair of a properly convex projective structure on $M$. Let $\wt \End_k$ be the lift of $\End_k$ to the universal cover of $M$ corresponding to $\Gamma_k$. In this setting, the end $\End_k$ is a \emph{generalized cusp end} if (up to isotopy) $\dev(\wt \End_k)/\hol(\Gamma_k)$ is projectively equivalent to a generalized cusp (cf. \S\ref{subsec:gen_cusp}). By $\hol$--equivariance of $\dev$, this notion does not depend on the choice of representative $\Gamma_k$ and therefore it is well defined. Next, we define $\Tt^c(M)\subset \Tt(M)$ to be the set of properly convex projective structures on $M$ such that each end is a generalized cusp end. In  this context, there is a recent conjecture of Cooper and Tillmann:

\begin{conj}[Cooper-Tillmann]\label{conj:gen_cusps}
Suppose $\pi_1(M)$ is hyperbolic relative peripheral subgroups, then every end of a properly convex structure on $M$ is a generalized cusp. In other words $\Tt^c(M)=\Tt(M)$.
\end{conj}
\begin{remark}\label{rem:conjecture_is_proven}
    We were recently made aware via personal communication that Cooper and Tillmann have a proof of this conjecture. 
\end{remark}
\begin{remark}\label{rem:hyp_rel_subgroup}
    To avoid introducing more unnecessary details, we will not define what it means for $\pi_1(M)$ to be hyperbolic relative peripheral subgroups, but refer the reader to \cite{FARB}. We simply remark that all finite volume hyperbolic manifolds have this property (see \cite[Thm. 5.1]{FARB}).
\end{remark}

We are now ready to prove Theorem~\ref{thm:main_thm2}. In the proof, we are implicitly going to make use of the identification $\mathcal{D}_\Delta \cong \cTriFL$ (cf. Theorem~\ref{thm:parametrization_def_space}).

\begin{proof}[Proof of Theorem~\ref{thm:main_thm2}]
Let $p_0 = [\Phi_0,\rho_0] \in \cTriFL$ such that $\sigma_0 = [\hat \Phi_0,\rho_0] := \Ext(p_0) \in \Tt^c$, and let $\sigma_1 = [\dev_1,\rho_1] \in \Tt^c$ be a properly convex structure near $\sigma_0$. We want to show that, if $\sigma_1$ is close enough to $\sigma_0$, then there is $p_1 \in \cTriFL$ close to $p_0$ such that $\sigma_1 =\Ext(p_1)$.

Let $v \in \Vertex(\Delta)$ and let $\End_v$ be the corresponding end in $M$. For every lift $\wt v \in \Vertex(\wt \Delta)$ of $v$ there is a corresponding lift $\wt \End_v$ of $\End_v$ and a peripheral subgroup $\Gamma_v \le \pi_1(M)$. Since $\sigma_0$ is a properly convex structure coming from a triangulation of flags, the flag $ \Phi_0(\wt v) = (V_0,\eta_0)$ is a supporting flag for $\hat \Phi_0(\wt \End_v)/\rho_0(\Gamma_v)$. If $\sigma_1$ is close enough to $\sigma_0$, then there is a representative pair $(\dev_1,\rho_1)$ of $\sigma_1$ whose holonomy $\rho_1$ is close to $\rho_0$. In particular there is a $\rho_1(\Gamma_v)$--invariant flag $(V_1,\eta_1)$ that is close to $(V_0,\eta_0)$. It is possible that $(V_1,\eta_1)$ is not unique if the type of the generalized cusp at $v$ corresponding to $\sigma_1$ is larger than the generalized cusp corresponding to $\sigma_0$, 
{however in this case one can be chosen arbitrarily as long the choice is made equivariantly for all lifts of $v$.}

Thus we have just defined a $\rho_1$--equivariant map $\Phi_1 : \Vertex(\wt \Delta) \to \FL$, such that $\Phi_1(\wt v) = (V_1,\eta_1)$. By construction, the map $\Phi_1$ maps each $\wt v\in \Vertex(\wt \Delta)$ close to its $\Phi_0$--image. Since the property of being a tetrahedron of flags is an open condition in the space of quadruples of flags and there are only finitely many tetrahedra in $\Delta$, it follows that $\Phi_1$ maps the vertices of each $\wt \tet\in \Tet(\wt\Delta)$ to a tetrahedron of flags. The property of a pair of tetrahedra of flags being geometric is also an open condition and  so it follows that $p_1 := [\Phi_1,\rho_1] \in \cTriFL$.

Finally, we claim that $\Ext(p_1) = \sigma_1$. Clearly there are representative developing/holonomy pairs for $\Ext(p_1)$ and $\sigma_1$ where the holonomy is the same. Furthermore, since $\Phi_1$ and $\Phi_0$ are close, it follows from Theorem \ref{thm:geom_flag_struc_are_developable} that $\hat \Phi_1$ and $\hat \Phi_0$ are also close. Since $\sigma_0$ and $\sigma_1$ are close it follows that $\hat \Phi_0$ and $\dev_1$ are close. These two facts imply that $\hat\Phi_1$ and $\dev_1$ are nearby developing maps with the same holonomy. Therefore $\Ext(p_1) = \sigma_1$ by the Ehresmann-Thurston principle (cf. Theorem~\ref{thm:ersh_thu_principle}). Finally, the finiteness--to--one property of $\Ext$ is a consequence of the fact that each vertex $\wt v\in \Vertex(\wt\Delta)$ corresponds to a generalized cusp end, and by Lemma \ref{lem:framing_end}, such ends can be decorated with incomplete flags in only finitely many ways. 
\end{proof}


\section{Examples}\label{sec:examples}

In this section we describe some instances where we are able to (partially) solve our gluing equations, and describe the geometry of the corresponding projective structures. 


\subsection{Figure-eight knot complement}\label{sec:fig8}

Let $M_8$ be the figure-eight knot complement also known as the manifold {\tt m004} in the SnapPy cusped census \cite{SnapPy}. $M_8$ has an ideal triangulation $\Delta$ consisting of two tetrahedra with the gluing pattern shown in Figure \ref{fig8_gluing}. Here we describe a curve in the deformation space $\mathcal{D}_8 := \mathcal{D}_\rp(M_8;\Delta)$. We recall that $\mathcal{D}_8$ can be regarded as semi-algebraic subset of $\RR_{>0}^{48}$. Specifically we have an edge and a gluing parameter for each pair $(\tet,\sigma)\in \Tet(\Delta)\times \edgeface$ (we do not need the triple ratios since they can be recovered from the edge ratios using \eqref{eq:fe1}).  For each of the two tetrahedra $\tet_1$ and $\tet_2$ there is an edge ratio and gluing parameter for each edge-face. Let $e_\sigma$ (resp.\ $f_\sigma$) denote the edge ratio of the edge-face $\sigma$ for the tetrahedron $\tet_1$ (resp.\ $\tet_2$) and let $g_\sigma$ (resp.\ $h_\sigma$) denote the gluing parameter of the edge-face $\sigma$ for the tetrahedron $\tet_1$ (resp.\ $\tet_2$).
We now define a $1$--parameter family of points in $\mathcal{D}_8$.  First, let each $e_\sigma$ and $f_\sigma$ be equal to either $t^{\pm 1}$ according to Table \ref{tab:fig8}. Next, if $\sigma=(12)0$ or $(30)2$ then let $h_\sigma=1/g_\sigma=t^2$. For all other $\sigma$, let $h_\sigma=g_\sigma=1$. 

Letting $t$ range over $(0,\infty)$ gives a path in $\RR_{>0}^{48}$. We now show that this path lies in $\mathcal{D}_8$. In both rows of Table $\ref{tab:fig8}$, conjugate edge-faces have the same sign and there are the same number of $+$'s and $-$'s and so the internal consistency relations \eqref{eq:fe3} are satisfied for both tetrahedra. It is straightforward to check that the gluing consistency equations \eqref{eq:gluing_eq_1} and \eqref{eq:gluing_eq_2} are satisfied for each relevant pair of edge-faces, for each point on this curve. There are two sets of edge equations for this ideal triangulation (one for each edge class). More specifically, let $\sigma_i$ be the edge-face determined by the values in Table \ref{tab:eftab1}. These equations are equivalent to the following matrices being equal to the identity:
\begin{align}\label{eq:fig8_edge_mats}
M_1=\prod_{i=1}^6F^{i+1}_{\sigma_i}G^{i+1}_{\sigma_i}, \qquad \text{and} \qquad
M_2=\prod_{i=7}^{12}F^{i+1}_{\sigma_i}G^{i+1}_{\sigma_i}.
\end{align}
Here $G^1_{\sigma_i}:=\Glue(h_{\sigma_i})$, $G^2_{\sigma_i}:=\Glue(g_{\sigma_i})$, $F^1_{\sigma_i}:=\Flip_{\sigma_i}([\F_1])$, and $F^2_{\sigma_i}:=\Flip_{\sigma_i}([\F_2])$ where $[\F_i]$ is the tetrahedron of flags corresponding to the parameters for $\tet_i$, and superscripts in \eqref{eq:fig8_edge_mats} are taken $\mod{2}$. Using one's favorite computer algebra system it is easily verified that both of the elements in \eqref{eq:fig8_edge_mats} are trivial, and so this curve lies in $\mathcal{D}_8$.

\begin{figure}[t]
	\centering
	\def\svgscale{.5}
	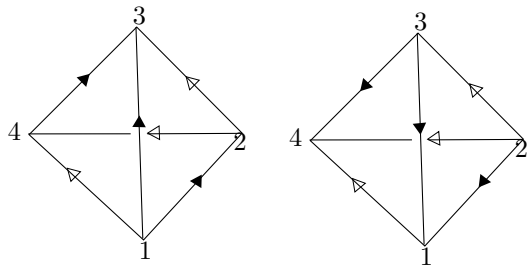
	\caption{An ideal triangulation for the figure-eight knot complement.}
	\label{fig8_gluing}
\end{figure}

\begin{table}[t]
	\begin{center}
		\begin{tabular}{|c|c|c|c|c|c|c|c|c|c|c|c|c|}
			
			\hline
			$\sigma$ & (12)3 &(21)4 &(34)1 & (43)2 & (13)4 & (31)2 & (24)3 & (42)1 & (14)2 & (41)3 & (23)1 & (32)4\\
			\hline
			$e_\sigma$ & + & + & + & + & - & - & + & + & - & - & - & - \\
			\hline
			$f_\sigma$ & - & - & - & - & + & + & - & - & + & + & + & +\\
			\hline
		\end{tabular}
	\end{center}
	\caption{Table of edge ratios for the figure-eight knot complement.}
	\label{tab:fig8}
\end{table}

\begin{table}[t]
	\centering
	\begin{tabular}{|c|c|c|c|c|c|c|c|c|c|c|c|c|}
		\hline
		i & 1 & 2 & 3 & 4 & 5 & 6 & 7 & 8 & 9 & 10 & 11 & 12 \\
		\hline
		$\sigma_i$ & (31)2 & (43)2 & (34)1 & (13)4 & (21)4 & (12)3 & (24)3 & (23)1 & (14)2 & (24)3 & (23)1 & (14)2\\
		\hline
	\end{tabular}
	\caption{Edge-face table for \eqref{eq:fig8_edge_mats}.}
	\label{tab:eftab1}
\end{table}

This curve of solutions turns out to contain many points corresponding to properly convex projective structures on the figure-eight knot complement. To see this we observe that every parameter is equal to $1$ for the solution with $t=1$. It is easy to check using Theorem \ref{thm:phi_image} that this solution comes from a hyperbolic structure on $M_8$. In fact, it can be checked that both tetrahedra are regular and ideal, and so this is the complete hyperbolic structure on $M_8$.

For values of $t$ near $1$, the peripheral holonomy preserves a complete real flag, and so by \cite{CLTII} it follows that these other solutions also correspond to properly convex structures on $M_8$. Furthermore, a more detailed analysis of the peripheral subgroup shows that for every $t\neq 1$ the holonomy of the peripheral subgroup is that of a type 1 generalized cusp (see \cite{BCL} for details). Furthermore, the holonomy of the meridian and longitude for these structures may be calculated by multiplying the transformations on the appropriate edges in the monodromy complex (see Figure \ref{fig8_cusp}). In this way, the holonomy of the meridian is always unipotent, and so it follows that this is the same path of deformations found in \cite{BAL14}. In particular, for values $t\neq 1$, the holonomy is that of a type 1 generalized cusp (see \cite{BCL}).

\begin{figure}
  \centering
  \def\svgscale{.4}
  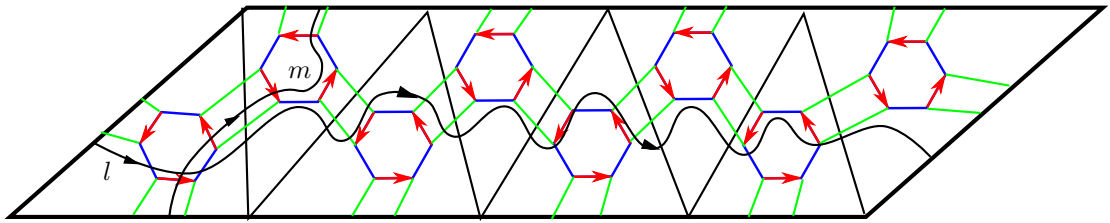
  \caption{The cusp cross section for the figure-eight, its intersection with the monodromy graph, and a meridian/longitude pair.}
  \label{fig8_cusp}
\end{figure}


\subsection{Figure-eight sister manifold}\label{sec:sister}

\begin{figure}
  \centering
  \def\svgscale{.5}
  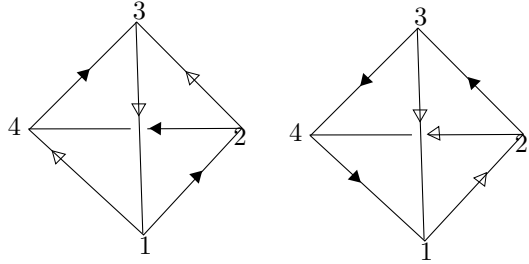
  \caption{A gluing for the figure-eight sister}\label{sister_gluing}
\end{figure}

Let $N_8$ be the manifold $\tt{m003}$ in the SnapPy \cite{SnapPy} cusped census (this manifold is also known as the figure-eight sister). $N_8$ also admits an ideal triangulation $\Delta$ with two tetrahedra, displayed in Figure~\ref{sister_gluing}. We will again adopt the convention that $e_\sigma$ and $g_\sigma$ are the edge and gluing parameters for the first tetrahedron, while $f_\sigma$ and $h_\sigma$ are the parameters for the second tetrahedron. We again define a curve by setting $t^2=h_{(14)2}=h_{(32)4}=1/g_{(24)3}=1/g_{(13)4}$, and setting each of the remaining gluing parameters equal to 1. Next, as before, we let $e_\sigma$ and $f_\sigma$ take values $t^{\pm 1}$ according to Table \ref{tab:fig8_sister}.

\begin{table}[h]
\begin{center}
\begin{tabular}{|c|c|c|c|c|c|c|c|c|c|c|c|c|}
\hline
$\sigma$ & (12)3 &(21)4 &(34)1 & (43)2 & (13)4 & (31)2 & (24)3 & (42)1 & (14)2 & (41)3 & (23)1 & (32)4\\
\hline
$e_\sigma$ & + & + & + & + & - & - & - & - & + & + & - & - \\
\hline
$f_\sigma$ & - & + & - & + & - & - & - & - & + & + & + & +\\
\hline
\end{tabular}
\end{center}
\caption{Table of edge ratios for figure-eight sister.}\label{tab:fig8_sister}
\end{table}

As with the previous example it is a simple, but tedious matter to check that the internal consistency and gluing consistency equations are satisfied for the points on this curve. Once again, there are two sets of edge equations coming from the following two matrices being trivial:
\begin{align}\label{eq:sis_edge_mats}
N_1=\prod_{i=1}^6F^{i+1}_{\sigma_i}G^{i+1}_{\sigma_i}, \qquad \text{and} \qquad
N_2=\prod_{i=7}^{12}F^{i+1}_{\sigma_i}G^{i+1}_{\sigma_i},
\end{align}
where again $G^1_{\sigma_i}=\Glue(h_{\sigma_i})$, $G^2_{\sigma_i}=\Glue(g_{\sigma_i})$, $F^1_{\sigma_i}=\Flip_{\sigma_i}([\F_1])$, and $F^2_{\sigma_i}=\Flip_{\sigma_i}([\F_2])$ where $[\F_i]$ is the tetrahedron of flags corresponding to the parameters for $\tet_i$, superscripts in \eqref{eq:sis_edge_mats} are taken mod $2$ and $\sigma_i$ is determined by Table~\ref{tab:eftab2}. Again, these matrices can be computed and checked to be the identity, so the edge equations are also satisfied. Thus the above curve is contained in $\mathcal{D}_\rp(N_8,\Delta)$.

\begin{table}[t]
    \centering
    \begin{tabular}{|c|c|c|c|c|c|c|c|c|c|c|c|c|}
    \hline
         i & 1 & 2 & 3 & 4 & 5 & 6 & 7 & 8 & 9 & 10 & 11 & 12 \\
         \hline
         $\sigma_i$ & (41)3 & (24)3 & (23)1 & (43)2 & (34)1 & (12)3 & (31)2 & (31)2 & (24)3 & (23)1 & (12)3 & (14)2\\
         \hline
    \end{tabular}
    \caption{Edge-face table for \eqref{eq:sis_edge_mats}.}
    \label{tab:eftab2}
\end{table}

As with the case for $M_8$, this curve consists of many properly convex structures on $N_8$, allowing us to prove Theorem~\ref{thm:sis_thm}.
\begin{figure}
  \centering
  \def\svgscale{.3}
  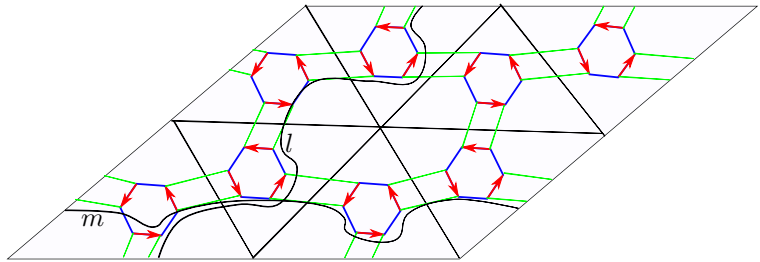
  \caption{The cusp cross section for the figure-eight sister, its intersection with the monodromy graph, and a meridian/longitude pair.}
  \label{fig8_sister_cusp}
\end{figure}

\begin{proof}[Proof of Theorem~\ref{thm:sis_thm}]
The point $t=1$ corresponds to the complete hyperbolic structure on $N_8$, which is well known to be built out of two regular hyperbolic ideal tetrahedra. By analyzing the peripheral subgroup one can check that for each $t$ that this group preserves a complete real flag, and thus for $t$ near 1 these solutions correspond to properly convex projective structures on $N_8$ by \cite[Thm.\ 0.2]{CLTII}. Using Figure \ref{fig8_sister_cusp} the holonomy of the peripheral subgroup can be computed and one finds that for $t\neq 1$ this give the holonomy of a type 1 cusp on the torus. Finally, using Theorem 0.6 of \cite{BCL}, it follows that these properly convex structures have finite volume.
\end{proof}

This case is of additional interest since such structures were not previously known to exist. 


\subsection{Hopf link orbifold}\label{sec:Hopf_link}

We close this section by describing solutions to the gluing equations for an orbifold $O$ whose underlying space is $S^3$ and whose singular locus is the union of the Hopf link and an unknotted arc connecting the two boundary components, where all singular components have order $3$ (cf. Figure \ref{hopf_gluing}). The orbifold $O$ admits an ideal triangulation with a single tetrahedron (cf. Figure \ref{hopf_gluing}). This orbifold has two cusps, each of which has a cross section that is an Euclidean $2$--orbifold with underlying space $S^2$, and singular locus three points each of order $3$. 

\begin{figure}
  \centering
  \def\svgscale{.5}
  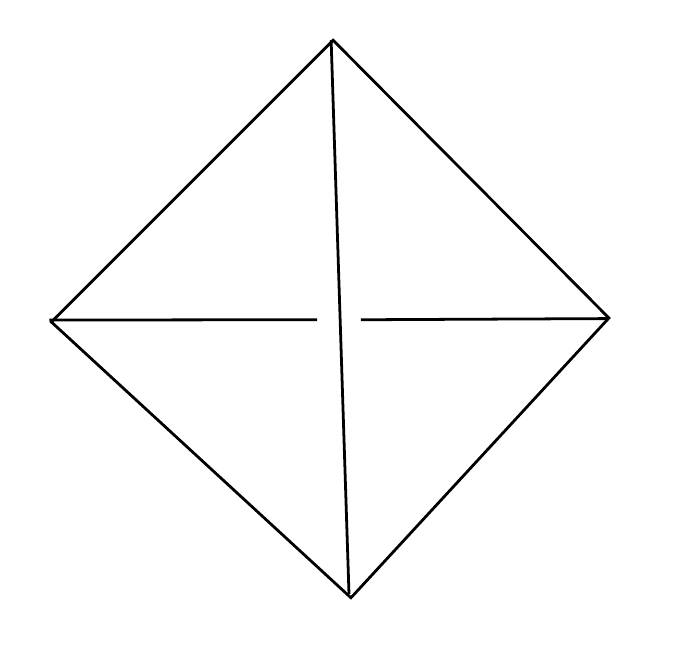
  \def\svgscale{.25}
  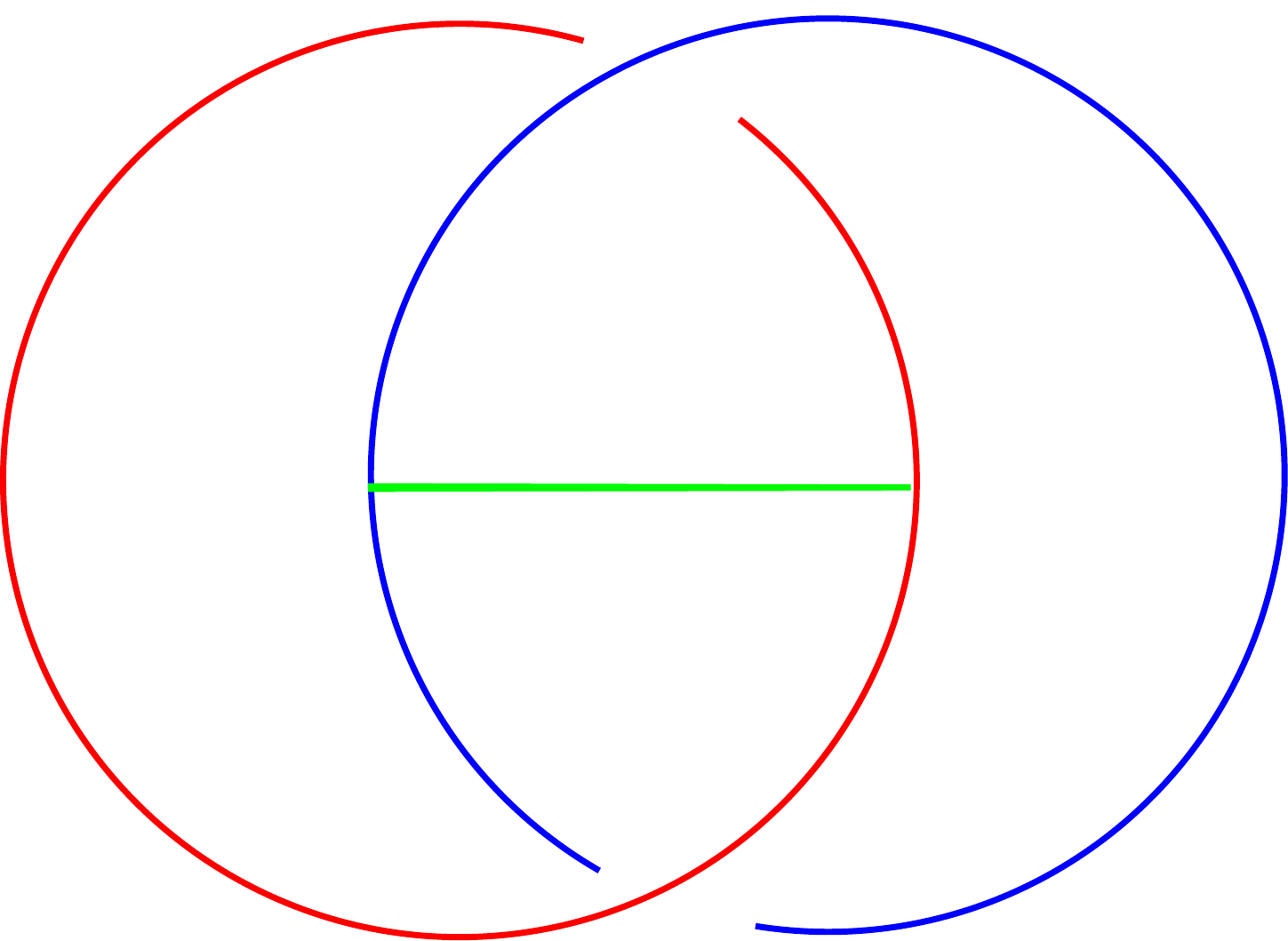
  \caption{(Left) An ideal triangulation for the Hopf link orbifold. (Right) The singular locus of the Hopf orbifold.}\label{hopf_gluing}
\end{figure}

We now describe how to solve (a modified version of) the gluing equations for this orbifold. The three equivalence classes of edges in the triangulation in Figure \ref{hopf_gluing} correspond to the three components of the singular locus. Each of these components has degree 3 and so we must modify the edge equations so that instead of the corresponding product of transformations being equal to the identity, they are instead equal to an elliptic element of order 3. In other words, the holonomy around each edge must preserve a pair of projective lines, one of which is fixed pointwise and one where the action is elliptic of order 3. 

Let $\sigma = (ij)k$. For this example we adopt the notation $e_{ij}:=e_{\sigma}$ for the edge parameters and $g_{(ij)k} = g_{\sigma}$ for the gluing parameters, of the tetrahedra. We begin by noting some simple equations. Since the edges $(12)$ and $(34)$ are only glued to themselves, it follows from \eqref{eq:simple_entry_gluing_equations} that $e_{12}^3=e_{34}^3=1$. Since we are only interested in real solutions, they imply that
\begin{align}\label{eq:hopf_e_parm_1}
    e_{12}=e_{34}=1.
\end{align} 
Furthermore, the previous equation, together that the internal consistency equations, imply that
\begin{align}\label{eq:hopf_e_param_2}
    e_{13}e_{14}e_{23}e_{24}=1.
\end{align}

Furthermore, the edge transformations can be calculated as 
\begin{align}\label{eq:hopf_edge_mats}
    &\Glue(g_{(21)4})\Flip_{(12)3}, \qquad \text{and} \qquad  \Glue(g_{(34)1})\Flip_{(43)2},\\ \nonumber
    \Glue(g_{(42)1})&\Flip_{(24)3}\Glue(g_{(32)4})\Flip_{(23)1}\Glue(g_{(31)2})\Flip_{(13)4}\Glue(g_{(41)3})\Flip_{(14)2},
\end{align}
where $\Flip_{(ij)k}:=\Flip_{(ij)k}([\F])$ with $[\F]$ equal to the tetrahedron of flags determined by the $e_{ij}$.

Since equations \eqref{eq:hopf_e_parm_1} and \eqref{eq:hopf_e_param_2} are satisfied, it follows that each of the transformations in \eqref{eq:hopf_edge_mats} is of the form 
$$\begin{bmatrix}
1 & 0 & 0 & 0\\
0 & 1 & 0 & 0\\
0 & 0 & m_{33} & m_{34}\\
0 & 0 & m_{43} & m_{44}
\end{bmatrix}.
$$
Such a matrix has order 3 if and only if both its trace and determinant are $1$. The determinants of the first two matrices in \eqref{eq:hopf_edge_mats} are easily calculated to be $g_{(21)4}$ and $g_{(34)1}$ (cf. \eqref{eq:simple_det_gluing_equations}), and so both of these gluing parameters must be equal to $1$. Assuming that, then the first two matrices in \eqref{eq:hopf_edge_mats} have trace $1$ if and only if $e_{14}=e_{23}=3$.

Setting the trace and determinant of the last matrix in \eqref{eq:hopf_edge_mats} to 1 and solving the face equations gives a unique solution 
\begin{align}\label{eq:hopf_soln}
   e_{12}=e_{34}=1,\qquad
   e_{13}=e_{24}=1/3,\qquad
   e_{14}=e_{23}=3.
 \end{align}
 with the values of $g_\sigma$ being determined by Table \ref{tab:hopf_gluing_params}. It is easy to check that this solution corresponds to the complete hyperbolic structure on $O$, which is constructed from an ideal tetrahedron whose Thurston's parameter is the third root of unity $z=-\frac{1}{2}+\frac{\sqrt{3}}{2}i$. The cusps of the orbifold $O$ are rigid cusps (i.e. their Teichm\"uller space is trivial). It follows that there are no incomplete hyperbolic structures on $O$, and so the gluing equations detect all hyperbolic structures on $O$. 

\begin{table}[h]
    \centering
    \begin{tabular}{|c|c|c|c|c|c|c|c|c|c|c|c|c|}
\hline
$\sigma$ & (12)3 &(21)4 &(34)1 & (43)2 & (13)4 & (31)2 & (24)3 & (42)1 & (14)2 & (41)3 & (23)1 & (32)4\\
\hline
$g_\sigma$ & 1 & 1 & 1 & 1 & 1/3 & 1/3 &1/3 &1/3 &3 &3 & 3 &3 \\ 
\hline
\end{tabular}
    \caption{Gluing parameters for Hopf link orbifold.}
    \label{tab:hopf_gluing_params}
\end{table}

In light of Theorem~\ref{thm:main_thm2}, it is tempting to conclude that $O$ does not admit convex projective structures other than the complete hyperbolic structure, however, we were recently made aware (personal communication) that J.\ Porti and S.\ Tillmann have computed the moduli space of convex projective structures on $O$ and found that it has positive dimension (notice that this does not contradict Theorem~\ref{thm:main_thm2} since $O$ is an orbifold and not a manifold). The reason that our equations do not detect these projective structures is that the peripheral subgroups do not preserve an incomplete real flag (cf. Lemma~\ref{lem:peripheral_fixes_flag}). More precisely, if $H$ is a peripheral subgroup of $O$ then $H$ is virtually abelian with finite index abelian normal subgroup $H'$. For the non-hyperbolic convex projective structures on $O$, the subgroup $H'$ preserves finitely many incomplete flags (the cusps turn out to be type 3 generalized cusps). The group $H$ acts on the set of incomplete flags preserved by $H'$, however this action does not have a global fixed point. 

\bibliographystyle{plain}
\bibliography{biblio}

\end{document}